\documentclass{amsart}
\usepackage{graphics}

\usepackage{url}
\makeatletter
\def\url@leostyle{%
  \@ifundefined{selectfont}{\def\UrlFont{\sf}}{\def\UrlFont{\small\ttfamily}}}
\makeatother
\newtheorem{theorem}{Theorem}
\newtheorem{proposition}{Proposition}
\newtheorem{lemma}{Lemma}
\newtheorem{cor}{Corollary}
\newtheorem{conjecture}{Conjecture}

\def\metrictwo#1{\langle\!\langle #1 \rangle\!\rangle}
\def\normtwo#1{|| #1 ||}

\def\conj#1{\overline{#1}}

\def\R{\mathbf{R}}
\def\N{\mathbf{N}}

\def\RP{\mathbf{RP}}
\def\C{\mathbf{C}}

\def\CP{\mathbf{CP}}
\def\l{\lambda}

\def\a{\alpha}
\def\b{\beta}
\def\e{\epsilon}
\def\g{\gamma}

\def\d{\delta}

\def\k{\kappa}

\def\t{\theta}

\def\w{\omega}

\def\into{\rightarrow}

\def\intersection{\cap}

\def\union{\cup}

\def\nbhd{{ neighborhood } }

\DeclareMathOperator{\im}{im}

\def\smallskip{\par\vspace{1mm}}
\def\medskip{\par\vspace{2mm}}
\def\bigskip{\par\vspace{3mm}}

\def\fr#1#2{\frac{#1}{#2}}
\def\smfr#1#2{\tfrac{#1}{#2}}

\def\m#1{\begin{bmatrix}#1\end{bmatrix}}

\newcount\equationcount
\def\thenumber{0}
\def\eq#1{\global\advance\equationcount by 1
   \def\thenumber{\number\equationcount}
                        {$$#1\eqno(\thenumber)$$}}

\def\nbhd{{neighborhood}} 
\def\Ri{Riemannian }
\def\R{I\!\! R}

 \newcommand{\dsty}{\displaystyle}

\begin{document}

\title[Brake-Syzygy]{From Brake to Syzygy}

\author{Richard Moeckel}
\author{Richard Montgomery}
\author{Andrea Venturelli}
\address{School of Mathematics\\ University of Minnesota\\ Minneapolis MN 55455}
\address{Dept. of Mathematics\\ University of California, Santa Cruz\\ Santa Cruz CA}
\address{Laboratoire d'Analyse non lin\'eaire et g\'em\'etrie \\ Universit\'e d'Avignon \\ Avignon (FR) }

\email{rick@math.umn.edu}
\email{rmont@count.ucsc.edu}
\email{andrea.venturelli@univ-avignon.fr}

\date{June 1, 2011 (Preliminary Version)}

\keywords{Celestial mechanics, three-body problem, brake orbits, syzygy}

\subjclass[2000]{70F10, 70F15, 37N05, 70G40, 70G60, 70H12}

\thanks{}

\begin{abstract}  
In the planar three-body problem, we study solutions with zero initial velocity (brake orbits).  Following such a solution until the three masses become collinear (syzygy), we obtain a continuous, flow-induced Poincar\'e map.  We study the image of the map in the set of collinear configurations and define a continuous extension to the Lagrange triple collision orbit.  In addition we provide a variational characterization of some of the resulting brake-to-syzygy orbits and find simple examples of periodic brake orbits.
\end{abstract}

\maketitle

\section{Introduction and Main Results.}\label{SecIntro}
This paper concerns the interplay between brake orbits and syzygies in the
 Newtonian three body problem.   
A {\em brake orbit}  is a   solution, not necessarily periodic, 
for which  the velocities of all three  bodies are zero  at some instant, the `brake instant'.  
Brake orbits have zero angular momentum and negative energy.
A  {\em syzygy} occurs when the three bodies
become collinear.  We will count binary collisions as  syzygies, but exclude triple collision. 

%Triple collision acts like an essential singularity for the three body flow.
Lagrange \cite{Lagrange} discovered a  brake orbit which ends in   triple collision.  
The three bodies form an  equilateral triangle at each instant,   the  triangle shrinking homothetically to 
triple collision.   Extended over   its maximum interval of existence,   Lagrange's solution
explodes out of  triple collision, reaches   maximum size at the brake instant, and then shrinks back to triple collision. 

Lagrange's solution is the only negative energy, zero angular momentum solution without  syzygies
\cite{Mont, Mont2}.  In particular,   brake orbits have negative energy, and zero angular momentum,
and so all of them,   {\it except} Lagrange's,  suffer syzygies.    Thus we have a map taking a  
brake initial condition  to its first syzygy.  We call this map   the {\em syzygy map}.  
Upon fixing the energy and reducing by symmetries,  the domain and range of the   map are topologically  punctured 
open discs, the punctures corresponding to the  Lagrange orbit.  See  figure~\ref{fig_Hillspherical} where the map takes the top ``zero-velocity surface", or upper Hill boundary, 
of the solid Hill's region to the plane inside representing the collinear configurations.  The exceptional Lagrange orbit runs from the central point on the top surface to the origin in the plane (which corresponds to triple collision),  connecting the 
puncture  in the domain to the puncture in the range.   
\begin{theorem} 
\label{thm:cty} The syzygy map is continuous.  
 Its image  contains a neighborhood of the binary collision locus.  
  For a large open set of mass parameters, including equal masses,  
the map extends continuously to the puncture, taking the equilateral triangle  of Lagrange 
  to   triple collision. 
 %When the three masses are equal the image contains the entire collision locus: the three binary  collision rays in addition to triple collision. 
\end{theorem}

{\bf Remark on the Range.}  Numerical evidence suggests that the syzygy map is not onto (see figure~\ref{fig_syzygyimage}).  The closure of its range lies strictly inside the collinear Hill's region.  A heuristic explanation for this is as follows.  The boundary of the domain of the syzygy map is the collinear zero velocity curve, i.e., the collinear Hill boundary.  Orbits starting on this curve remain collinear for all time and so are in a permanent state of syzygy.   Nearby, non-collinear orbits oscillate around the collinear invariant manifold and take a certain time to reach syzygy, which need not approach zero as the initial point approaches the boundary of the domain.  Thus the nearby orbits have time to move away from the boundary before reaching syzygy.  It may be that the syzygy map extends 
continuously to the boundary but we do not pursue this question here.

Collision-free syzygies come in three types, 1, 2, and 3
depending on which mass lies between the other two at syzygy.  Listing the  syzygy types   in temporal order
yields the  syzygy sequence of a solution. (The syzygy sequence of a  periodic  collision-free solution encodes
its free homotopy type, or  braid type.)    
In (\cite{Mont3}, \cite{Tanikawa})  the notion of syzygy sequence was used   as a topological sorting  tool for  the three-body problem.
(See also  \cite{Moore}  and \cite{Mikkola}.)   
  A ``stuttering orbit'' is a solution  whose syzygy sequence has a stutter, meaning that   
  the same
symbol  occurs twice in a row, as   in ``11'' ``22'' or `33''.   For topological and variational reasons,
one of us had believed that stuttering sequences were rare.  The theorem easily proves the contrary to be true.

\begin{cor} Within the   negative energy,  zero angular momentum phase space for the  
three body problem there is an open  and unbounded set corresponding to  
stuttering orbits. 
\end{cor}

\begin{proof}[Proof of the corollary.]
%{\bf Proof of the Corollary.}  
 If  a collinear configuration $q$ is in the image of the syzygy map, and if $v$ is the velocity of the
  brake orbit segment at $q$, then by running this   orbit backwards,
  which is to say, considering  the solution with initial condition $(q, -v)$,
we obtain a   brake orbit whose next syzygy is $q$, with velocity $+v$.  This  brake orbit is a  stuttering orbit
as long as $q$ is not a collision point.    Perturbing initial conditions slightly cannot destroy    
stutters, due to transversality of the
orbit with the syzygy plane.  
\end{proof}

%\vskip .3 cm'
{\bf Periodic Brake Orbits.} In 1893  a mathematician named Meissel 
conjectured  that if    masses
in the ratio 3, 4, 5 are place at the vertices of a 3-4-5 triangle and let   go from rest
then the corresponding brake orbit is periodic.  Burrau \cite{Burrau}   reported the conversation  with
Meissel and performed a   pen-and-paper numerical tour-de-force which suggested the conjecture may be false. This
``Pythagorean three-body problem''  became a test case for numerical integration methods. 
Szehebely   \cite{Sz1} carried  the  integration    further and found the motion ends  (and begins) in an elliptic-hyperbolic escape
Peters  and  Szehebeley \cite{Sz2}  perturbed away
 from the Pythagorean initial conditions and with the help of Newton iteration found  a 
 periodic brake orbit.
 % having  a   collision  half way between the brake endpoints.  numerically.
  
Modern   investigations into periodic brake orbits in general Hamiltonian systems began with Seifert's \cite{Seifert} 1948  topological existence 
proof for 
the existence of such orbits 
for harmonic-oscillator type potentials.  (Otto Raul Ruiz coined the term ``brake orbits'' in 
 \cite{Ruiz}.)  
We will establish    existence of  periodic brake orbits in the three-body problem   
by  looking for  brake orbits which  hit the syzygy plane $C$ orthogonally.
Reflecting such an orbit yields a periodic brake orbit.   Assume the masses are $m_1=m_2=1, m_3>0$.  
Then there is an invariant isosceles subsystem of the three-body problem and we will prove:

\begin{theorem}\label{theorem_isosbrake}
For $m_3$ in an open set of mass parameters, including $m_3=1$, there is a periodic isosceles brake orbit which hits 
the syzygy plane $C$ orthogonally upon its 2nd hit (see figure~\ref{fig_isosHill}).
\end{theorem}

Do their exist brake orbits, besides Lagrange's  whose 1st intersection with $C$ is orthogonal?
We conjecture not.  Let $I(t)$ be the total moment of inertia of the three bodies at time $t$.  
The metric on shape space  is such that away from triple collision, a curve orthogonal to $C$ must
have $\dot I  = 0$ at   intersection.    This non-existence conjecture 
would then follow  from  the validity of 
\begin{conjecture}\label{conj1}
$\dot I  (t)< 0$ holds along any brake orbit segment, from the brake time up to and including the time of   
1st syzygy.
\end{conjecture}

Our evidence for conjecture \ref{conj1} is primarily numerical.   
If this conjecture is true then we can eliminate the restriction on the masses in theorem~\ref{thm:cty}.
See the remark following proposition \ref{prop_collisionsyzygy}.

%\vskip .3cm

{\bf Variational Methods.} We are   interested in the interplay between variational methods, brake orbits, and syzygies.
If the energy is fixed to be $-h$ then the natural (and oldest) variational principle to use 
is that often called the  Jacobi-Maupertuis action principle, described below in section
\ref{SecVariational}.   (See also  \cite{Birkhoff}, p. 37 eq. (2).)  The associated action functional
will be denoted $A_{JM}$ (see eq. (\ref{Jac-Maup})).  
Non-collision critical points $\gamma$ for $A_{JM}$ which lie in the Hill region are solutions to Newton's equations
with energy $-h$.  Curves inside the Hill region which minimize $A_{JM}$  
among all compenting curves in the Hill region connecting two fixed points, or two fixed subsets
will be called {\it JM minimizers}.

We gain understanding of
the  syzygy map by considering JM-minimizers connecting a   fixed syzygy   configuration $q$  to the Hill boundary.
\begin{theorem}
\label{thm:variational} (i).JM minimizers exist  from any  chosen  point $q_0$  in the interior of the Hill
boundary to the  Hill boundary.   These minimizers
are solutions.  When not  collinear,  a minimizer has   at most one syzygy:  $q_0$.   

(ii).There exists a neighborhood $\mathcal U$ of the binary  collision locus,
such that if  $q_0 \in \mathcal U$, then  the  minimizers  are not collinear. 

(iii).If $q_0$ is triple collision then the minimizer is unique up to reflection and is   
one half of the Lagrange homothetic brake solution. 
\end{theorem}

{\bf Proof of the part of Theorem \ref{thm:cty} regarding the image. }    
Let $\mathcal U$ be the neighborhood of collision locus given by Theorem 3. If $q_0\in \mathcal U$, 
the minimizers of Theorem 3 realize non-collinear brake orbits whose first syzygy is $q_0$, 
therefore the image of the syzygy map contains $\mathcal U$. \qed
\vspace{2mm} \\ \noindent
An important step in the proof of theorem \ref{thm:variational} is of independent interest.

\begin{lemma} \label{lemma:JM_Marchal} [Jacobi-Maupertuis Marchal's lemma]  Given two points $q_0$ and $q_1$ 
in the Hill region, a JM minimizer  exists  for the fixed endpoint problem  of minimizing $A_{JM} (\gamma)$ among all paths
$\gamma$ lying in the Hill region and connecting $q_0$ to $q_1$. 
Any such minimizer is collision-free except possibly at its endpoint.
If a minimizer does not touch the Hill boundary (except possibly at one endpoint) then after  reparametrization 
it is   a solution with energy $-h$.  
\end{lemma}

\vskip .3cm 
 We can be more precise about   minimizers to binary collision  when two or all masses are equal.
 Let $r_{ij}$ denote the distance beween mass $i$ and mass $j$. 
 \begin{theorem} [Case of equal masses.] \label{thm:equal_masses}
 (a) If $m_1 = m_2$  and if the
 starting point $q_0$ is a collision point  with $r_{12} = 0$ then the minimizers of  
Theorem \ref{thm:variational}  
are isosceles brake orbits: 
 $r_{13} = r_{23}$ throughout the orbit.  
 
 (b) If $m_1 = m_2$   and if the starting collinear point $q_0$ is such that $r_{13}<r_{23}$ 
(resp. $r_{13}>r_{23}$), 
then a minimizer $\gamma$ of Theorem \ref{thm:variational} satisfies this same   inequality:  
at every point $\gamma(t)$ we have $r_{13}(t)<r_{23}(t)$ (resp. $r_{13}(t)>r_{23}(t)$).  

(c) If all  three masses are equal, and if $q_0$ is a collinear point, a minimizer $\gamma$  of Theorem 
\ref{thm:variational} satisfies the  same side length inequalities  as $q_0$ : 
if $r_{12} < r_{13} < r_{23}$
for $q_0$, then at every point  $\gamma (t)$ of $\gamma$  we have $r_{12} (t)< r_{13} (t)< r_{23}(t)$. 
\end{theorem}
  
  Part (c) of this theorem suggest:  
\begin{conjecture}  If three   equal masses are let go at rest, in the shape 
of a scalene triangle with side lengths $r_{12} < r_{13} < r_{23}$ and    attract each other  according to Newton's law
then these side
length inequalities $r_{12} (t) < r_{13} (t) < r_{23}(t)$ persist up to the  instant $t$ of first  syzygy,
\end{conjecture}

%The conjecture should also hold for any power law potential  $1/r^a$,  $a > 0$,     $a =1$ being Newton's. 

{\bf Commentary.}   Our original goal in using variational methods was to construct the inverse of the syzygy map using
JM minimizers.   This approach was thwarted due to our inability
to exclude or deal with {\it caustics}:  brake orbits which cross each other
in configuration space before syzygy. Points on the  boundary of the image of the syzygy map appear to be 
 conjugate points --  points where non-collinear brake orbits ``focus''  onto a point of a collinear brake orbit.

{\bf Outline and notation.} 
  In the next section we derive the equations of motion in terms  suitable for our purposes.
 In section~\ref{SecExistSyzygy} we use these equations to rederive the theorem of \cite{Mont}, \cite{Mont2} regarding
 infinitely many syzygies.  We also set up the syzygy map.  In section~\ref{SecContinuity} we prove theorem 1
 regarding continuity of the syzygy map.    In section ~\ref{SecVariational} we investigate variational properties of the Jacobi-Maupertuis metric and  prove 
 theorems \ref{thm:variational} and \ref{thm:equal_masses} and the lemmas around them.
   In section~\ref{SecPeriodicBrake} we establish  Theorem 2 concerning a 
periodic isosceles  brake orbit.
 
\section{Equation of Motion and Reduction}\label{SecEq}
 Consider the planar three-body problem with masses $m_i >0 , i=1,2,3$.   
 Let the positions be  $q_i\in\R^2 \cong \C$ and the velocities be $v_i = \dot q_i \in\R^2$.  
Newton's laws of motion are the Euler-Lagrange equation of the Lagrangian
 \begin{equation}\label{eq_Lag}
 L =  K + U
 \end{equation}
 where 
 \begin{equation}\label{eq_KUCart}
\begin{aligned}
2K &= m_1|v_1|^2+m_2|v_2|^2+m_3|v_3|^2\\
U &= \fr{m_1m_2}{r_{12}}+\fr{m_1m_3}{r_{13}}+\fr{m_2m_3}{r_{23}}.
\end{aligned}
\end{equation}
Here $r_{ij} = |q_i-q_j|$ denotes the distance between the $i$-th and $j$-th masses.  
The total energy of the system is constant:
$$K - U = -h\qquad h>0.$$
 
 Assume without loss of generality that total momentum is zero and that the center of mass is at the origin, i.e.,
 $$m_1v_1+m_2v_2+m_3v_3 =m_1q_1+m_2q_2+m_2q_3=0.$$
 Introduce Jacobi variables 
 \begin{equation}\label{eq_Jac}
%\begin{aligned}
 \xi_1 =  q_2-q_1 \qquad \xi_2 =  q_3 - \fr{m_1 q_1+ m_2 q_2}{m_1+m_2}
 \end{equation}
 and their velocities $\dot \xi_i$.
 Then the equations of motion are given by a Lagrangian of the same form (\ref{eq_Lag}) where now
\begin{equation}\label{eq_KUJac}
\begin{aligned}
K &= \mu_1 |\dot \xi_1|^2 + \mu_2 |\dot \xi_2|^2\\
U &= \fr{m_1 m_2}{r_{12}}+\fr{m_1 m_3}{r_{13}}+\fr{m_2 m_3}{r_{23}}.
\end{aligned}
\end{equation}
The mass parameters are:
\begin{equation}\label{eq_muij}
\mu_1 = \fr{m_1 m_2}{m_1+m_2}\qquad \mu_2 =  \fr{(m_1+ m_2)m_3}{m_1+m_2+m_3} = \fr{(m_1+ m_2)m_3}{m}
\end{equation}
where 
$$m=m_1+m_2+m_3$$
 is the total mass.
The mutual distances are given by
\begin{equation}\label{eq_rijJac}
\begin{aligned}
r_{12} &= |\xi_1|\\
r_{13} &= |\xi_2+\nu_2 \xi_1|\\
r_{23} &= |\xi_2-\nu_1 \xi_1|
\end{aligned}
\end{equation}
where 
$$\nu_1= \fr{m_1}{m_1+m_2}\qquad \nu_2= \fr{m_2}{m_1+m_2}.$$

\subsection{Reduction}\label{sec_reduction}

 Jacobi coordinates (\ref{eq_Jac}) eliminate the translational symmetry, reducing  the number of degrees of freedom 
from 6 to 4.   The next step is the elimination of the rotational symmetry to reduce from 4 to 3 degrees of freedom. 
This reduction  is accomplished by fixing the angular momentum and working in the quotient space
 by rotations.  When the angular momentum is zero, there is a particulary elegant way to accomplish this reduction.   

Regard the Jacobi variables $\xi$ as complex numbers: $\xi = (\xi_1,\xi_2)\in\C^2$.
Introduce a Hermitian metric on $\C^2$:
\begin{equation}\label{eq_Hermitian}
\metrictwo{u,v} = \mu_1v_1\conj{w_1} + \mu_2v_2\conj{w_2}
\end{equation}
If  $\normtwo{v}^2 = \metrictwo{v,v}$ denotes the corresponding norm then the kinetic energy is given by 
$$\normtwo{\dot \xi}^2 = \mu_1 |\dot \xi_1|^2 + \mu_2 |\dot \xi_2|^2 = K$$
while 
\begin{equation} \label{massa-scalar-prod}
||\xi||^2 = \mu_1 |\xi_1|^2 + \mu_2 |\xi_2|^2 = I
\end{equation}
is the moment of inertia.   We will also use the alternative formula of Lagrange:
\begin{equation}\label{eq_xisquared}
||\xi||^2 = \fr1m\left(m_1m_2 r_{12}^2 + m_1m_3 r_{13}^2 + m_2 m_3 r_{23}^2\right)
\end{equation}
where  the distances $r_{ij}$ are given by(\ref{eq_rijJac}).  
The real part of this Hermitian metric is a Riemannian metric on $\C^2$. The imaginary part of the Hermitian metric is 
 a nondegenerate two-form on $\C^2$ with respect to which the angular momentum constant  $\omega$ of the three-body 
problem takes the form
$$\omega = \im \metrictwo{\xi,\dot \xi}.$$

The rotation group $S^1 = SO(2)$ acts on   $\C^2$  according to $(\xi_1,\xi_2)\rightarrow e^{i\t}(\xi_1,\xi_2)$.
To eliminate this symmetry introduce a new variable
$$r = ||\xi|| = \sqrt{I}$$
to measure the overall size of the configuration and let $[\xi] = [\xi_1,\xi_2]\in\CP^1$ be the point in projective space with 
homogeneous
coordinates $\xi$.  Explicity, $[\xi]$ is an equivalence class of pairs of point of $\C^2\setminus 0$ where $\xi\equiv \xi'$ 
if and only if $\xi' = k \xi$ 
for some nonzero complex constant $k$.   Thus $[\xi]$ describes the shape of the configuration up to rotation and scaling.   
The variables $(r, [\xi])$ together coordinatize the quotient space $\C^2/ S^1$.  

Recall that the one-dimensional complex projective space $\CP^1$ is essentially    the usual Riemann sphere 
$\C\union\infty$.   The formula $\alpha([\xi_1,\xi_2]) = \xi_2/ \xi_1$ gives a   map $\alpha:\CP^1\into \C\union\infty$,
the standard ``affine chart''.  Alternatively, one has the  diffeomorphism $St: \CP^1 \to S^2$ (defined in 
(\ref{hopf})) to the standard unit sphere $S^2\subset\R^3$ by composing 
$\alpha$ with the inverse of a stereographic projection map $\sigma:S^2\into \C\union\infty$.  
The space $S = \CP^1$ in  any of these three forms will be called the {\em shape sphere}.   

In the papers \cite{Mont} and \cite{ChencinerMontgomery} the sphere version of shape space  was used, and the variables 
$r, [\xi]$ were combined at times to give an isomorphism $\C^2/ S^1 \to \R^3$ sending $(r, [\xi]) \mapsto r St([\xi])$.
The projective version of the shape sphere, although less familiar, makes some of the computations below much simpler.
Triple collision corresponds to $0 \in \C^2$ and the quotient map $\C^2 \setminus 0 \to (\C^2 \setminus 0)/S^1$
is realized by the map 
\begin{equation}
\label{eq:quotientmap}
 \pi:\C^2\setminus 0\into Q = (0, \infty) \times \CP^1 ; \\ \\
\pi(\xi) = (||\xi||, [\xi]).
\end{equation}
%If $\xi(t)$ is a solution to the three body problem, its projection $\pi(\xi(t))$ will be a solution of a certain reduced   Lagrangian system on $TQ$.  

To write down the quotient  dynamics  we need a description of the kinetic energy in  quotient variables,
and so we need a way of describing tangent vectors to  $\CP^1$.  
 Define the equivalence relation $\equiv$ by $(\xi,u)\equiv (\xi',u')$ if and only if there are complex numbers $k,l$ with $k\ne 0$ such that $(\xi',u') = (k\xi,ku+l\xi)$.  It is easy to see that two pairs are equivalent if and only if $T\pi(\xi,u) = T\pi(\xi',u')$ where $T\pi:T(\C^2\setminus 0)\into T\CP^1$ is the derivative of the quotient map.  Thus an equivalence class $[\xi,u]$ of such pairs represents an element of $T_{[\xi]}\CP^1$, i.e., a {\em shape velocity} at the shape $[\xi]$. 
 One verifies that the  expression 
\begin{equation}\label{eq_FSnorm}
||[\xi,\dot \xi]||^2 =  \fr{\mu_1\mu_2}{||\xi||^4}|\xi_1\dot \xi_2-\xi_2\dot \xi_1|^2.
\end{equation}
defines a quadratic form on tangent vectors at $[\xi]$ and as such is a metric.
It is the Fubini-Study metric, which corresponds under the diffeomorphism $St$ to the standard `round' metric on the sphere,
scaled so that the radius of the sphere is $1/2$.  We emphasize that in  
this expression, and in the subsequent ones   involving the variables $r,[\xi]$,  the variable $\xi$ is to be viewed as a homogeneous coordinate on $\CP^1$ 
so that the 
$||\xi|| = (\mu_1|\xi_1|^2+\mu_2|\xi_2|^2)^{1/2}$
occurring in the denominator is not linked to $r$, which is taken as an independent variable.  Indeed (\ref{eq_FSnorm}) is invariant under rotation and scaling of $\xi, \dot \xi$.

We have the following  nice formula for the kinetic energy:
\begin{proposition} \label{kinetic}
$\displaystyle K = \smfr12||\dot \xi||^2 = \smfr12\left(\dot r^2 + \frac{\omega^2}{r^2} + r^2 ||[\xi,\dot \xi]||^2\right).$
\end{proposition}
We leave the proof up to the reader, or refer to \cite{ChencinerMontgomery} for an equivalent version.
Taking $\omega = 0$ gives the simple formula
 \begin{equation}\label{eq_Kred}
 K_0 = \smfr12\dot r^2 +  \smfr12  r^2  \fr{\mu_1\mu_2}{||\xi||^4}|\xi_1\dot \xi_2-\xi_2\dot \xi_1|^2
  \end{equation} 

By homogeneity, the negative  potential energy $U$  can also be expressed in terms of $r, [\xi]$.  Set
\begin{equation}\label{eq_V}
V([\xi]) =  ||\xi||\, U(\xi).
\end{equation} 
thus defining $V$.  Equivalently,  $V([\xi]) = U(\xi/\|\xi\|)$.
Since the right-hand  side is homogeneous of degree 0 with respect to $\xi$, and since $U$ is
invariant under rotations, the value of $V([\xi])$ is independent of the choice of representative for $[\xi]$.   Clearly we have $U(\xi) = \smfr1r V([\xi])$ for 
$\xi\in\C^2\setminus 0$ and $(r,[\xi])= \pi(\xi)$.  The function $V:S\into\R$ will be called the {\em shape potential}.

The function $L_{red}:TQ\into\R$ given by 
\begin{equation} \label{lagr_red}
L_{red}(r,\dot r,[\xi,\dot\xi]) =  K_0 + \fr{1}{r}V([\xi])
\end{equation}
will be called the {\em reduced Lagrangian}.  The theory of Lagrangian reduction \cite{Mars,Mont} then gives
\begin{proposition} \label{reduced-solutions}
Let $\xi(t)$ be a zero angular momentum solution of the three-body problem in Jacobi coordinates.  Then
$(r(t),[\xi(t)]) = \pi(\xi(t)) \in Q = (0, \infty) \times \CP^1$ is a solution of the Euler-Lagrange equations for the reduced Lagrangian $L_{red}$ on $TQ$.
\end{proposition}

\subsection{The Shape Sphere and the Shape Potential}
Let $C$ be the set of collinear configurations. In terms of Jacobi variables 
$\xi=(\xi_1,\xi_2)\in C$ if and only if the ratio of $\xi_1,\xi_2$ is real.   The corresponding projective point then satisfies $[\xi]\in\RP^1\subset\CP^1$.  
 Recall that $\RP^1$ can also be viewed as the extended real line $\R\union\infty$ or as the  circle $S^1$.  Thus one can say that the normalized collinear shapes form a circle in the shape sphere, $S$.  Taking the size into account one has
$C = \R^+\times S^1\subset \R^+\times S$.  Here and throughout, by abuse of notation, 
 we will write $C$ as the set of collinear states, either viewed before or after reduction
by the circle action (so that $C \subset \bar Q$), or by reduction by  the circle action and scaling (so $C \subset \CP^2$).  

The binary collision configurations and the Lagrangian equilateral triangles
 play an important role in this paper.  Viewed in $\CP^1$, these form five distinguished points, points whose    homogeneous coordinates
 are easily found.  Setting the mutual distances (\ref{eq_rijJac}) equal to zero one finds collision shapes:
$$b_{12}=[0,1]\qquad b_{13} = [1,-\nu_2]\qquad b_{23} = [1,\nu_1]$$
where, as usual, the notation $[\xi_1,\xi_2]$ means that $(\xi_1,\xi_2)$ is a representative of the projective point.  Switching to the Riemann sphere model by setting $z= \xi_2/\xi_1$ gives
$$b_{12}=\infty \qquad b_{13} =-\nu_2 \qquad b_{23} = \nu_1.$$
The equilateral triangles are found to be at $[1,l_\pm]\in\CP^1$ or at $l_\pm\in\C$ where
\begin{equation}\label{eq_lpm}
l_\pm = \fr{m_1-m_2}{2(m_1+m_2)} \pm \fr{\sqrt{3}}{2}\,i = \fr{\nu_1-\nu_2}{2}\pm \fr{\sqrt{3}}{2}\,i .
\end{equation}

We will  choose coordinates on the shape sphere such that all of these special shapes 
have simple coordinate representations \cite{Mont}.  In these coordinates, the shape potential will also have a relatively simple form. 
To carry out this coordinate change,
we use the well-known fact from complex analysis that there is a unique conformal isomorphism (i.e., a fractional linear map) of the Riemann sphere taking any triple of points to any other triple.   Thus one can move the binary collisions to any convenient locations.  We move them to  the third roots unity on the unit circle. 
Working projectively in homogeneous coordinates, a fractional linear map 
$$z = \frac{c w+d}{aw+b}$$
becomes a linear map
$$\m{\xi_1\\ \xi_2} = \m{a&b\\c&d}\m{\eta_1\\ \eta_2}$$
where $[\xi_1,\xi_2] = [1,z]$ and $[\eta_1,\eta_2] = [1,w]$.

\begin{proposition}
Let $\lambda = e^\fr{2\pi i}{3}$ and $\phi$ be the unique conformal map taking $1, \lambda,\conj\lambda$ to $b_{12}, b_{13},b_{23}$ respectively.   Then $\phi$  
 maps the unit circle  to the collinear shapes and $0$ and $\infty$ to the equilateral shapes $l_+$ and $l_-$.  
Moreover, in homogeneous coordinates
\begin{equation}\label{eq_phimatrix}
\phi([\eta]) = \m{1&-1\\l_+&-l_-}\m{\eta_1\\ \eta_2}
\end{equation}
\end{proposition}

The proof is routine, aided by the fact that $\phi$ preserves  cross ratios.

We will be using  $\eta = (\eta_1,\eta_2)$ as homogeneous coordinates on $\CP^1$ and setting 
$$w = \eta_2/\eta_1\in C.$$
In $w$ coordinates, we have seen that   the collinear shapes form  the unit circle, 
with the binary collisions at the third roots of unity and the Lagrange shapes at $w=0,\infty$.   
We also need   the  shape potential in $w$-variables.  It can be calculated from the formulas in the previous subsection by simply setting 
\begin{equation}\label{eq_etas}\xi_1 = \eta_1-\eta_2\qquad \xi_2 =  l_+ \eta_1-l_- \eta_2 \qquad \eta_1= 1\qquad \eta_2= x+iy.
\end{equation}
First,  (\ref{eq_rijJac}) gives the remarkably simple expressions
\begin{equation}\label{eq_rijxy}
\begin{aligned}
r_{12}^2 &= (x-1)^2+y^2\\
r_{13}^2 &=(x+\smfr12)^2+(y-\smfr{\sqrt{3}}{2})^2\\
r_{23}^2 &=(x+\smfr12)^2+(y+\smfr{\sqrt{3}}{2})^2.
\end{aligned}
\end{equation}
Using these, one can express the norm of the homogeneous coordinates and the shape potential as functions of $(x,y)$.  $||\xi||$  is given by (\ref{eq_xisquared})
and 
\begin{equation}\label{eq_Vxy}
V(x,y) =  ||\xi||\left(\fr{m_1m_2}{r_{12}}+\fr{m_1m_3}{r_{13}}+\fr{m_2m_3}{r_{23}}\right)
= \fr{m_1 m_2}{\rho_{12}}+\fr{m_1 m_3}{\rho_{13}}+\fr{m_2 m_3}{\rho_{23}}
\end{equation} 
with $\rho_{ij} =  r_{ij}/||\xi||$.

{\bf Remark.}
  It is worth saying a bit about  the meaning of   the expressions eq. (\ref{eq_rijxy})  and the variables $\rho_{ij}$
  occuring in eq (\ref{eq_Vxy}).   A  function on $\C^2$ which is  homogeneous of
degree 0  and rotationally invariant   defines a function on $\CP^1$.  But the $r_{ij}$ are   homogeneous of degree $1$,
so do not define functions on $\CP^1$ in this simple manner So what is   eq. (\ref{eq_rijxy}) saying? 
Introduce the   local section $\sigma: \CP^1 \setminus \{ \infty \} \to \C^2$, 
given by  $[1, w ] \to (1, w)$
and the linear map $\tilde \Phi: \C^2 \to \C^2$ which induces   $\phi$.   Apply $\tilde \Phi \circ \sigma$ to
the point $[1, w]$ to
  form  
   $\tilde \Phi (\sigma ([1, \eta]) = (\xi_1, \xi_2) \in \C^2$ and then    apply the distance functions $r_{ij}$ to this configuration
   in $\C^2$  to get the $r_{ij}$ of  
  eq. (\ref{eq_rijxy}).  That is,  the  functions  of  eq. (\ref{eq_rijxy})  are $r_{ij} \circ \tilde \Phi \circ \sigma$. 
  Then  $|| \xi ||$ in
the expression for $V$ is   the moment of inertia $I = r^2$ as given by 
(\ref{eq_xisquared})  with the $r_{ij}$ there being those given by  eq. (\ref{eq_rijxy}). 

Alternatively,  we can view $\CP^1$ as $S^3/S^1$ and realize the $S^3$ by setting
$\| \xi \| = 1$.  Then the $\rho_{ij}$ are $r_{ij}$ restricted to this $S^3$, and then understood
as $S^1$ invariant functions.

\vskip .3cm

Figure~\ref{fig_spherecontour} shows a spherical contour plot of $V$ for equal masses $m_1=m_2=m_3$.    The equator features the three binary collision singularities as well as three saddle points corresponding to the three collinear or Eulerian central configurations.  The equilateral points at the north and south poles of the sphere  are the Lagrangian central configurations which are minima of $V$.

\begin{figure}
\scalebox{.6}{\includegraphics{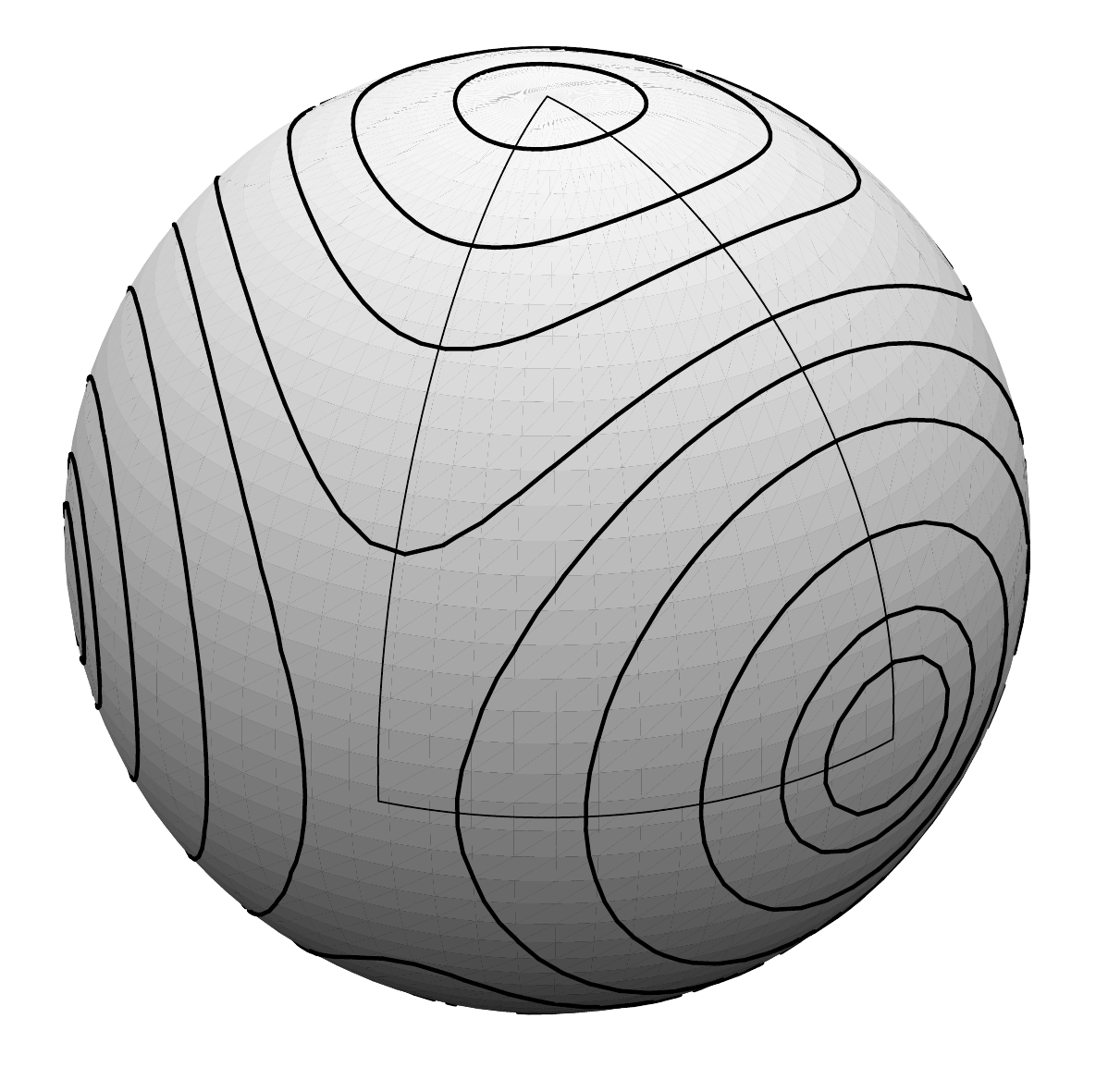}}
\caption{Contour plot of the shape potential on the unit sphere $s_1^2+s_2^2+s_3^2=1$  in the equal mass case.  There is a discrete symmetry of order twelve generated by the reflections in the sides of the indicated spherical triangle.} \label{fig_spherecontour}
\end{figure}

Figure~\ref{fig_Vcontourplots} shows contour plots of the shape potential in stereographic coordinates $(x,y)$ for the equal mass case and for $m_1=1,m_2=2,m_3=10$.   The unit disk in stereographic coordinates corresponds to the upper hemisphere in the sphere model.  When the masses are not equal, the potential is not as symmetric, but due to the choice of coordinates, the binary collisions are still at the roots of unity and the Lagrangian central configuration (which is still the minimum of $V$)  is at the origin.

The variables in figures~\ref{fig_spherecontour} and \ref{fig_Vcontourplots} are related by stereographic projection:
$$s_1= \fr{2x}{1+x^2+y^2} \qquad s_2= \fr{2y}{1+x^2+y^2} \qquad s_3= \fr{1-x^2-y^2}{1+x^2+y^2}\qquad.$$

\begin{figure}
\scalebox{.75}{\includegraphics{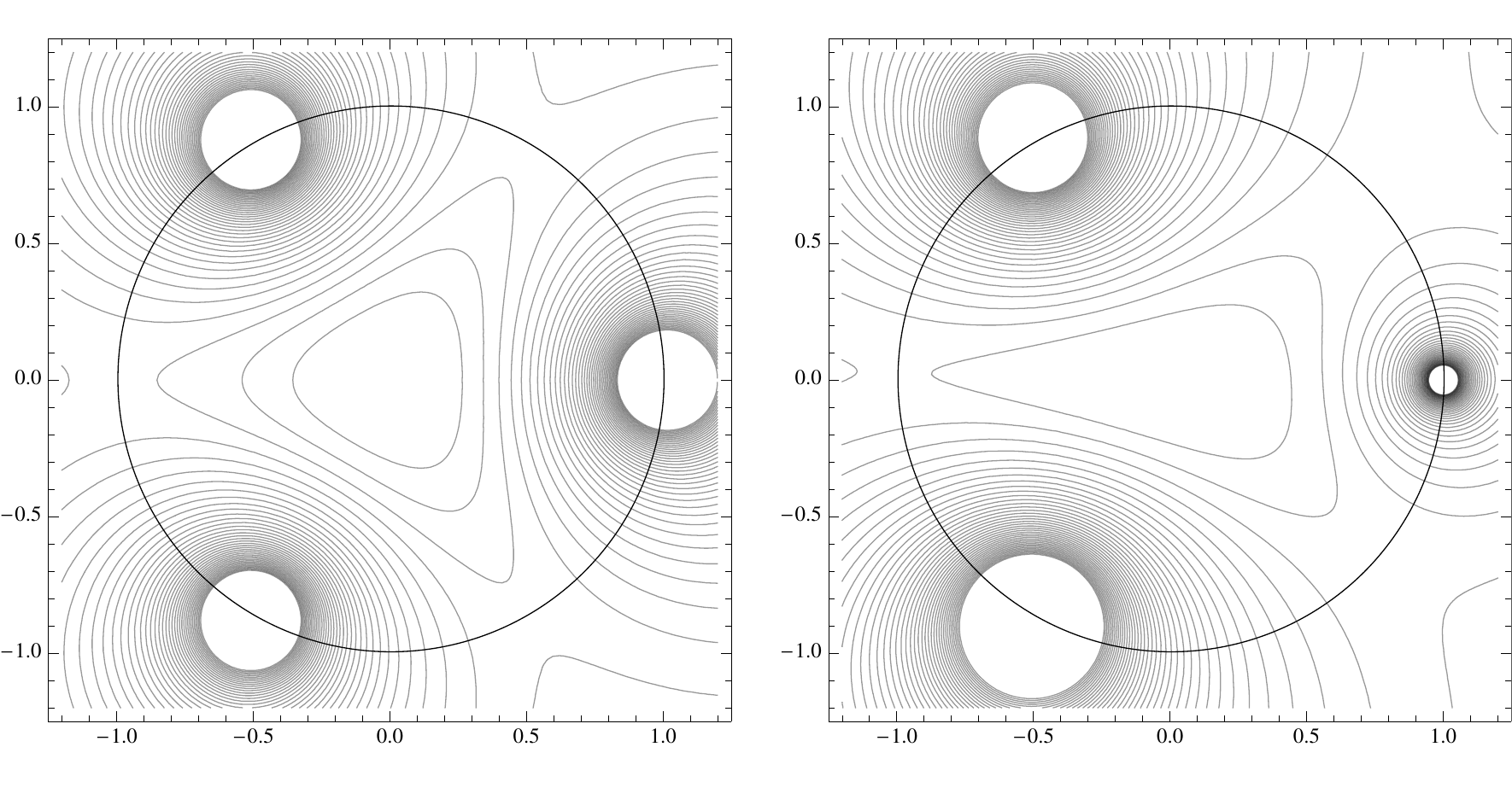}}
\caption{Contour plot of the shape potential in stereographic coordinates $(x,y)$.  The unit disk corresponds to the upper hemisphere in the sphere model.  On the left is the equal mass case as in figure~\ref{fig_spherecontour}. On the right, the masses are $m_1=1,m_2=2,m_3=10$. }\label{fig_Vcontourplots}
\end{figure}

The following result about the behavior of the shape potential will be useful \cite{Mont}.  Consider the potential in the upper hemisphere (the unit disk in stereographic coordinates).  $V(x,y)$ achieves its minimum at the origin.  It turns out that $V$ is strictly increasing along radial line segments from the origin to the equator.
%(Compare lemma 4, section 6 of \cite{Mont}.)

\begin{proposition}\label{prop_Vderiv} (Compare with lemma 4, section 6 of \cite{Mont}.)
For all positive masses,  the shape potential $V(x,y)$ satisfies
$$x V_x + y V_y = \phi(x,y)(1-x^2-y^2)$$
where $\phi(x,y) \ge 0$ with strict inequality if $(x,y)\ne (0,0)$.
\end{proposition}

\begin{proof}
Write 
$$V = \fr{m_1 m_2}{\rho_{12}}+\fr{m_1 m_3}{\rho_{13}}+\fr{m_2 m_3}{\rho_{23}}$$
with $\rho_{ij} =  r_{ij}/||\xi||$, and $r_{ij}, ||\xi||$ expressed as functions of $(x,y)$ using (\ref{eq_rijxy}) and (\ref{eq_xisquared}).  Then a computation shows that
$$x V_x + y V_y = \phi(x,y)(1-x^2-y^2)$$
where
\begin{equation}\label{eq_phi}
\phi =  \fr{\,m_1m_2m_3}{2(m_1+m_2+m_3)||\xi||}(m_1 g_1+m_2 g_2+m_3 g_3)
\end{equation}
and
$$g_1=(r_{13}^2-r_{12}^2)(\fr{1}{r_{12}^3}-\fr{1}{r_{13}^3})\quad g_2=(r_{23}^2-r_{12}^2)(\fr{1}{r_{12}^3}-\fr{1}{r_{23}^3}) \quad g_3=(r_{23}^2-r_{13}^2)(\fr{1}{r_{13}^3}-\fr{1}{r_{23}^3}).$$
Note that $g_1\ge 0$ with strict inequality except on the line where $r_{12}=r_{13}$.  Similar properties hold for $g_2$ and $g_3$ and the proposition follows.
\end{proof}

\subsection{Equations of Motion and Hill's Region}\label{SecEqsOfMotion}
 We can derive the equations of motion  on $Q = (0, \infty) \times \CP^1$
  by  calculating the reduced Lagrangian $L_{red}$ in any convenient coordinates and then   writing out the 
 resulting Euler-Lagrange equations.   We use the coordinates $r, x, y$ as  above 
 with $w = \eta_2/\eta_1 = x +iy$. (See (\ref{eq:quotientmap}, (\ref{eq_etas}) and also 
 eq. (\ref{eq_rijxy}), (\ref{eq_xisquared}) and (\ref{eq_Vxy}). )
 Then 
$$K_0 = \smfr12 {\dot r}^2 + \smfr12 \kappa(x,y) r^2 ({\dot x}^2+{\dot y}^2)\qquad\qquad \kappa = \fr{3\mu_1\mu_2}{||\xi||^4},$$
where $\| \xi\|^2$ as a function of $x,y$ is obtained by plugging the expressions (\ref{eq_rijxy}) into Lagrange's identity
(\ref{eq_xisquared}). Then
$$L_{red}(r,x,y,\dot r,\dot x,\dot y) = K_0 + \fr{1}{r} V(x,y)$$
and  so the Euler-Lagrange equations are
\begin{equation}\label{eq_ELrxy}
\begin{aligned}
\ddot r&= -\fr{1}{r^2} V + \k r ({\dot x}^2+{\dot y}^2)\\
(\k r^2 \dot x)^\cdot &=  \fr{1}{r} V_x +  \fr12 \kappa_x r^2 ({\dot x}^2+{\dot y}^2)\\
(\k r^2 \dot y)^\cdot &=  \fr{1}{r} V_y+  \fr12 \kappa_y r^2 ({\dot x}^2+{\dot y}^2).
\end{aligned}
\end{equation}
Conservation of energy gives
$$K_0 - \fr{1}{r} V(x,y) = -h.$$

{\bf Remark} The expression $\kappa(x,y) (dx^2 + dy^2)$ describes a spherically symmetric metric on the shape sphere. 
For example, when $m_1 = m_2 = m_3$ one computes that $\kappa = {{1} \over { (1 + x^2 + y^2)^2}}$
which is the standard conformal factor for expressing the metric on the sphere of radius $1/2$
in stereographic coordinates $x,y$.

These equations describe the zero angular momentum three-body problem reduced to 3 degrees of freedom by elimination of 
all the symmetries and separated into size and shape variables.   
An additional  improvement is achieved by blowing up  the triple collision singularity at $r=0$ 
by  introducing the time rescaling ${}' = r^{\smfr32}\;\dot{}$ and the variable $v= r'/r$ \cite{Mc}. 
 The result is the following system of differential equations:
\begin{equation}\label{eq_blowuprxy}
\begin{aligned}
r' &= vr\\
v' &= \smfr12 v^2 + \k ({x'}^2+{y'}^2)-V \\
(\k x')'  &= V_x - \smfr12 \k v x' + \smfr12 \k_x({x'}^2+{y'}^2)  \\
(\k y')'  &= V_y - \smfr12 \k v y' + \smfr12 \k_y({x'}^2+{y'}^2).
\end{aligned}
\end{equation}
The energy conservation equation is
\begin{equation}\label{eq_energystereo}
\smfr12 v^2 +  \smfr12 \k ({x'}^2+{y'}^2)- V(x,y) = -rh
\end{equation}
Note that $\{r=0\}$ is now an invariant set for the flow, called the {\em triple collision manifold}.  
Also, the differential equations for $(v,x,y)$ are independent of $r$.  Call the rescaled time variable $s$. Since the rescaling is such that $t'(s) = r^\frac32$, behavior near triple collision that is fast with respect to the usual time may be slow with respect to $s$.  Motion on the collision manifold could be said to occur in zero $t$-time.

\begin{figure}
\scalebox{.6}{\includegraphics{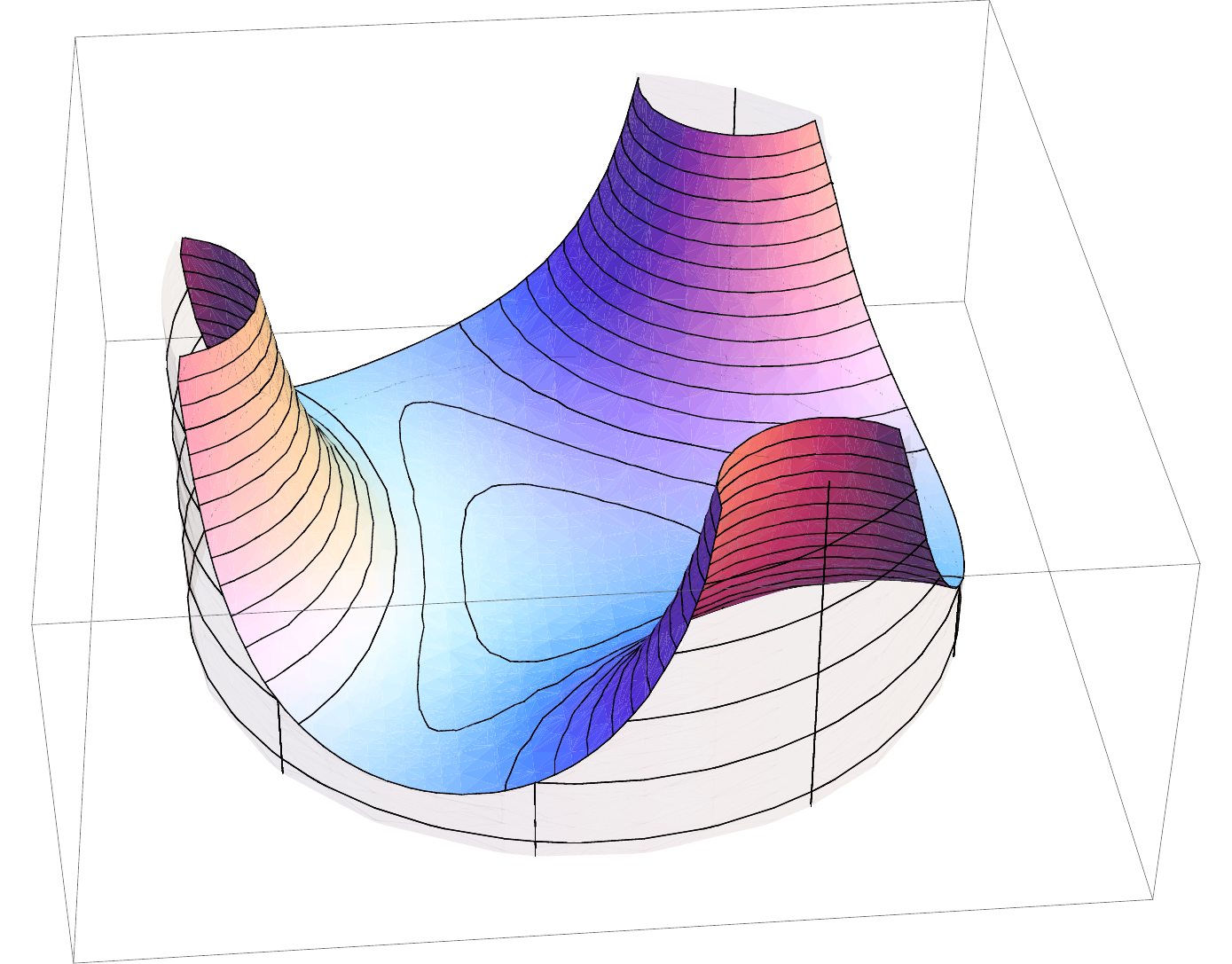}}
\caption{Half of the Hill's region for the equal mass three-body problem in stereographic coordinates $(x,y,r)$ where the unit disk corresponds to the upper half of the shape sphere.  In these coordinates, the configurations with collinear shapes are in the cylindrical over the unit circle.}\label{fig_Hillstereo}
\end{figure}

System (\ref{eq_blowuprxy}) could be written as a first-order system in the variables $(r,x,y,v,x',y')$.  Fixing an energy $-h<0$ defines a five-dimensional energy manifold
$$P_h = \{(r,x,y,v,x',y'):r\ge 0, \;(\ref{eq_energystereo})\,holds\}.$$
The projection of this manifold to configuration space is the {\em Hill's region}, $Q_h$.  Since the kinetic energy is non-negative the Hill's region is given by
$$Q_h = \{(r,x,y): 0 \le r \le V(x,y)/h  \}.$$
Figure~\ref{fig_Hillstereo} shows the part of the Hill's region over the unit disk in the $(x,y)$-plane for the equal mass case.  The resulting solid region has three boundary surfaces.  The top boundary surface $r = V(x,y)/h$ is part of the projection to configuration space of the {\em zero-velocity surface}, the bottom surface,  $r=0$, is contained in the projection of the triple collision manifold.  The side walls are part of the vertical cylinder over the unit circle which represents collinear shapes.

\subsection{Visualizing the Syzygy Map.}
Figure~\ref{fig_Hillspherical} shows a different visualization of the same Hill's region.  This time the shape is viewed as a point 
${\vec s}= (s_1,s_2,s_3)$  on the unit sphere which is then scaled by the size variable $r$ to form $r {\vec s}$ and plotted.  The region of figure~\ref{fig_Hillstereo} corresponds to the upper half of the solid in the new figure.  The collision manifold $r=0$ is collapsed the origin so that the bottom circle of the syzygy cylinder in figure ~\ref{fig_Hillstereo} has been collapsed to a point.  The collinear states within
the Hill region now form the  an unbounded ``three-armed''
 planar region homeomorphic to the interior of the unit  disk of  figure ~\ref{fig_Hillstereo}.  The syzygy map 
is the flow-induced map from the top half of the boundary surface in figure~\ref{fig_Hillspherical}  to the interior of this 
three-armed planar region.  

%However, the blown-up coordinates of figure~\ref{fig_Hillstereo} will be  needed for studying behavior %near triple collision, as needed for the continuity at the punctures of theorem 1.

\begin{figure}
\scalebox{.6}{\includegraphics{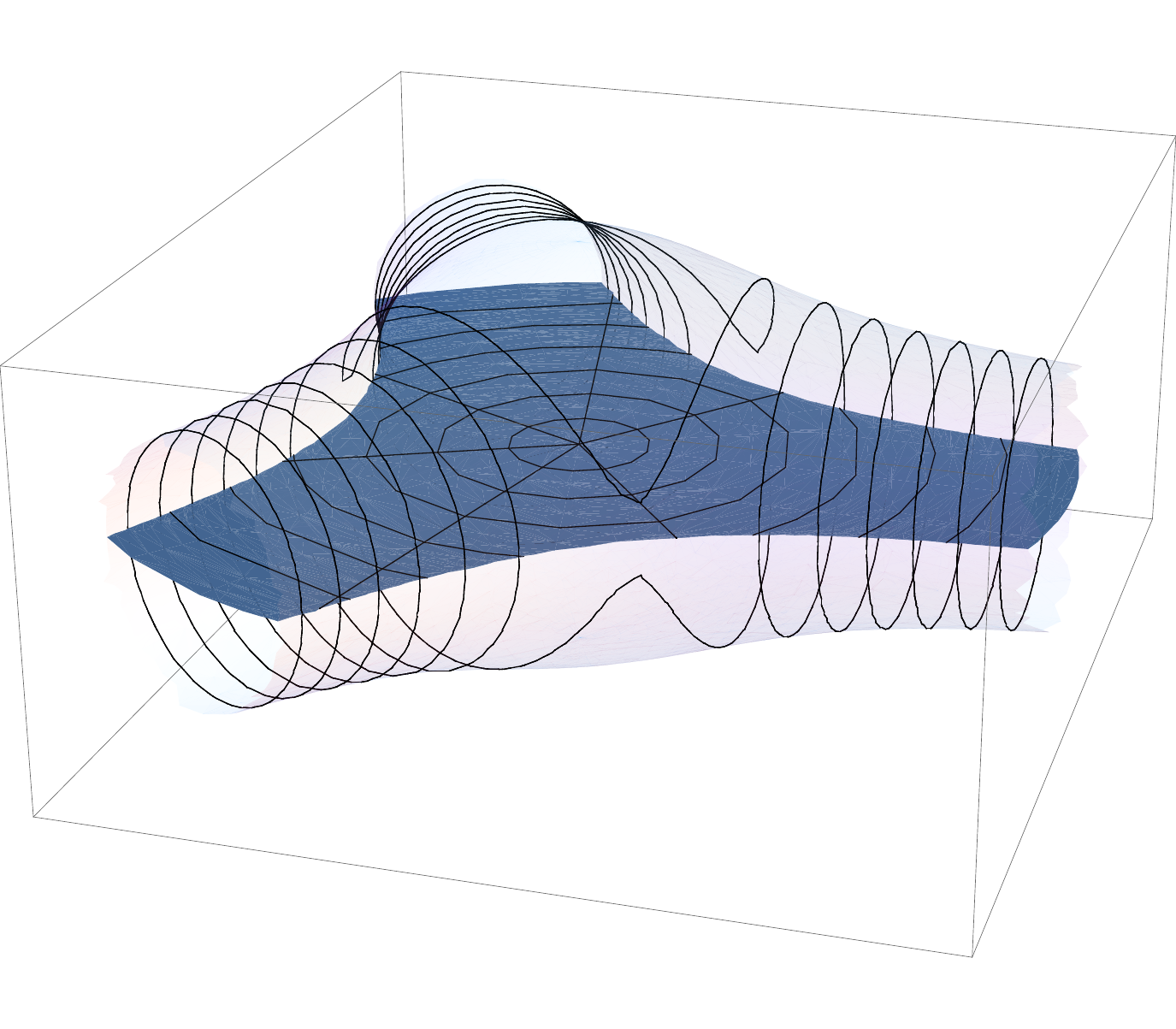}}
\caption{The entire Hill's region for the equal mass three-body problem in coordinates $r(s_1,s_2,s_3)$. The syzygy configurations $C_h$ form the planar region, homeomorphic to a disk, dividing the Hill's region in half.  The origin represents triple collision (which we count as a syzygy). The syzygy map maps from the upper half of the Hill region's boundary, $\partial Q_h^+$,   to $C_h$.} \label{fig_Hillspherical}
\end{figure}

Figure~\ref{fig_syzygyimage} illustrates the behavior of the map in the equal mass case, using coordinates which compress the unbounded region $C_h$ into a bounded one.   The open upper hemisphere of the shape sphere has been identified with the domain of the syzygy map.  The figure shows the numerically computed images of several lines of constant latitude and longitude on the shape sphere.   The figure illustrates some of the claims of theorem~\ref{thm:cty}.  Note that the image of the map seems to be strictly smaller than $C_h$ and is apparently bounded by curves connecting the binary collision points on the boundary.  It does contain a neighborhood of the binary collision rays however.  Also note that a line of high latitude, near the Lagrange homothetic initial condition (the North pole),  maps to a small curve encircling the origin.  Apparently there is a strong tendency, as yet unexplained,  to reach syzygy near binary collision and near the boundary of the image.

\begin{figure}
\scalebox{.6}{\includegraphics{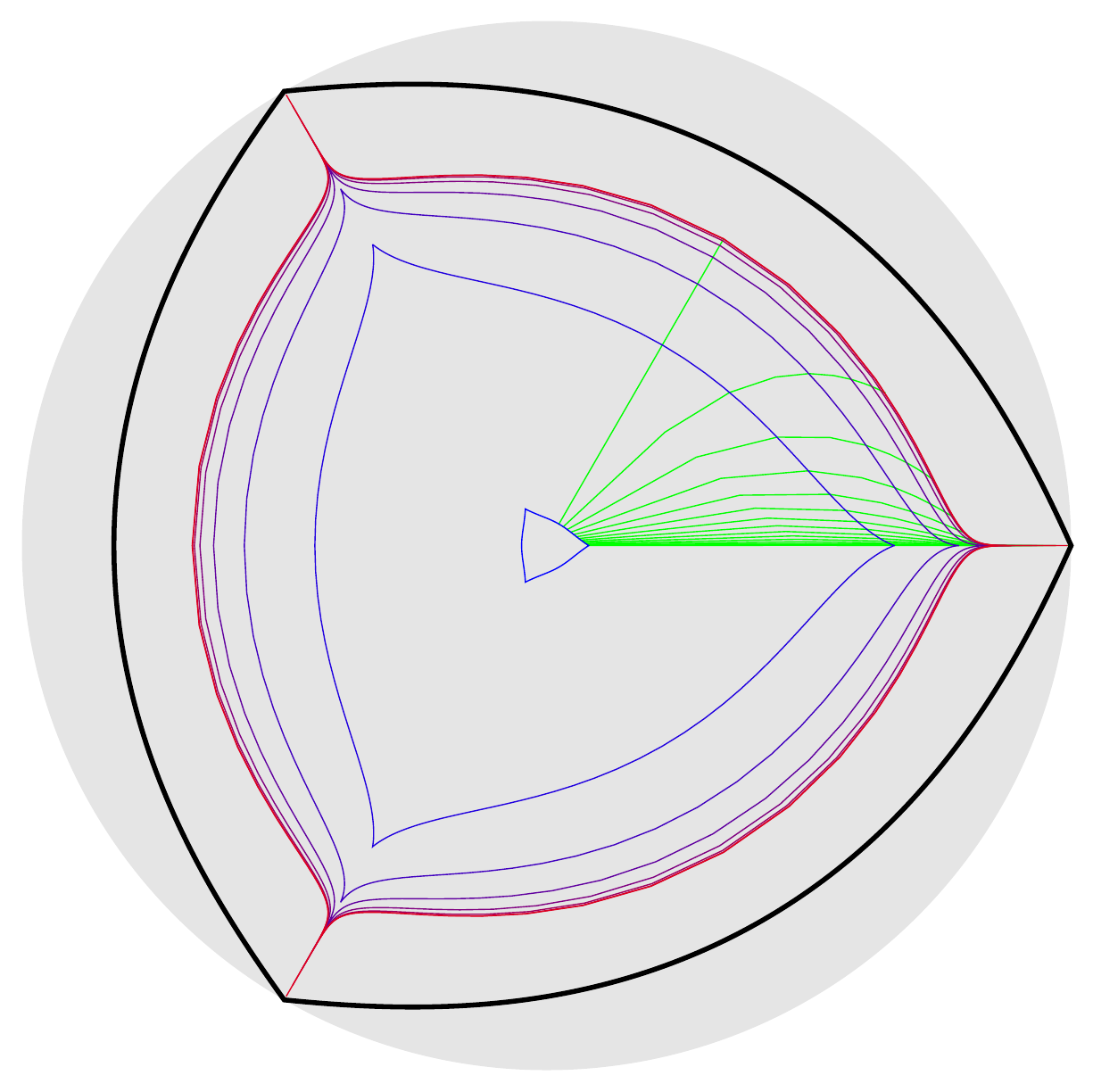}}
\caption{Image of the syzygy map in the equal mass case.  The open upper hemisphere of the shape sphere has been identified with the domain of the syzygy map.  Several circles of constant latitude and a sector of arcs of constant longitude have been followed forward in time to the first syzygy and the resulting image points plotted.  The coordinates on the image are $\arctan(r)\,(s_1,s_2)$ so that the horizontal plane in figure~\ref{fig_Hillspherical} is compressed into the shaded open disk.  The syzygy configurations $C_h$ now form the region bounded by the heavy black curve.   The image of the map, however, seems to be strictly smaller.} \label{fig_syzygyimage}
\end{figure}

\subsection{Flow on the collision manifold: linearization results}\label{sec_collisionflow}

We will need some information about the flow on the triple collision manifold $r = 0$ which we take from \cite{Moe}.  
Eq. (\ref{eq_energystereo}) expresses   the triple collision manifold as a two-sphere bundle over the shape sphere,
with  bundle projection  $(0,x,y, v, x^{\prime}, y^{\prime}) \to  (x,y)$.  
 The vector field (\ref{eq_blowuprxy}) restricted to the triple collision manifold $r = 0$ flow has $10$ critical points, 
 coming in  pairs  $\{ p_+, p_-\}$, one pair  for each central configuration $p$.   Two of these
 central configurations  correspond to the equilateral triangle 
 configurations of Lagrange and are located in our $xy$ coordinates at  the origin    and at  infinity.  $(x,y) = \infty$.
 We write $L$ for the one at the origin. The other  three
 central configurations are  collinear,  were   found by Euler, and are   located 
on the unit circle in the $xy$ plane,  alternating  between  the three binary collision points.
The two equilibria $p_{\pm} = (x_0, y_0, v_{\pm}, 0,0)$  for a given central configuration $p = (x_0, y_0)$ are  obtained by  
solving for $v$ from the $r=0$  energy equation (\ref{eq_energystereo})
to get $v=v_{\pm}=\pm\sqrt{2V(x_0,y_0)}$.  The positive square-root  
$v_+ $   corresponds to solutions `exploding out'' homothetically from that configuration, and
the negative square-root   $v_-$  with $v = v_- < 0$ corresponds to solutions  collapsing into that central  configuration.  
Associated to each central configuration we also have the corresponding homothetic solution, which lives in $r > 0$
and forms a  heteroclinic connection  connecting $p_+$ to $p_-$.

We will also need information regarding the stable and unstable manifolds of  the equilibria $p_{\pm}$. 
This information can be found in \cite{Moe}, Prop 3.3. 
\begin{proposition}
\label{prop_linearization}Each equilibrium $p_{\pm} = (0,x_0, y_0, v_{\pm}, 0, 0)$ is linearly hyperbolic. Each has
  $v_{\pm}$ as an eigenvalue with corresponding  eigenvector   tangent to the
homothetic solution and so transverse to the collision manifold.   Its remaining 4  eigenvectors  are tangent to 
the collision manifold.   

The Lagrange point $L_-$ has a 3 dimensional stable manifold  transverse  to the collision manifold
%(as it contains the $v_-$ eigenspace which is tangent to  the homothetic collapsing orbit) 
and a 2-dimensional unstable manifold  contained in
the collision manifold.  The linearized projection of the unstable eigenspace to the tangent space to the shape sphere is onto.  
At $L_+$ the dimensions and properties of the stable and unstable manifolds are reversed.
 
 The Euler points $E_-$  each  have a 2-dimensional stable manifold  transverse
 to the collision manifold 
 %(a transverse direction being  the $v_-$ eigenspace) 
 and contained in the collinear invariant submanifold 
 and  a  3-dimensional unstable manifold    contained in the collision manifold.  At $E_+$, 
  the dimensions and properties of the stable and unstable manifolds are reversed.
 \end{proposition}
 
{\bf Remarks.} 
1. The $r > 0$ solutions in the stable manifold of $L_-$
are solutions which limit to triple collision in forward time, tending asymptotically
to the Lagrange configuration in shape.  The final arcs of the  solutions of this set can be obtained  
 by minimizing the Jacobi-Maupertuis length between points  $P$ and the triple collision point $0$
for $P$ varying in some open set of the form $U \setminus C$ where $U$  is a \nbhd  of $0$ and $C$
is the collinear set.    

2.  Because the unstable eigenspace for $L_-$ projects linearly onto  
the tangent space to the shape sphere it follows that the unstable manifold of $L_-$
cannot be contained in any shape sphere  neighborhood  $x^2+y^2 <  \epsilon$ of $L$.  

3. The unstable manifold of an exploding Euler equilibrium $E_+$
lie  entirely within the collinear space $z =0$
and  real solutions lying in it   form a two-dimensional set of curves.
These curves can be obtained by  minimizing the Jacobi-Maupertuis length
among all {\it collinear} paths connecting points  $P$ and the triple collision point $0$
as $P$  varies over collinear configurations in   some neighborhood $U \cap C$ of  $0$. 

4. 
The Sundman inequality implies that $v^{\prime} \ge 0$ everywhere on the triple collision manifold
and that $v(s)$ is strictly increasing except  at the 10 equilibrium points.  That is, $v$ acts like a Liapanov function on the
collision manifold.  

\section{The Syzygy Map}\label{SecSyzygyMap}
In this section the syzygy map taking brake initial conditions to their first syzygy will be studied.   As mentioned in the introduction, it will be shown that every non-collinear, zero-velocity initial condition in phase space can be followed forward in time to its first syzygy.  The goal is to study the continuity and image of the resulting mapping.

\subsection{Existence of Syzygies}\label{SecExistSyzygy}
In \cite{Mont, Mont2} Montgomery shows that every solution of the zero angular momentum three-body problem, except the Lagrange (equilateral) homothetic triple collision orbit, must have a syzygy in forward or backward time.  In forward time, the only solutions which avoid syzygy are those which tend to  Lagrangian triple collision.  
We now rederive this result, using our coordinates.
 
The result will  follow  from a study of the differential equation governing the (signed)  distance to syzygy in shape space.  
We  take for the signed distance 
$$z = 1-x^2-y^2$$
where $x,y$ are the   coordinates of section~\ref{SecEqsOfMotion}.
Note the unit circle $z=0$ is precisely the set of collinear shapes.  
From (\ref{eq_ELrxy}), one finds
$$
\begin{aligned}
\dot z&= \frac{2}{\k r^2} p_1\\
\dot p_1 &= -\frac{1}{r}(xV_x+yV_y)-r^2(\dot x^2+\dot y^2)(\k +\smfr12(x\k_x+y\k_y)).
\end{aligned}
$$
A computation shows that
$$\k + \smfr12(x\k_x+y\k_y) = \frac{c(1-x^2-y^2)}{\|\xi\|^6}$$
where 
$$c=\frac{3 m_1m_2m_3(m_1m_2+m_1m_3+m_2m_3)}{(m_1+m_2+m_3)^2}.$$
Using this and proposition~\ref{prop_Vderiv} gives
\begin{equation}\label{eq_zp1}
\begin{aligned}
\dot z&= \frac{2}{\k r^2} p_1\\
\dot p_1 &= -F_1(r,x,y,\dot x,\dot y) z
\end{aligned}
\end{equation}
where
$$F_1 = \frac{1}{r}\phi(x,y) + \frac{c r^2 (\dot x^2+\dot y^2)}{\|\xi\|^6}.$$
Note that $F_1$ is a smooth function for $r>0$ and satisfies $F_1\ge 0$ with equality only when $x=y=\dot x = \dot y=0$.

Using (\ref{eq_zp1}) one can construct a proof of Montgomery's result.  First note that $z=0$ defines the syzygy set and $z=p_1=0$ is an invariant set, namely the phase space of the  collinear three-body problem.  Without loss of generality, consider an initial condition with shape $(x,y)$ in the unit disk, i.e., the upper hemisphere in the shape sphere model.  Our goal is to show that all such solutions reach $z=0$.
% The following proposition shows that if the initial shape is not equilateral and if the initial shape velocity tends toward syzygy then the solution reaches syzygy in a bounded time provided the size remains bounded.

{\bf Remark.} The key result of \cite{Mont} is a differential equation  very similar to   eq (\ref{eq_zp1})
for a variable which was also called $z$ but which we will call   $z_I$ now, in order to compare the two. 
The relation between the current $z$ and this $z_I$ is  
$$z_I = { {z} \over {2-z}}$$
and can be derived from the expression $z_I = {{ 1 - x^2 - y^2} \over {1 + x^2 + y^2}}$
for the height component of the stereographic projection map $\R^2 \to S^2 \setminus {(0,0,-1)}$.
The important values of these functions on the shape sphere are
$$\text{ collinear}: z= 0 ; z_I = 0$$
$$\text{Lagrange, pos. oriented}: z= 1 ; z_I = 1$$
$$\text{Lagrange, neg. oriented}: z= +\infty ; z_I = -1$$

\begin{proposition}\label{prop_bounded}
Consider a solution with initial condition lying inside the punctured unit disc in the shape plane
:   $0<z(0)<c_1<1$,  and pointing outward (or at least not inward): $\dot z(0) \le 0$.  Assume that the size of the configuration satisfies $0< r(t) \le c_2$ for all time $t \ge 0$, and some positive constant  $c_2$.  
Then there is a constant $T_0(c_1, c_2)>0$ and a time $t_0\in [0,T_0]$ such that $0< z(t) \le z(0)$ for $t\in [0,t_0)$ and $z(t_0)=0$.
\end{proposition}
\begin{proof}
Consider the projection of the solution to the $(z,p_1)$ plane. By hypothesis, the initial point $(z(0),p_1(0))$ lies in the fourth quadrant of the plane.   Since $F_1\ge 0$,  (\ref{eq_zp1}) shows that $z(t)$ decreases monotonically  on any time interval $[0,t_0]$ such that $z(t)\ge 0$ so $0\le z(t)\le c_1$ holds.  An upper bound for $t_0$ will be now be found.  

Let $\a_1$ denote a clockwise angular variable in the $(z,p_1)$ plane. Then (\ref{eq_zp1}) gives
$$(z^2+p_1^2)\dot \a_1 = F_1(r,x,y,\dot x,\dot y) z^2 + \frac{2}{\k r^2} p_1^2\ge \frac{1}{r}\phi(x,y) z^2 + \frac{2}{\k r^2} p_1^2.$$
On the set where $0\le z\le c_1$ and $0<r \le c_2$,  the coefficients of $z^2$ and $p_1^2$ of the right-hand side each have positive lower bounds. Hence there is a constant $k>0$ such that $\dot \a_1 \ge k$ holds on the interval $[0,t_0]$.  It follows that $t_0 \le \frac{\pi}{2 k}$.   

It remains to show that the solution actually exists long enough to reach syzygy.  It is well-known that the only singularities of the three-body problem are due to collisions.  Double collisions are regularizable (and in any case, count as syzygies).  Triple collision orbits are known to have shapes approaching either the Lagrangian or Eulerian central configurations.  The Lagrangian case is ruled out by the upper bound on $z(t)$.  
 Eulerian triple collisions can only occur for orbits in the invariant collinear manifold, so this case is also ruled out. 
\end{proof}

The next result gives a uniform bound on   time to syzygy for solutions    far from triple
collision.    It is predicated on the well-known fact that if  $r$ is large and the energy is negative
then configuration space is split up into three disjoint regions, one for each choice of binary pair,
and   within each region that binary pair moves approximately in a bound Keplerian motion.  
The approximate period of that motion is obtained from knowledge of the two masses and the
percentage of the total energy involved in the binary pair motion.  The `worst' case, i.e. longest
period, is achieved by taking the pair to be that with  greatest   masses,
in a parabolic escape to infinity so that all the energy $-h$ is involved in their near Keplerian motion,
and thus the kinetic energy of the escaping smallest mass is tending to zero.  
This limiting `worst case' period is 
$$\tau_* = {1 \over {(2 h)^{3/2}} }[ {{m_i ^2 m_j ^2}\over {m_i + m_j}}]^{3/2}$$
where  the excluded mass $m_k$ is the smallest of the three.

Here is a precise proposition. 
\begin{proposition}\label{prop_unbounded}  
Let $\tau_*$ be the constant above and let $\beta$ be any positive constant less than $1$. Then there is a (small) positive constant
$\epsilon_0 = \epsilon_0 (\beta)$ such that all solutions 
 with    $r(0) \ge { 1 \over {2 \epsilon_0}}$, energy $-h <0 $, and angular momentum $0$
  have a syzygy within the  time interval $[0, {1 \over \beta} \tau_*]$.
  Moreover, this syzygy occurs before $r = {1 \over 2}  r(0)$,  so  that $r(t) \ge {1\over \epsilon_0}$
 at this syzygy. 
   \end{proposition}
   
   The final  sentence of the proposition  is added because if initial conditions are such that  the   far  mass
   approaches  the binary pair at a high speed  then  the perturbation conditions required in  the proof will   
  be violated quickly: $r(t)$ will  become $O(1)$ in a short time, 
   well before the required syzygy   time ${1 \over \beta} \tau_*$.  In this case the approximate  Keplerian frequency of the bound
   pair is also  accordingly high, guaranteeing a   syzygy  well before   perturbation estimates
   break down and well before the required syzygy time.

   \begin{proof} The proof is  perturbation theoretic and divides   into two parts.
In the first part we derive the equations of motions in
a coordinate system quite similar to the one which   Robinson  used to compactify the infinity corresponding to $r \to \infty$ at constant $h$.
In the second part  we use these equations to derive the result.

Part 1. Deriving the equations in the new variables. 
We will use coordinates adapted to studying the dynamics near infinity which are a variation on those introduced by McGehee  (\cite{Mc2}) and then modified for  the planar three-body problem by Easton, McGehee and Robinson (\cite{Easton, EastonMcGehee, Robinson}).   Going back to the Jacobi coordinates $\xi_1, \xi_2$
 set 
$$\xi_1 = u e^{i \theta},  \xi_2 = \rho e^{i\theta}.$$
thus defining coordinates 
 $(\rho, u, \theta) \in \R^+ \times \C \times S^1$. The
 variables   $\rho, u$   coordinatize    shape space
while   $\theta$ coordinatizes the overall rotation in inertial space.  
Then
$$r^2 = \mu_1 |u|^2 + \mu_2 \rho^2.$$
The reduced kinetic energy (i.e. metric on shape space) is given by 
$$K_0 = \frac{\mu_1}{2} |\dot u|^2 +\frac{\mu_2}{2}  \dot \rho^2 -  \frac{\mu_1^2}{2r^2} ( u \wedge \dot u )^2 .$$
Here
$$u \wedge \dot u = u_1\dot u_2-u_2 \dot u_1 = Im(\bar{u}\dot u).$$
($K_0$  can be computed two ways: either plug the expressions for the $\dot \xi_i$ in terms of $\dot \rho, \dot u, \dot \theta, ...$
 into the expression  for the kinetic energy
and then minimize over $\dot \theta$, or  plug these same expressions into the reduced metric expression.)   
Next we make a Levi-Civita transformation by setting $u= z^2$ which gives
$$K_0 =2\mu_1|z|^2 |\dot z|^2 +\frac{\mu_2}{2}  \dot \rho^2 -  \frac{2\mu_1^2|z|^4}{r^2} ( z \wedge \dot z )^2 .$$

We want to introduce the conjugate momenta and take a Hamiltonian approach.
We can write
$$K_0 =  2\mu_1|z|^2\langle A(\rho,z))\dot z,\dot z\rangle + {1 \over 2}\mu_2 \dot\rho^2$$
where $A$ is the symmetric matrix $A=I-B$ with
$$B(\rho,z)  = \frac{\mu_1|z|^2}{r^2} \m{z_2^2&-z_1z_2\\ -z_1z_2&z_1^2}.$$
Here $z_1, z_2$ denote the real and imaginary parts of $z$, not new complex variables.
The conjugate momenta are
$$\eta = 4 \mu_1 |z|^2 A \dot z \qquad   y = \mu_2\dot\rho.$$
If we write
$$A^{-1} = I + \mu_1 f(\rho,z)$$
we find
$$K_0 = \frac{1}{8\mu_1 |z|^2}|\eta|^2 +\frac{1}{2\mu_2} y^2 +\frac{1}{8|z|^2}\langle f(\rho,z)\eta,\eta\rangle   .$$
For later use, note that $\mu_1 f =  B + B^2 + \ldots$ is a positive semi-definite symmetric matrix and $f=O(|z|^4/r^2)$.

The  negative of the   potential energy is 
$$U(\rho,z) = \frac{m_1 m_2}{|z|^2} + \frac{m \mu_2}{\rho}  + g(\rho,z)$$
where the ``coupling term'' $g$ is 
\[
\begin{array}{lcl} g(\rho,z) & = & \dsty {{m_1 m_3} \over {\| \rho + \nu_2 z^2 \|}} + 
{{m_2 m_3} \over { \| \rho - \nu_1 z^2 \|}}- {{m \mu_2} \over{\rho}} \\
 \end{array} = O(|z|^4/\rho^3)
\] 
For $r > r_0$ sufficiently large the Hill region breaks up into three disjoint regions. 
We are interested for now in the region   centered on the 12 binary collision ray. In this
case, for $r> r_0$ sufficiently large 
one finds that $|z|$ is bounded by a constant depending only on the masses and $h$
and the choice of $r_0$. This bound  on $|z|$ tends to a nonzero constant as $r_0 \to \infty$.
It then  follows from the identity $r^2 = \mu_1 \rho^2 + \mu_2 |z|^4$   that $r = \sqrt{\mu_2} \rho + O(1)$ for $r$ large.

Next we  complete the regularization of  the binary collision by means of the  the time rescaling 
${{d} \over {d \tau}} = |z|^2 {{d} \over {d t}} $.  Using the Poincar\'e trick, the rescaled solutions with energy $-h$ become the zero-energy solutions of the Hamiltonian system with Hamiltonian function, 
$$
\begin{aligned}
\tilde H &= |z^2|(K_0 - U + h) \\
&=\frac{1}{8\mu_1}|\eta|^2 +\frac{|z|^2}{2\mu_2}y^2 +\frac{1}{8}\langle f(\rho,z)\eta,\eta\rangle 
 - m_1m_2   -  \frac{m \mu_2|z|^2}{\rho} - g(\rho,z)|z|^2 +h|z|^2 .
\end{aligned}
$$

Computing Hamilton's equations,  making the additional substitution 
$$x = 1/ \rho$$
to move infinity to $x=0$, and writing $\, ' $ for $d/d \tau$   gives the differential equations:
\begin{equation}\label{Robinsoneqn}
\begin{aligned}
 x^{\prime} &= -{|z|^2 \over \mu_2} x^2 y \\
 y ^{\prime} &= - m \mu_2  |z|^2x^2 + |z|^2 g_{\rho} - \frac{1}{8} \langle f_{\rho} \eta,\eta \rangle\\  
 z ^{\prime}  &= \frac{1} {4 \mu_1} \eta +   \frac{1} {4}f\eta \\
  \eta^{\prime} &= -2(h+\frac{y^2}{2\mu_2}-m\mu_2 x - g)z - \frac{1}{8} \langle f_{z} \eta,\eta \rangle + |z|^2g_z \\
\end{aligned}
\end{equation}
where the subscripts on $f$ and $g$ denote partial derivatives.  These satisfy the bounds:
$$
\begin{aligned}
f = O(x^2 |z|^4)\quad f_z = O(x^2|z|^3)\quad  f_{\rho} = O(x^3|z|^4)\\  
g  = O(x^3|z|^4)\quad  g_z = O(x^3|z|^3)\quad  g_{\rho} = O(x^4|z|^4).
\end{aligned}$$
The energy equation is $\tilde H = 0$.

Infinity has become   $x=0$, an invariant manifold.  At infinity we have  
  $x'=y'=0$ while  other variables satisfy
$$
\begin{aligned}
 z ^{\prime}  &= \frac{1} {4 \mu_1} \eta  \\
  \eta^{\prime} &= -2(h+\frac{y^2}{2\mu_2})z.
\end{aligned}
$$
Since $y$ is constant this is the equation of a two-dimensional harmonic oscillator.

Part 2.  Analysis.  Observe that the configuration is 
in  syzygy if and only if  the   variable $u$ is real .
 Since $u =z^2$ we have   syzygy at   time $t$  if and only if   $z(t)$ intersects either the real axis or the imaginary
  axis.  If $z$'s dynamics were exactly that of a harmonic oscillator
  then it would intersect one or the other axis (typically both)
  twice per period.  We argue that $z$ is sufficiently close to an oscillator
  that these intersections persist.
  
  We begin by establishing uniform bounds on $z$ and $\eta$ valid for all $x$ sufficiently small.
  These come from the energy. 
  We rearrange the expression for energy into the form
   $$\frac{1}{8\mu_1}|\eta|^2 +\frac{|z|^2}{2\mu_2}y^2 +\frac{1}{8}\langle f(\rho,z)\eta,\eta\rangle 
 +(h -  m\mu_2 x  -g(x,z) )|z|^2 = m_1m_2.$$
  Choose $\epsilon_0$ so that $x \le \epsilon_0$ implies $|m\mu_2 x+ g | \le h/2$.
  By the positive semi-definiteness of $f$ we have
   $$\frac{1}{8\mu_1}|\eta|^2   +(h/2 )|z|^2 \le m_1m_2$$
   which gives our uniform bounds on $z, \eta$. 

It now follows from equations~(\ref{Robinsoneqn})
that for all $\epsilon < \epsilon_0$ and $x \le \epsilon$ we have
\begin{equation}\label{eqnEstimates}
\begin{aligned}
 |x^{\prime}| \le C x^2 |y| \\
  |y ^{\prime}| \le C x^2  \\
  z ^{\prime}  = {1 \over {4 \mu_1}} \eta +   O(\epsilon^2) \\
   \eta^{\prime} = 2 H_{12} z + O(\epsilon^2)  
\end{aligned}
\end{equation}
where $H_{12} = -h-\frac{y^2}{2\mu_2}+m\mu x +g$, a quantity which represents the energy of the binary formed by masses $m_1,m_2$.  Here $C$ is a constant depending only on the masses, the energy and $\epsilon_0$.

Fix an initial condition $x_0, y_0, z_0, \eta_0$ 
with $x_0 < \epsilon/2$.  Write $(x(t), y(t), z(t), \eta(t))$
for the corresponding solution. 
For $C$ a positive constant, write  ${\mathcal R} = {\mathcal R}_{C, \epsilon}$ for  the rectangle
$0 \le x \le  \epsilon,   |y-y_0| \le  C\epsilon$ in the $xy$ plane.  
We  will show there exists   $C_1$ depending only on the masses,
$h$, on the constant $C$ and $y_0$ such that  the projection $x(t), y(t)$  of our solution lies in  ${\mathcal R}_{C, \epsilon}$
for all times $t$ with $|t| \le 1/C_1 \epsilon$. 
Suppose that the projection of our  solution leaves the rectangle ${\mathcal R}$ in some time $t$.  
If it first leaves through the y-side, then we  have $C \epsilon = |y(t) - y_0| = |\int \dot y dt | < C \epsilon^ 2\int dt = C \epsilon^2 t$,
asserting that $1/\epsilon  \le t$.  Thus it takes at least a   time $1/\epsilon$ to escape out the $y$-side. 
To analyze escape through the $x$-side at $x = \epsilon$
we  enlist  Gronwall.   Let $C_2(y_0)= C(|y_0|+C\epsilon)$.  Compare  $x(t)$ to the solution $\tilde x$ to $\dot x = C_2 x^2$
 sharing initial condition with $x(t)$, so that
$\tilde x(0) = x_0$.  The exact solution   is $\tilde x = x_0/ (1 - x_0 C_2 t)$. Gronwall asserts $x(t) \le \tilde x (t)$ as long as 
$x(t), y(t)$ remain in the rectangle (so that the estimates (\ref{eqnEstimates}) are valid).    But $\tilde x(t) \le \epsilon$ for
 $t \le 1/\epsilon C_2$. Consequently  it takes our projected solution at least time $t = 1/\epsilon C_2$
to escape out of the $x$-side, and thus $(x(t), y(t))$  lies within  the rectangle for time $|t| \le 1/ \epsilon C_1$, with
$1/C_1 = min\{1, 1/C_2(y_0) \}$.

We now   analyze the oscillatory part of  equations ~(\ref{Robinsoneqn}). 
We have
$$H_{12} = -h-\frac{y^2}{2\mu_2}+m\mu x +g = - h -\frac{y^2}{2\mu_2} + O(\epsilon).$$
%and since $|f| \le C |u|^2 \epsilon^2$ for $x \le \epsilon$ we have that
%$\langle f (p_u), p_u \rangle \le C |z|^2 |\eta|^2 \epsilon^2 \le C \epsilon^2$
%from which it follows that 
%$$ H_{12}  = - h  -{1 \over 2} y^2 + O(\epsilon^2).$$
and let $H_{12}^0 = - h  -{1 \over 2} y_0^2$.  Then as long as   $(x,y) \in {\mathcal R}_{C,\epsilon}$
we have the bound 
$$| H_{12}^0 - H_{12} (x,y, z, \eta)| = O(\epsilon).$$ 
holds.    Thus the difference between the vector field defining our equations
and the  ``frozen oscillator''   approximating equations 
\begin{equation}
\label{frozenOsc}
\begin{aligned} 
  z ^{\prime}  = {1 \over {4 \mu_1}} \eta  \\
   \eta^{\prime} = 2 H_{12} ^0 z 
   \end{aligned}
\end{equation} 
which we get by throwing out the $O(\epsilon^2)$ error terms in (\ref{eqnEstimates}) and replacing  $H_{12}$ by $H_{12}^0$ is $O(\epsilon)$.

Now the period of the frozen oscillator is $T_0 = 2\pi \sqrt{2\mu_1}/\sqrt{|H_{12}^0|}$.  On the other hand we remain in $ {\mathcal R}_{C, \epsilon}$ for at least time $T_ {\mathcal R} = 1/ \epsilon C_1$.  Both of these bounds depend on $y_0$ but we have
$$\frac{T_0}{T_ {\mathcal R}} = \epsilon \frac{2\pi \sqrt{2\mu_1}\max\{1,C_2(y_0)\}}{\sqrt{h+y_0^2/2}} \le C_3 \epsilon$$
where $C_3$ does not depend on $y_0$.  Moreover we have a uniform upper bound 
$$T_0\le T_{\max} = 2\pi \sqrt{2\mu_1}/\sqrt{h}.$$

Hence for $\epsilon_0>0$ sufficiently small and $x_0<\epsilon/2, \epsilon<\epsilon_0$, a solution of~(\ref{eqnEstimates}) remains in 
${\mathcal R}_{C, \epsilon}$ for at least time $T_{\max}$ and  the difference between
the component $z(t)$ of our solution and the corresponding solution $\tilde z$ to the linear frozen oscillator equation (\ref{frozenOsc})  is of order $\epsilon$ in the $C^1$-norm: $C^1$ because we also get the $O(\epsilon)$ bound on $\eta = 4 \mu_1 z'$.   Now, any solution $\tilde z$ to the frozen
oscillator crosses either the real or imaginary axis, transversally, indeed at an angle of 45 degrees or more,
once per half-period. (The worst case scenario is when the oscillator is constrained to a line segment).  Thus the same
can be said of the real solution $z(t)$ for $\epsilon_0$
small enough: it crosses either the real or imaginary axis at least once per half period. 

Finally,  the time bounds involving $\tau_*$  of the proposition  were stated in the   Newtonian time.
To see these bounds use the inverse   Levi-Civita
transformation, and note that  a half-period of the Levi-Civita harmonic oscillator
corresponds to a full period of an approximate Kepler problem with energy $H_{12}^0$.
This   approximation becomes better and better as $\epsilon \to 0$.
Moreover, the Kepler period decreases monontonically with increasing absolute value of the Kepler  energy
so that the maximum period corresponds to 
the infimum of   $|H_{12}^0|$ and this is achieved by setting  
$y = 0$ (parabolic escape), in which case $H_{12}^0=-h$ and we get the claimed value of $\tau_*$ with $\beta = 1$.
(Letting  $\beta \to 1$ corresponds to  letting $\epsilon_0 \to 0$. )
\end{proof}

These  two propositions, Proposition \ref{prop_unbounded} and Proposition \ref{prop_bounded}, lead to Montgomery's result:  

\begin{proposition}\label{prop_forwardsyzygy}
Consider an orbit with initial conditions satisfying $0<z(0)\le 1$ and $r(0)>0$.  Either the orbit ends in Lagrangian triple collision with $z(t)$ increasing monotonically to 1, or else there is some time $t_0>0$ such  that $z(t_0)=0$.
\end{proposition}
\begin{proof}
Consider a solution with no syzygies in forward time.  By proposition~\ref{prop_unbounded} there must be an upper bound on the size: $r(t)\le c_2$ for some $c_2>0$.  Next suppose $z(0)=1$, i.e., the initial shape is equilateral which means $(x(0),y(0))=(0,0)$.   If  $(\dot x(0),\dot y(0))= (0,0)$ the orbit is the Lagrange homothetic orbit which is in accord with the proposition.
 If $(\dot x(0),\dot y(0))\ne (0,0)$ then for all sufficiently small positive times $t_1$ one has $0<z(t_1)<1$ and $\dot z(t_1)<0$.   Then proposition~\ref{prop_bounded} shows that there will be a syzygy.  It remains to consider orbits such that $0<z(0)<1$.
 
 Suppose $0<z(0)<1$ and that the orbit has no syzygies in forward time.  It cannot be the Lagrange homothetic  orbit and, by the first part of the proof, it can never reach $z(t)=1$.  Proposition~\ref{prop_bounded} then shows that $\dot z(t)>0$ holds as long as the orbit continues to exist, so $z(t)$ is strictly monotonically increasing.   By proposition~\ref{prop_unbounded} there is a uniform bound $r(t)\le c_2$ for some $c_2$.  There cannot be a bound of the form $z(t)\le c_1<1$, otherwise a lower bound $\dot\a_1\ge k>0$ as
 in the proof of proposition~\ref{prop_bounded} would apply.  It follows that $z(t)\rightarrow 1$.

To show that the orbit tends to triple collision, it is convenient to switch to the blown-up coordinates and rescaled time of equations~(\ref{eq_blowuprxy}).   Since the part of the blown-up energy manifold with $z\ge z(0)>0$ is compact, the $\w$-limit set must be a non-empty, compact, invariant subset of $\{z=1\}$, i.e., $\{(x,y) = (0,0)\}$.  The only invariant subsets are the Lagrange homothetic orbit and the Lagrange restpoints $L_\pm$.  Since $W^s(L_+)\subset \{r=0\}$,  the orbit must converge to the restpoint $L_-$ in the collision manfold as $s\into\infty$.  In the original timescale, we have a triple collision after a finite time.
 \end{proof}

\subsection{Existence of Syzygies for Orbits in the Collision Manifold}
In the last subsection, it was shown that  zero angular momentum orbits with $r>0$ and not tending monotonically to Lagrange triple collision have a syzygy in forward time.   This is the basic result underlying the existence and continuity of the syzygy map.  However, to study the behavior of the map near Lagrange triple collision orbits we need to extend the result to orbits in the collision manifold $\{r=0\}$.   This entails using the equations (\ref{eq_blowuprxy}).  

Setting $z=1-x^2-y^2$ as before one finds
$$
\begin{aligned}
z'&= \frac{2}{\k} p_2\\
p_2' &= -(xV_x+yV_y)-(x'^2+ y'^2)(\k +\smfr12(x\k_x+y\k_y)) + \smfr12 v p_2.
\end{aligned}
$$
As in the last subsection, this can be written
\begin{equation}\label{eq_zp2}
\begin{aligned}
z' &= \frac{2}{\k} p_2\\
p_2' &= -F_2(x,y,x',y') z - \smfr12 v p_2
\end{aligned}
\end{equation}
where
$$F_2 =\phi(x,y) + c (x'^2+y'^2).$$
$F_2$ is a smooth function satisfying $F_2\ge 0$ with equality only when $x=y=x'=y'=0$.  Recall that ${}'$ denotes differentiation with respect to a rescaled time variable, $s$.  

To see which triple collision orbits reach syzygy, introduce a clockwise angle $\a_2$ in the $(z,p_2)$-plane.  Then
\begin{equation}\label{eq_quadform}
(z^2+p_2^2)\a_2' = F_2 z^2 + \frac{2}{\k} p_2^2 +\smfr12 v z p_2.
\end{equation}
The cross term in this quadratic form introduces  complications and may prevent certain orbits with $r=0$ from reaching syzygy. 

To begin the analysis, consider a case where the cross term is small compared to the other terms, namely orbits  with bounded size
near binary collision.

\begin{proposition}\label{prop_binarysyzygy}
Let $c_2>0$.  There is a neighborhood $U$ of the three binary collision shapes and a time $S_0(U,c_2)$ such that if an orbit has shape $(x(s),y(s))\in U$ and size  $r(s)\le c_2$ for $s\in[0,S_0]$ then there is a (rescaled) time $s_0\in [0,S_0]$ such that $z(s_0)=0$.
\end{proposition}
\begin{proof}
 Let $\delta,\tau$ be the determinant and trace of the symmetric matrix of the quadratic form
$$\m{F_2& \smfr14 v\\ \smfr14v&2\k^{-1}}.$$
Iff $\delta, \tau>0$ then   the smallest eigenvalue of this matrix is greater than $2 \delta/\tau$, which yields the  lower bound
$$\a_2' \ge \fr {2\delta}\tau >0.$$
From (\ref{eq_rijxy}) it follows that near binary collision, one of the three shape variables $r_{ij}\approx 0$ while the other two satisfy $r_{ik}, r_{jk}\approx \sqrt{3}$.
Consider, for example, the binary collision $r_{12}=0$ at $(x,y)=(1,0)$.  One has $F_2\ge \phi(x,y)$ and
$\phi(x,y)\approx C/r_{12}^3$ for some constant $C>0$.  The matrix entry $\k^{-1}$ is bounded and $v$ can be estimates using the energy relation (\ref{eq_energystereo})
$$v^2 \le 2(V(x,y) - rh) = 2m_1m_2/r_{12} +O(1)$$
near $(x,y)=(1,0), r\le c_2$.
Using these estimates in the trace and determinant gives the asymptotic estimate
$$\a_2' \ge \frac{2F_2\k^{-1}-\smfr1{16}v^2}{F_2+2\k^{-1}} \approx 2\k^{-1}>0.$$
Similar analysis near the other binary collision points yields a neighborhood $U$ in which there is a positive lower bound for $\a_2'$.  This forces a syzygy ($z=0$) in a bounded rescaled time, as required.
\end{proof}

This result, together with some well-known properties of the flow on the triple collision manifold leads to a characterization of possible triple collision orbits with no syzygy in forward time.  

\begin{proposition}\label{prop_collisionsyzygy}
Consider an orbit with initial condition $r(0)=0$ and $0<z(0) \le 1$.  Either the orbit tends asymptotically to one of the restpoints on the collision manifold or there is a (rescaled) time $s_0>0$ such that $z(s_0)=0$.
\end{proposition}
\begin{proof}
Recall from section~\ref{sec_collisionflow} that
the flow on the triple collision manifold is gradient-like with respect to the variable $v$, i.e., $v(s)$ is strictly increasing except at the restpoints.  It follows that every solution which is not in the stable manifold of  one of the restpoints satifies $v(s)\rightarrow\infty$.  In this case, the energy equation (\ref{eq_energystereo}) shows that
the shape potential $V(x(s),y(s))\rightarrow\infty$ so the shape must approach one of the binary collision shapes.  For such orbits, proposition~\ref{prop_binarysyzygy} gives a syzygy in a bounded rescaled time and the proposition follows.
\end{proof}

We will also need a result analogous to proposition~\ref{prop_bounded}.  Consider an orbit in $\{r=0\}$ with initial conditions satifying $0<z(0)\le c_1<1$ and $z'(0)\le 0$.   It will be shown that, under certain assumptions on the masses, every such orbit has $z(s_0)=0$ at some time $s_0>0$.

On any time interval $(0,s_0]$ such that $z(s)>0$ one has $z'(s)<0$ and hence $z(s)$ is monotonically decreasing.  This follows from the ``convexity'' condition that $z''<0$ whenever $0<z<1$ and $z'=0$ which is easily verified from  (\ref{eq_zp2}).
If the orbit does not reach syzygy in forward time, then $0<z(s)\le c_1$ for all $s\ge 0$.  It is certainly not in the stable manifold of one of the Lagrangian restpoints at $z=1$ so, by proposition~\ref{prop_collisionsyzygy}, it must be in the stable manifold of one of the Eulerian (collinear) restpoints at $z=0$. 

It will now be shown that, for most choices of the masses, even orbits in these stable manifolds have syzygies.   Let $e_j = (x_j,y_j), j=1,2,3$ denote the collinear central configuration with mass $m_j$ between the other two masses on the line.  The two corresponding restpoints on the collision manifold are $E_{j-}$ and $E_{j+}$ with coordinates
$(r,x,y,v,x',y') = (0,x_j,y_j,\pm\sqrt{2V(x_j,y_j)},0,0)$.  The two-dimensional manifolds $W^s(E_{j-})$ and $W^u(E_{j+})$ are contained in the collinear invariant submanifold.  In particular, an orbit with $z(0)>1$ cannot converge to $E_{j-}$ in forward time.  On the other hand, $W^s(E_{j+})$ is three-dimensional and its intersection with the collision manifold is two-dimensional.  These are the orbits which might not reach syzygy.

{\bf Remark.}
Conjecture  \ref{conj1} asserts that all  orbits reach syzygy with $v<0$.  So, if the
conjecture is valid  then convergence to $E_{j+}$ would be impossible and the discussion to follow, and   the restriction on the masses in the theorem, would be unnecessary.

For most choices of the masses, the two stable eigenvalues of $E_{j+}$  with eigenspaces tangent to the collision manifold are non-real. See \cite{Moe}.   In this case we will say that $E_{j+}$ is {\em spiraling}.  The spiraling case is more common:  real eigenvalues occur only when the mass $m_j$ is much larger than the other two masses.  For all masses, at least two of the three Eulerian restpoints $E_{j+}$ are spiraling and for a large open set of masses where no one mass dominates, all of the Eulerian restpoints are spiraling.

\begin{proposition}\label{prop_spiralsyzygy}
Let $E_{j+}$ be a spiraling Eulerian restpoint.  Then there is a neighborhood $U$ of $E_{j+}$, and a time $S_0(U)$ such that any non-collinear orbit with initial condition in the local stable manifold $W^s_U(E_{j+})$ has a syzygy in every time interval of length at least $S_0$.
\end{proposition}
\begin{proof}
Introduce local coordinates in the energy manifold near $E_{j+}$ of the form $(r,a,b,z,p_2)$ where $(a,b)$ are local coodinates in the collinear collision manifold (the intersection of the invariant collinear manifold with $\{r=0\}$) and $(z,p_2)$ are the variables of (\ref{eq_zp2}).   The invariance of the collinear manifold $z=p_2=0$ implies that the linearized differential equations for $(z,p_2)$ take the form
$$\m{z\\ p_2}' = \m{\a&\b \\ \g&\d}\m{z\\ p_2}$$
where the eigenvalues of the matrix are the non-real eigenvalues at the restpoint.   The spiraling assumption implies that the angle $\a_2$ in the $(z,p_2)$-plane satisifies $\a_2'>k>0$ for some constant $k$.  So for the full nonlinear equations, one has
$\a_2'>k/2>0$ in some neighborhood $U$.  Since non-collinear orbits in the local stable manifold remain in $U$ and have nonzero projections to the $(z,p_2)$-plane, the proposition holds with $S_0=2\pi/k$.
\end{proof}

A mass vector $(m_1,m_2,m_3)$ will said to satisfy the {\em spiraling assumption}  if all of the Eulerian restpoints are spiraling.   The next result follows from propositions~\ref{prop_collisionsyzygy} and \ref{prop_spiralsyzygy}.

\begin{proposition}\label{prop_monotonecollisionsyzygy}
Consider an orbit in $\{r=0\}$ with initial conditions satifying $0<z(0)\le c_1<1$ and $z'(0)\le 0$.  Suppose that the masses satisfy the spiraling assumption.   Then there is a (rescaled) timetime $s_0>0$ such that $z(s_0)=0$.  Moreover $z(s)$ is monotonically decreasing on $[0, s_0]$
\end{proposition}

\subsection{Continuity of the Syzygy Map}\label{SecContinuity}
In this section, we will prove the statement in Theorem~\ref{thm:cty} about continuity  of the syzygy map and its extension to the Lagrange homothetic orbit.  

We begin by viewing the syzygy map in blow-up coordinates  $(r,x,y,v,x',y')$.  Recall the notations $P_h$ for the energy manifold and $Q_h$ for the Hill's region.  Points on the boundary surface $\partial Q_h$ of the Hill's region can be uniquely lifted to zero-velocity (brake) initial conditions in $P_h$.  Let  $\partial Q_h ^+$ be the subset of the boundary with shapes in the open unit disk (which corresponds to the open upper hemisphere in the shape sphere model).  This is the upper boundary surface in figure~\ref{fig_Hillstereo}.  It will be the domain of the syzygy map.

To describe the range, let $\tilde C_h =\{(r,x,y,v,x',y'): r > 0, x^2+y^2=1,\;(\ref{eq_energystereo})\,holds\}$ be the subset of the energy manifold whose shapes are collinear and let $C_h \subset Q_h$ be its projection to the Hill's region, i.e., the set
of  syzygy configurations having allowable energies.  $C_h$  is the cylindrical surface over the unit circle in figure~\ref{fig_Hillstereo} (the unit circle in $\{r=0\}$ is the blow-up of the puncture at the origin).  The first version of the syzygy map will be a map from part of  $ \partial Q_h ^+$ to $C_h$.

Recall that the origin $(x,y)=(0,0)$ represents the equilateral shape.  The corresponding point $p_0\in \partial Q_h ^+$ lifts to a brake initial condition in $P_h$ which is on the Lagrange homothetic orbit.   This orbit converges to the Lagrange restpoint without reaching syzygy.  It turns out that $p_0$ is the only point of $\partial Q_h ^+$ which does not reach syzygy and this leads to the first version of the syzygy map.

\begin{proposition}\label{prop_puncturedsyzygymap}
Every point of $\partial Q_h ^+ \setminus p_0$ determines a brake orbit which has a syzygy in forward time.  The map $F:\partial Q_h ^+ \setminus p_0 \into C_h$ determined by following these orbits to their first intersections with  $\tilde C_h$ and then projecting to $C_h$ is continuous.
\end{proposition}
\begin{proof}
The surface  $\partial Q_h$ intersects the line $x=y=0$ transversely at $p_0$.  So every point of $\partial Q_h ^+ \setminus p_0$ satisfies $0<z<1$ where $z = 1-x^2-y^2$ as before.   Also $r>0$ on the whole surface $\partial Q_h$.

Consider the brake initial condition corresponding to such a point.  Since all  the velocities vanish one has $\dot z(0) = 0$.  As in the first paragraph of the proof of proposition~\ref{prop_bounded}, one finds that $z(t)$ is decreasing as long as $z(t)\ge 0$.  In particular, $z(t)$ does not monotonically increase toward $1$.  It follows from proposition~\ref{prop_forwardsyzygy} that there is a time $t_0>0$ such that $z(t_0)=0$.  In other words, the forward orbit reaches  $\tilde C_h$.  

To see that the flow-defined map $\tilde F:\partial Q_h ^+ \setminus p_0 \into \tilde C_h$ is continuous, note that if the first syzygy is not a binary collision then $\dot z(t_0)<0$.  This follows since $z(t)$ is decreasing and since the orbit does not lie in the invariant collinear manifold with $z=\dot z = 0$.  So the orbit meets $\tilde C_h$ transversely.  After regularization of double collisions, even orbits whose first syzygy occurs at binary collision can be seen as meeting $\tilde C_h$ transversely as in the proof of proposition~\ref{prop_unbounded}.  It follows from transversality that $\tilde F$ is continuous.  Composing with the projection to the Hill's region shows that $F$ is also continuous.
\end{proof}

To get a continuous extension of $F$ to $p_0$, we have to collapse the triple collision manifold back to a point.  One way to do this is to replace the coordinates $(r,x,y)$ with $(X,Y,Z) = r(s_1,s_2,s_3)$ where $s_i$ are given by inverse stereographic projection.  The Hill's region is shown in figure~\ref{fig_Hillspherical}.  In this figure, $\partial Q_h ^+$ is the open upper half of the  boundary surface and $C_h$ is the planar surface inside (minus the origin).  In this model, triple collision has been collapsed to the origin $(X,Y,Z)=(0,0,0)$.   It is natural to extend the syzygy map by mapping the Lagrange homothetic point $p_0$ to the triple collision point, i.e., by setting $F(p_0)= (0,0,0)$.    The extension  maps into $\bar C_h = C_h\union 0$.

\begin{theorem}
If the masses satisfy the spiraling assumption then the extended syzygy map $F:\partial Q_h ^+ \into \bar C_h$ is continuous.   Moreover, it has topological degree one near $p_0$.
\end{theorem}
\begin{proof}
To prove continuity at $p_0$ it suffices to show that points in $\partial Q_h ^+$ near $p_0$ have their first syzygies near $r=0$.   The proof will use the blown-up coordinates $(r,x,y)$ and the rescaled time variable $s$.

Let $L_-$ denote the Lagrange restpoint on the collision manifold to which the orbit of $p_0$ converges.  Let $W^{s+}(L_-), W^{u+}(L_-)$ be the local stable and unstable manifolds of $L_-$ where {\em local} means that the orbits converge to $L_-$ while remaining in $\{z>0\}$.  It follows from proposition~\ref{prop_monotonecollisionsyzygy} that 
$z(s)$ is monotonically increasing to $1$ along orbits in $W^{s+}(L_-)\intersection\{r=0\}$.  A similar argument with time reversed shows that $z(s)$ monotonically decreases from $1$ along orbits in  $W^{u+}(L_-)$. These last orbits
lie   entirely in $\{r=0\}$) by proposition \ref{prop_linearization}.

The unstable manifold $W^{u+}(L_-)$ is two-dimensional and its  projection to the $(x,y)$ plane is a local diffeomorphism near $L_-$.  Let $D^u$ be a small disk around $L_-$ in $W^{u+}(L_-)$.  It follows from proposition~\ref{prop_monotonecollisionsyzygy}  that all of the points in $D^u\setminus L_-$ can be followed forward to meet $\tilde C_h$ transversely.  Moreover the monotonic decrease of $z(s)$ implies that if this flow-induced mapping is composed with the projection to $C_h$ and then to the unit circle, the resulting map from the punctured disk to the circle will have degree one.

Let $\g^u = \partial D^u$ be the boundary curve of such an unstable disk.  Since the unstable manifold is contained in the collision manifold, the first-syzygy map takes $\g^u$ into $\tilde C_h\intersection \{r=0\}$.  Transversality implies that the first syzygy map is defined and continuous near $\g^u$.  Hence, given any $\e>0$  there is a neighborhood $U$ of $\g_u$ such that initial conditions in $U$ have their first syzgies with $r<\e$.  Standard analysis of the hyperbolic restpoint $L_-$ then shows that any orbit passing sufficiently close to $L_-$ will exit a neighborhood of $L_-$ through the neigborhood $U$ of $\g^u$ and therefore reaches its first syzygy with $r<\e$. We have established the continuity of the syzygy map at $p_0$. 

Now consider a small disk, $D$, around $p_0$ in $\partial Q_h ^+$.    It has already been shown that every point in $D\setminus p_0$ 
can be followed forward under the flow to meet $\tilde C_h$ transversely with $z(s)$ decreasing monotonically.   
If $D$ is sufficiently small it will follow the Lagrange homothetic orbit close enough to $L_-$ for the argument of the previous 
paragraph to apply.  In other words the syzygy map takes $D\setminus p_0$ into $\{r<\e \}$ which implies continuity of the extension 
at $p_0$.   As before, the monotonic decrease of $z(s)$ implies that the map from the punctured disk to the unit circle has degree 
one. Hence the map with triple collision collapsed to the origin has local degree one.
\end{proof}

This completes the proof of the continuity statements in Theorem~\ref{thm:cty} and also shows that the extended map covers a 
neighborhood  of the triple collision point.  See figure~\ref{fig_syzygyimage} for a picture illustrating this theorem in the equal mass case.

\section{Variational methods: Existence and regularity of minimizers}
\label{SecVariational}
In this section, we  prove theorems \ref{thm:variational}, \ref{thm:equal_masses} and lemmas~\ref{lemma:JM_Marchal}
by applying the direct method of the Calculus of Variations to the  Jacobi-Maupertuis [JM] action:
\begin{equation}  \label{Jac-Maup}
A_{JM}(\gamma)=\int_{t_0}^{t_1} \sqrt{K_0} \sqrt{2(U-h)}\, dt.
\end{equation}
Here $\gamma : [t_0,t_1]\rightarrow Q_h$ is a curve.   Curves which minimize $A_{JM}$ within some
class of curves will be called ``JM minimizers''. 

The integrand of the functional  $A_{JM}$
is  homogeneous of degree one in velocities and so
is independent of  how $\gamma$ is parameterized. 
  The natural domain of definition of the functional is the space of  rectifiable curves in $Q_h$. 
We recall some notions about Fr\'echet rectifiable curves.  See \cite{Ewing} for more details.
Kinetic energy $K_0$  induces a complete \Ri metric, denoted $2K_0$ on the manifold
with boundary  $Q_h$.  We denote by $d_0$ the associated Riemannian distance. The Fr\'echet distance between two continuous curves 
$\gamma : [t_0,t_1]\rightarrow Q_h$ and $\gamma^\prime :[t_0^\prime,t_1^\prime]\rightarrow Q_h$ is defined to be 
$$
\rho(\gamma,\gamma^\prime)=\inf\limits_h \sup\limits_{t\in [t_0,t_1]} d_0(\gamma(t),\gamma^\prime(h(t))),
$$
where the infimum is taken over all orientation preserving homeomorphisms $h:[t_0,t_1]\rightarrow [t_0^\prime,t_1^\prime]$.  
We   consider two curves   equivalent if the Fr\'echet distance between them  is   zero. 
An equivalence class of such curves  can be seen as an unparametrized curve.
The  Fr\'echet distance is a complete metric on the set of 
unparametrized curves. 
The length ${\mathcal L}(\gamma)$ 
of a parametrized curve $\gamma : [t_0,t_1]\rightarrow Q_h$ can be defined to be the supremum of the lengths 
of broken geodesics associated to subdivisions of $[t_0,t_1]$. This length  equals the usual
$2K_0$ length when the curve is $C^1$.  Equivalent curves have 
equal lengths. A curve is called  rectifiable if its length is finite.   
 For a generic parametrization of a rectifiable curve, the integral (\ref{Jac-Maup})
must be  interpeted as a Weierstrass integral (see \cite{Ewing}).  This  integral  is independent
of parameterization.  

We will use the following two facts in  our proof of Theorem \ref{thm:variational}.
If $K\subset Q_h$ is compact, 
then the set of curves in $Q_h$ which intersect  $K$ and have length
bounded by a given constant forms a  compact set in the Fr\'echet topology. (This fact
is a theorem attributed to Hilbert.) The second fact asserts that
both the length functional ${\mathcal L}$ and the JM action functional $A_{JM}$ are   lower semicontinuous in the Fr\'echet topology.  (Again, see \cite{Ewing}).

Since  the Jacobi-Maupertuis action degenerates on the Hill boundary  
we will need to construct a tubular neighborhood of the Hill boundary within which we can characterize $JM$-minimizers to the Hill boundary. Our construction follows \cite{Seifert}.
\begin{proposition} \label{seifert-nbd}
There exists $\epsilon>0$, a neighborhood $S^{\epsilon}_h$ of $\partial Q_h$ (in $Q_h$), 
an analytic diffeomorphism $\Phi: \partial Q_h\times [0,\epsilon]\rightarrow S^{\epsilon}_h$ 
(satisfying $\Phi(\partial Q_h,{0})=\partial Q_h$)
and a strictly positive constant $M$, such that if $\delta\in (0,\epsilon)$,
and $x\in \partial Q_h$, the curve $y\mapsto \Phi(x,y),\ y\in [0,\delta]$ 
is a reparametrization of an arc of the brake solution 
starting from $x$, and its length is smaller or equal to $M\delta$. 
If $q=\Phi(x,\delta)$ then for every rectifiable curve $\gamma$ in $Q_h$ joining $q$ to $\partial Q_h$ 
we have 
$$
A_{JM}(\gamma)\ge  \delta^{3/2},
$$ 
with equality if and only if $\gamma$ is obtained by pasting $\left(\Phi(x,y)\right)_{y\in[0,\delta]}$ to
 any arc contained in $\partial Q_h$. 
Moreover, there exists a strictly positive constant $\alpha$, independent from $\delta$, such that $U\ge h+\alpha\delta$ 
on $Q_h\setminus S_h^{\delta}$, where we term $S_h^{\delta}=\Phi(\partial Q_h,[0,\delta])$.   
\end{proposition}
We postpone the proof of this Proposition, and instead  use it now to prove  Theorem \ref{thm:variational} and Lemma \ref{lemma:JM_Marchal}.   We will say that $S^{\epsilon}_h$ is a Seifert tubular neighborhood of the Hill boundary, and that 
$\Phi(\partial Q_h,\epsilon)$ is the inner boundary of $S^{\epsilon}_h$.
We prove now the first part of Theorem \ref{thm:variational}.
\begin{proposition} \label{exist-minim-hill}
Given a point $q_0$ in the interior of the Hill region $Q_h$, there exist a $JM$ action minimizer among rectifiable curves 
joining $q_0$ to the Hill boundary $\partial Q_h$.
\end{proposition}
\begin{proof}
Let $S^{\epsilon}_h$ be the Seifert tubular neighborhood given by Proposition \ref{seifert-nbd}. 
If $q_0\in S^{\epsilon}_h$ then  
the unique brake orbit joining $q_0$ to the Hill boundary is a $JM$-minimizer. 

If $q_0\notin S^{\epsilon}_h$, let $\gamma_n$ be a $JM$ minimizing sequence  of rectifiable curve joining 
$q_0$ to $\partial Q_h$.
Without loss of generality we can assume all these curves are   Lipschitz and   defined on the unit interval $[0,1]$. 
Let $t_n\in (0,1)$ be the first time 
such that $\gamma_n(t)$ touches the inner boundary $\Phi(\partial Q_h,\epsilon)$ of $S^{\epsilon}_h$. Let us replace $\gamma\left|_{[t_n,1]}\right.$ by the unique brake 
solution joining $\gamma_n(t_n)$ to $\partial Q_h$. By Proposition \ref{seifert-nbd}, this modification decreases the JM action of 
$\gamma_n$, so our sequence is still a minimizing one. Let $C$ be an upper bound for the numbers $A_{JM}(\gamma_n)$.
Applying   Proposition \ref{seifert-nbd}  again we get  
\begin{equation} \label{estim-JM-minim}
C\ge A_{JM}(\gamma_n) \ge  \epsilon^{3/2}+{\mathcal L}(\gamma_n\left|_{[0,t_n]}\right.)\sqrt{\alpha\epsilon},
\end{equation}   
therefore 
\begin{equation}  \label{estim-length-JM-min}
{\mathcal L}(\gamma_n)\le \frac{C-\epsilon^{3/2}}{\sqrt{\alpha\epsilon}} +M\epsilon,
\end{equation}
for all $n$, proving that the  lengths ${\mathcal L}(\gamma_n)$ are bounded. Since $\gamma_n(0)=q_0$ for all $n$, 
 Hilbert's theorem discussed above applies:  
the sequence of curves $ \gamma_n$ is relatively compact in the Fr\'echet topology. 
Therefore a subsequence of this sequence converges to  a rectifiable curve $\gamma$ in $Q_h$.
This 
$\gamma$ is a JM minimizer by the  lower semicontinuity of the JM action. 
\end{proof}
Let  $d_{JM}(q_0,q_1)$ denote the infimum of the $JM$-action
among rectifiable curves in $Q_h$ joining $q_0$ to $q_1$, and let
 $d_{JM}(q_0,\partial Q_h)$  denote 
the minimum of the $JM$-action among rectifiable curves in $Q_h$ joining $q_0$  
to the Hill boundary. \vspace{1mm}\\ \noindent

{\bf Lemma \ref{lemma:JM_Marchal}, the JM Marchal lemma, proof.}
\begin{proof}
First  we establish  the existence of a minimizer. 

If $d_{JM}(q_0,q_1)\ge d_{JM}(q_0,\partial Q_h)+d_{JM}(q_1,\partial Q_h)$, 
then   take a curve realizing the minimum in $d_{JM}(q_0,\partial Q_h)$, another  curve realizing the minimum in 
$d_{JM}(q_1,\partial Q_h)$ and join these two  curves  by any 
curve lying on the Hill boundary and connecting the endpoints of these two. In this way we get a   rectifiable curve
$\gamma$  joining 
$q_0$ to $q_1$ and (possibly) spending some time on the Hill boundary.  
Since the $JM$ action of any curve on the Hill boundary is     zero,  we have that  
$A_{JM}(\gamma)= d_{JM}(q_0,\partial Q_h)+d_{JM}(q_1,\partial Q_h)$. Hence $\gamma$ is our minimizer.
(As a bonus we have shown that  $d_{JM}(q_0,q_1)=d_{JM}(q_0,\partial Q_h)+d_{JM}(q_1,\partial Q_h)$
in this case.)

If $d_{JM}(q_0,q_1)< d_{JM}(q_0,\partial Q_h)+d_{JM}(q_1,\partial Q_h)$, let $\gamma_n$ be a $JM$-minimizing
sequence of rectifiable curves in $Q_h$ joining $q_0$ to $q_1$. We show  now that there exists $\epsilon>0$ such that
for  $n$ is sufficiently great,  the $\gamma_n$ do not intersect the Seifert tubular neighborhood $S^\epsilon_h$. 
Assume, for the sake of contradiction that there in fact exists a decreasing sequence of strictly positive real numbers 
$\epsilon_n$ such that $\epsilon_n\rightarrow 0$ as $n\rightarrow +\infty$ 
and a subsequence of the minimizing sequence, still denoted $(\gamma_n:[0,1]\rightarrow Q_h)_n$,  and a sequence
of times $t_n$ 
such that for every $n$, 
$\gamma_n (t_n)$ is in  $S^{\epsilon_n}_h$.  Let $\tilde{\gamma}_n$ be the curve 
defined in the following way.  Follow $\gamma_n$ from $q_0$ to $\gamma_n(t_n)$. Join $\gamma_n(t_n)$ to 
$\partial Q_h$ by the unique brake solution in $S^{\epsilon_n}_h$ with one end point  $\gamma_n(t_n)$.
Return along the  same brake solution, to $\gamma_n (t_n)$. 
Then continue along   $\gamma_n$ from $\gamma_n(t_n)$ to $q_1$. The curve 
$\tilde{\gamma}_n$ joins $q_0$ to $q_1$ is rectifiable  and touches $\partial Q_h$, hence 
$A_{JM}(\tilde{\gamma}_n)\ge d_{JM}(q_0,\partial Q_h)+d_{JM}(q_1,\partial Q_h)$. By proposition  \ref{seifert-nbd} we have
$$
d_{JM}(q_0,\partial Q_h)+d_{JM}(q_1,\partial Q_h)-2\epsilon_n^{3/2}\le  
A_{JM}(\tilde{\gamma}_n)-2\epsilon_n^{3/2}\le A_{JM}(\gamma_n).
$$     
Taking the limit $n\rightarrow +\infty$ we get 
$d_{JM}(q_0,\partial Q_h)+d_{JM}(q_1,\partial Q_h)\le d_{JM}(q_0,q_1)$,   a contradiction.
We have established our desired   $\epsilon>0$ with  the property that none of the  curves 
$\gamma_n$ ($n$ sufficiently large)  intersect  $S^\epsilon_h$.  

Let $C>0$ be an upper bound of $A_{JM}(\gamma_n)$ and 
$\alpha=\alpha(\epsilon)$ the positive constant given by Proposition  \ref{seifert-nbd}. 
A computation similar to (\ref{estim-JM-minim}) and (\ref{estim-length-JM-min}) gives 
$$
{\mathcal L}(\gamma_n)\le \frac{C}{\sqrt{\alpha\epsilon}}.
$$ 
Hilbert's Theorem again yields  a curve $\gamma$ in $Q_h$ joining $q_0$ to $q_1$ which
the  a subsequence of the  $\gamma_n$ converges to  in the Fr\'echet topology.   By lower semicontinuity of 
$A_{JM}$, the curve $\gamma$ is a $JM$-minimizer. 

Finally, we show that   any minimizer 
$\gamma : [t_0,t_1]\rightarrow Q_h$ is collision-free and  that every subarc of $\gamma$,
upon being  reparametrized,  is a true  solution 
to Newton's equation.  
The classical Lagrangian action on the reduced space $Q$ is defined by 
$$
{\mathcal A}_{red}(\sigma)=\int_{t_0}^{t_1} L_{red}(\sigma(t),\dot{\sigma}(t)) dt
$$
where $\sigma : [t_0,t_1]\rightarrow Q$ is any absolutely continuous path and the reduced 
Lagrangian $L_{red}$ is defined in (\ref{lagr_red}). If $t\in [t_0,t_1]$ is such that $\gamma(t)\in \partial Q_h$, obviously 
$\gamma(t)$ is not a collision, hence it suffices  to prove the statement for  subarcs  
$\gamma\left|_{[a,b]}\right.$, 
 $[a,b]\subset [t_0,t_1]$ of $\gamma$ which never touch the Hill boundary.  Let $[a,b]$ be such an interval. 
By the  inequality  $A^2 + B^2 \ge 2AB$
applied to the factors  $A = \sqrt{K_0},  B = \sqrt{ U-h}$ of the integrand for $A_{JM}$
we see that  
$$
\sqrt{2}A_{JM}\left(\gamma\left|_{[a,b]}\right.\right)\le {\mathcal A}_{red}(\gamma\left|_{[a,b]}\right.)-h(b-a)
$$ 
with equality if and only if the energy $K_0-U$ evaluated on $\gamma$ is equal to $-h$ for almost all $t\in [a,b]$. 
We may assume,  without loss of generality, that  $\gamma$ is parametrized by its $2K_0$-arclength parameter  $s$ 
so that in the integrand  $K_0(s)=1/2$ for almost every $s\in [a,b]$. Since $\gamma$ is a $JM$-minimizer, 
we know that $\gamma(a)\neq \gamma(b)$, that the action $A_{JM}(\gamma\left|_{[a,b]}\right.)$ is finite, and
that  the set of 
$s\in [a,b]$ such that $\gamma(s)$ is a collision is a closed set of   zero Lebesgue measure. 
Let $\sigma : [\alpha,\beta]\rightarrow Q_h$ be the 
reparametrization of $\gamma\left|_{[a,b]} \right.$ defined by $\sigma(t)=\gamma(s(t))$, 
where $s=s(t)$ is the inverse function of
$$
t=t(s)=\int_{a}^s \frac{d v}{\sqrt{2(U(\gamma(v))-h)}}, \qquad s\in [a,b].
$$    
The function $t=t(s)$ is ${\mathcal C}^1$, strictly increasing and $\frac{dt}{ds}(s)>0$ on an open set 
with full measure, hence the inverse function $s=s(t)$ is absolutely continuous, and since $s\mapsto \gamma(s)$ 
is Lipschitz, 
$t\mapsto \sigma(t)=\gamma(s(t))$ is also absolutely continuous. 
A simple computation shows that the energy function $K_0-U$ 
of $\sigma$ is  the constant  $-h$ for almost every $t\in [\alpha,\beta]$. Hence  
$$
\sqrt{2}A_{JM}\left(\gamma\left|_{[a,b]}\right.\right)=\sqrt{2}A_{JM}(\sigma)
={\mathcal A}_{red}(\sigma)-h(\beta-\alpha).
$$ 

But $\gamma\left|_{[a,b]}\right.$ is a minimizer of $A_{JM}$! So this equation
says that  $\sigma$  minimizes  ${\mathcal A}_{red}$ among all 
absolutely continuous paths in $Q_h$ joining $\sigma(\alpha)$ to $\sigma(\beta)$ in time $\beta-\alpha$!
Let $\xi :[\alpha,\beta]\rightarrow \C^2$ be  a continuous zero angular momentum lift of
$\sigma$ to the nonreduced space $\C^2$. (There are a circle's worth of such paths.)
 The path $\xi$ is a local minimizer of 
the nonreduced L agrangian action $\int_{\alpha}^{\beta} L dt$ among absolutely continuous paths joining $\xi(\alpha)$ 
to $\xi(\beta)$ in time $\beta-\alpha$. By Marchal's Theorem (see \cite{Marchal} and \cite{Chenciner-2002}), 
$\xi(t)$ is collision-free for $t\in(\alpha,\beta)$ and is a classical solution of Newton's   equations. The
quotient  path $\sigma(t)$ is 
thus a collision-free solution 
of the reduced equations with energy $-h$. 
\end{proof} 
%(for   Lemma \ref{lemma:JM_Marchal}.)

\vspace{1.5mm}
%\\ \noindent

To  prove theorem \ref{thm:variational} we will need certain  
monotonicity  properties for the shape potential $V$. Recall that 
the shape sphere $\CP^1$ endowed with its Fubini-Study metric (\ref{eq_FSnorm}) is isometric to $(S^2, ds^2/4)$, where $ds^2$ is the 
standard metric on the unit sphere $S^2$.  The isometry is  
$$
St : \CP^1\rightarrow S^2\subset \R^3,\qquad [\xi_1,\xi_2]\mapsto \frac{1}{r^2}(w_1,w_2,w_3)
$$
where 
\begin{equation} \label{hopf}
\begin{array}{rl}
w_1&=\mu_1|\xi_1|^2-\mu_2|\xi_2|^2 \\
w_2+\imath w_3&=2\sqrt{\mu_1 \mu_2}\xi_2 \overline{\xi_1} \\
r^2&=\mu_1|\xi_1|^2+\mu_2 |\xi_2|^2.
\end{array}
\end{equation}
Collinear configurations correspond   to the great circle $w_3=0$ which  we call the collinear equator. The
Northern hemisphere 
($w_3\ge 0$) corresponds to positively oriented triangles. The three binary collision points
 $b_{12}, b_{23}, b_{31}$ lie on the equator and split it into three arcs
 denoted $C_1, C_2, C_3$ with the endpoints of  $C_j$ being $b_{ij}$ and $b_{jk}$.
 Thus $C_j$ consists of  collinear 
shapes in which    $q_j$ lies between $q_i$ and $q_k$. 

Introduce  standard
  spherical polar coordinates $(\phi,\theta)$ on the shape 
sphere so that  one of the binary collision points $b$, say $b=b_{12}$ is  the origin ($\phi = 0$)  and so that   the Northern  hemisphere is  defined by $0 \le \theta \le \pi$.  The distance of a point from $b$ relative to the standard metric is
$\phi$ and is a  function of $r_{12}$ alone \cite{Mont}:  $\phi = \phi(r_{12});  r_{12} = r_{12} (\phi)$ on the shape sphere.
Thus  setting $\phi = const.$ is the same as setting $r_{12} = const.$, and  defines a geometric  circle 
$\sigma = \sigma(\phi)$ on the shape sphere.
 The collinear equator is the union of the curves    $\theta=0$ and $\theta = \pi$.  The arcs $C_1$ and $C_2$
 are adjacent to $b$ and we can choose coordinates so that    $C_1$  is contained in the half equator $\theta=0$ 
 while  $C_2$  contained in the other half equator
$\theta=\pi$.   In the case of equal masses, $\theta=\pi/2$  
defines the isosceles circle through the collision point $b=b_{12}$, lying in the Northern hemisphere.  
%The kinetic energy metric on the shape sphere has the form  ${1 \over 4}(d \phi^2 + sin^2 (\phi) d \theta^2)$.)

\begin{lemma} \label{mutual-distances}
Let $\sigma$ be any of the  half-circles $\phi =const$  centered at $b_{12}$ and lying in the Northern Hemisphere
of the shape sphere. Then $\sigma$ can be parameterized so that  $V|_{\sigma}$, the restriction of $V$ to $\sigma$,  is strictly convex. Consequently $V|_{\sigma}$   has a unique minimum.    If that minimum is an endpoint  of
$\sigma$  then
it  lies on the `distant' collinear arc $C_3$. Otherwise, the minimizer is an interior minimum and  coincides
with the intersection of $\sigma$ with the isosceles arc $r_{13} = r_{23}$.  
\end{lemma}

{\bf Remark 1:}  In all cases, an endpoint
of $\sigma$ which lies on    arc    $C_1$ or  $C_2$
  is a strict local maximum of   the restricted $V$. 
  
  {\bf Remark 2:}  The minimum of $V|_{\sigma}$ is an endpoint of $\sigma$  
  if and only if  $\sigma$ does not intersect the isosceles arc at  an interior point.

  {\bf Remark 3:}   When  $m_1=m_2$ the isosceles circle in the Northern hemispher (defined by $\theta=\pi/2$)
  bisects all the circles $\sigma$ and so the minimum of   $V|_{\sigma}$
  always occurs at this bisection point  and thus is always interior.

\begin{proof} Use    squared length coordinates $s_i = r_{jk}^2$  introduced by Lagrange
 and re-introduced to us by Albouy  \cite{Albouy-Chenciner}. 
 The $s_i$ are subject to the constraints
 $$ s_i \ge 0$$
 and
 \begin{equation}
 \label{eq:Heron}
 2 s_1 s_2 + 2 s_2 s_3 +2 s_3 s_1  -  s_1 ^2  - s_2 ^2  - s_3 ^2  \ge 0
 \end{equation}
 which together define a convex cone in  the 3-space with coordinates $s_i$. 
 The points of this cone   faithfully represent
 congruence classes of triangles,  where two triangles related by reflection are now considered
 equivalent.    The origin of the 
 second inequality is   Heron's formula
 for the signed area $\Delta$ of a triangle:  
 $16 \Delta^2 = 2 s_1 s_2 + 2 s_2 s_3 +2 s_3 s_1  -  s_1 ^2  - s_2 ^2  - s_3 ^2 $.
 
  In squared length  coordinates
 \begin{equation}
 \label{eq:C1}r^2 = \mu_1 s_1 + \mu_2 s_2 +  \mu_3 s_3 
 \end{equation}
 \begin{equation}\label{eq:C2}
U= m\left({\mu_1 \over \sqrt{s_1}} + {\mu_2 \over \sqrt{s_2}} + {\mu_3 \over \sqrt{s_3}}\right).
 \end{equation}
A  half-circle $\sigma$  as in the lemma
is defined by the two linear  constraints $r^2 = 1$,  $s_3 = r_{12}^2=c$
 in the $s_i$.  Consequently, such a  half-circle is represented by  a closed  interval whose endpoints
correspond to the   collinear triangles at which the  inequality (\ref{eq:Heron}) becomes  an equality.
  $U$ is a strictly convex function of $(s_1, s_2, s_3)$ in the region $s_i \ge 0$
and   $V$ on the half-circle becomes    $U(s_1, s_2, s_3)$ restricted to this closed interval.
Now a strictly convex function restricted to a convex set remains strictly convex, so relative to any
affine coordinate on the interval which represents $\sigma$ we see that $V$ is strictly convex, as claimed.

We can parameterize  the half-circle $\sigma$ by
$$
\left\{
\begin{array}{rl}
s_1(u)&=\frac{1}{\mu_1}(1-\mu_2 u -\mu_3 c) \\
s_2(u)&=u \\
s_3(u)&=c
\end{array}
\right.
$$
where $u\in [u_-,u_+]$, with    $u=u_-$ corresponding to $\theta=0$, and $u=u_+$ corresponding  to $\theta=\pi$.  
A simple computation gives
\begin{equation} \label{derivative-U-mutual}
\frac{d}{du} V(s_1(u),s_2(u),s_3(u))=\frac{m\mu_2}{2}\left(s_1(u)^{-3/2}-s_2(u)^{-3/2}\right),
\end{equation}
which proves that $V|_{\sigma}$ has an interior  critical point   if and only if $\sigma$  cross the isosceles curve 
$s_1=s_2$. Convexity implies this crossing point is the unique minimum. 

If the starting point $u = u_-$
of $\sigma$ is on $C_1$, then at this point we have
1 between 2 and 3 so that $r_{12} + r_{13} = r_{23}$.
Squaring, we see that $s_2 > s_1$ from which it follows that  the derivative (\ref{derivative-U-mutual}) is negative at $u=u_-$,
and showing that this point is a local maximum for $V_{\sigma}$.
A similar  argument shows that  if an endpoint    of $\sigma$ lies on  $C_3$ then that
endpoint is also a local maximum of $V_{\sigma}$. (If the endpoint of $\sigma$,  corresponding to $u =u_+$
lies on $C_3$
then the derivative (\ref{derivative-U-mutual}) is positive.)
\end{proof}

We prove now continuity properties of the $JM$-distance $d_{JM}(q_0,q_1)$ introduced just before the proof of Lemma 
\ref{lemma:JM_Marchal}. If $q_0$ and $q_1$ are collinear we denote by $d^{c}_{JM}(q_0,q_1)$ the infimum of the $JM$-action
among collinear rectifiable curves in $Q_h$ joining $q_0$ to $q_1$. In a similar way, $d^{c}_{JM}(q_0,\partial Q_h)$ denotes
the minimum of the $JM$-action among collinear rectifiable curves in $Q_h$ joining $q_0$ to the Hill boundary.
\begin{lemma} \label{continuity-JM-dist}
$d_{JM}:Q_h\times Q_h\rightarrow \R$ and $d^{c}_{JM}:C_h\times C_h\rightarrow \R$ are continuous functions.
\end{lemma}
\begin{proof}
By  the triangle inequality it  suffices to prove that for every $q_0\in Q_h$ or $c_0\in C_h$  the functions
$q\rightarrow d_{JM}(q_0,q)$ or $q\rightarrow d^{c}_{JM}(c_0,q)$) are continuous in $q_0$ or in $c_0$. 
We give the proof only for $q\rightarrow d_{JM}(q_0,q)$.  The proof for  $q\rightarrow d^c_{JM}(q_0,q)$  is similar.

If $q_0$ is a non-collision shape, not lying on the Hill boundary the $d_{JM}(q_0,q)$ is a regular Riemannian distance in a neighborhood of $q_0$, so 
the result is classical.  If $q_0$ lies on the Hill boundary, the function
is a pseudo-distance: $d(q_0, q) = 0$ with $q_0 \ne q$ can happen),
but the classical result goes through with no changes.
So assume now that $q_0$ is a collision point. We consider two case. \vspace{0.5mm} \\ \noindent
{\it 1st case : $q_0$ is the total collision. } Let us choose any non-collision 
collinear configuration on the shape sphere $\hat{s}\in S^2$. If $q=(r,s)\in Q_h$, 
and $\gamma$ is the path obtained by joining the total collision to 
$\hat{q}=(r,\hat{s})$ along the ray through $\hat{q}$, and then joining $\hat{q}$ to $q$ 
along a segment 
of geodesic of $\{r\}\times S^2$, a simple computation shows the existence of a constant $C>0$, 
independent of $q$, such that 
$A_{JM}(\gamma)\le C\sqrt{r}+{\mathcal O}(r\sqrt{r})$, hence $d_{JM}(q_0,q)\rightarrow 0$ as $q\rightarrow q_0$.
\vspace{2mm}\\ \noindent
{\it 2nd case : $q_0$ is a partial collision.} Assume that $q_0$ is on the collision ray $r_{12}=0$. Introduce
 spherical coordinates 
$(\phi,\theta)$ as before   centered at the binary collision $b_{12}$, 
so that $\phi=0$ corresponds to $b_{12}$, and $\theta=0$ or $\pi$ corresponds to the collinear circle. If 
$r_0$ is the radial coordinate of $q_0$ and $q_1\in Q_h$ is a point with radial coordinate $r_1$ and 
spherical coordinates $(\phi_1,\theta_1)$, let us choose a $\hat{\phi}\in (0,\pi)$ and define 
the following three paths in spherical coordinates 
$$
\begin{array}{rl}
\gamma_1 :\qquad &r(t)=r_0,\quad \phi(t)=t\hat{\phi},\quad \theta(t)=\theta_1,\quad t\in [0,1] \\
\gamma_2 :\qquad &r(t)=(1-t)r_0+t r_1, \quad \phi(t)=\hat{\phi},\quad \theta(t)=\theta_1,\quad t\in [0,1] \\
\gamma_3 :\qquad &r(t)=r_1,\quad \phi(t)=(1-t)\hat{\phi}+t \phi_1,\quad \theta(t)=\theta_1,\quad t\in [0,1].
\end{array}
$$  
Let $\gamma$ be the path obtained by pasting together $\gamma_1$, $\gamma_2$ and $\gamma_3$.
This   $\gamma$ joins $q_0$ to $q_1$. 
By (\ref{eq_rijJac}) and (\ref{hopf}) we get 
$$
V(\phi,\theta)=\frac{c}{\phi}+o(\frac{1}{\phi}),
$$
where $c>0$ depends only on the masses. 
If $r_0/2\le r_1\le 3r_0/2$, a simple computation shows that  
$$
\begin{array}{rl}
A_{JM}(\gamma_1)&=\sqrt{2c r_0 \hat{\phi}} +{\mathcal O}(\hat{\phi}^{3/2}), \\
A_{JM}(\gamma_2)&\le\frac{2 |r_1-r_0|}{\sqrt{r_0}}\left( \sqrt{\frac{c}{\hat{\phi}}}+o(1/\hat{\phi})\right), \\
A_{JM}(\gamma_3)&\le  \sqrt{2c r_1 \hat{\phi}} +{\mathcal O}(\hat{\phi}^{3/2}), 
\end{array}
$$
hence chosing $\hat{\phi}=|r_1-r_0|$ we get 
$$
d_{JM}(q_0,q_1)\le A_{JM}(\gamma)\le C\sqrt{|r_1-r_0|}+o(\sqrt{|r_1-r_0|}),
$$
where $C>0$ depends only on $r_0$, hence $d_{JM}(q_0,q_1)\rightarrow 0$ as $q_1\rightarrow q_0$.  
\end{proof}
{\bf Proof of Theorem \ref{thm:variational}. }
Let $q_0\in C_h$ be a collinear configuration. By Proposition \ref{exist-minim-hill} there exists a $JM$-minimizer 
$\gamma :[0,T]\rightarrow Q_h$ among paths starting in $q_0$ and ending on the Hill boundary $\partial Q_h$. 
The JM-action of a path on the Hill boundary is zero, therefore if $t_0$ is the first time 
where $\gamma(t_0)\in \partial Q_h$, then $\gamma(t)\in \partial Q_h$ for all $t\in [t_0,T]$. We
may cut out the    arc $\gamma\left|_{[t_0,T]}\right.$ from $\gamma$, or equivalently,
 assume that $\gamma(t)\notin \partial Q_h$ for all $t < T$.    
For every $t\in (0,T)$ the restriction $\gamma\left|_{[0,t]}\right.$ is a fixed end point $JM$-minimizer.  By Lemma \ref{lemma:JM_Marchal}, $\gamma$ is a classical brake solution, and  is collision-free for $t\in (0,T]$.

Let us prove now that if $\gamma$ is not collinear at all times, it has a unique syzygy : the point $q_0$.
Indeed, assume for the sake of contradiction that $\gamma(\tau)$ is collinear for some $0<\tau\le T$. 
Recall that the set of initial conditions tangent to the collinear submanifold yield collinear curves.
Collinear brake points (with zero velocity)  are such initial conditions.  Consequently,   $\tau=T$ is impossible,
for if   $\gamma(T) \in \partial Q_h$ were collinear, all of   $\gamma$ would be collinear. So $\tau < T$.
By the same reasoning, $\gamma$ cannot be tangent to the collinear  manifold at $\tau < T$. 
Consider the path obtained by keeping $\gamma\left|_{[0,\tau]}\right.$ as it is and reflecting $\gamma\left|_{[\tau,T]}\right.$ with respect to the syzygy plane.
This new curve has   the same $JM$-action as $\gamma$, and  joins $q_0$ to $\partial Q_h$, so it is a $JM$-minimizer.  But this  new path is not differentiable at $t=\tau$,
contradicting the fact that minimizers are solutions, hence smooth. 
We can now assume without loss of generality that $\gamma$ is all the time in 
the half-space of positive oriented shape.

If $q_0$ is the total collision we will now  show that $\gamma$ is the Lagrange
homothety solution. Write $\gamma(t)=(r(t),s(t))\in Q_h\subset \R_+\times S^2$. Denote by 
$s_L\in S^2$ the Lagrange shape in the Northern hemisphere. We want to show
$s(t) =  s_L$ for all $t$.   It is well known that the Lagrange shape is the unique minimum 
of the shape potential $V$.   Let $r_L$ be
the number such that ${1 \over r_L} V(s_L) = h$.  Suppose that $s(T) \ne s_L$.
Then $r(T) > r_L$ since $V(s(T)) = r(T) h > V(s_L) = r_L h$.  Since $r(0) = 0$ there will be a
smallest number   $t= \tau_*$  in the interval $[0,T]$
such that $r(t) = r_L$.  Then $r(t)< r_L$ for $t< \tau_*$.  The path $\gamma_L(t)=(r(t),s_L)$, $0\le t \le \tau_*$ lies in $Q_h$,  joins $q_0$ to $\partial Q_h$
and satisfies $U(\gamma_L (t)) \le U(\gamma(t)$ on this interval (with equality if and only if $s(t) = s_L$).
The  kinetic energy of this new path is pointwise  the same or less than that of the old path,  over their common domain.    Consequently 
$A_{JM}(\gamma_L)\le   A_{JM}(\gamma)$, with equality if and only if $s(t)=s_L$ for all $t\in [0,T]$ 
(and $\tau_* = T$)  as desired.   

Assume now $q_0$ is a double collision, say on the collision ray $r_{12}=0$.  We  prove that our minimizer  $\gamma$ is not collinear. 
Assume, for the sake of contradiction that $\gamma(t)$ is collinear at all time $t\in [0,T]$ . 
Since $\gamma$ is collision free except for $q_0$,
  $\gamma$ is contained in an angular sector of the syzygy plane between the collision   ray $r_{12}=0$ and one of the 
  two other collision rays.  Say it lies in the sector $C_1$  between $r_{12} = 0$ and   $r_{13}=0$.
Introduce  spherical coordinates $(\phi,\theta)$ on the shape sphere as in Lemma 
\ref{mutual-distances}. Then $\gamma(t)$ has coordinates $(r(t),\phi(t),\theta(t)) = (r(t), \phi(t) , 0)$,
where we have   arranged the   coordinates so that the sector within which $\gamma$ lies, sector $C_1$,
is given by   $\theta=0$.  To re-iterate,  at every moment $\gamma$'s spherical projection lies in the arc $C_1$
on the collinear equator. 
By Lemma \ref{mutual-distances} there exists 
$\delta\in (0,\pi)$ such that $V(\phi(t), \delta )< V(\phi(t),0)$ 
for all $t\in [0,T]$. 
Let $\tau_{\delta}$ be the first time in our interval such that   $V(\phi(t),\delta)/r(t)=h$. (There is such a time
by an argument similar to the one above when $q_0$ was triple collision.) Define the new path 
$\gamma_{\delta}(t)=(r(t),\phi(t),\delta)$, $t\in [0,\tau_\delta]$. Geometrically, $\gamma_\delta$ is obtained by rotating  
$\gamma(t)$ by an angle $\delta$ around the collision ray $r_{12}=0$ and keeping only that part of it which    stays inside $Q_h$. 
Such a rotation is an isometry for the metric (\ref{kinetic}), but strictly decreases $U$. 
Therefore $A_{JM}(\gamma_\delta)<A_{JM}(\gamma)$. But $\gamma$
is a minimizer, so we have  a contradiction. 

Using the  notations of Lemma   \ref{continuity-JM-dist} we have proved that 
if $q_0$ is a collision,  we have $d_{JM}(q_0,\partial Q_h)<d^c_{JM}(q_0,\partial Q_h)$.

Assume now that $\overline{q_0}\in C_h$ is a collision configuration.  We show the existence of a neighborhood 
${\mathcal O}(\overline{q_0})$ of $\overline{q_0}$ in $C_h$ such that for every $q_0\in {\mathcal O}(\overline{q_0})$, 
no $JM$-minimizer from $q_0$ to $\partial Q_h$ is  collinear. 
Indeed, let us assume, for the sake of contradiction, the existence of a sequence $(q_n)_{n\in \N}$ in $C_h$, converging to $\overline{q_0}$,
with corresponding collinear minimizers.  Then  
$d_{JM}(q_n,\partial Q_h)=d^c_{JM}(q_n,\partial Q_h)$ for all $n$. By the  triangle  inequality we have 
$$
\begin{array}{rl}
|d_{JM}(q_n,\partial Q_h)-d_{JM}(\overline{q_0},\partial Q_h)|&\le d_{JM}(q_n,\overline{q_0}),\\ 
|d^c_{JM}(q_n,\partial Q_h)-d^c_{JM}(\overline{q_0},\partial Q_h)|&\le d^c_{JM}(q_n,\overline{q_0}).
\end{array}
$$
Applying Lemma \ref{continuity-JM-dist} we get $d_{JM}(\overline{q_0},\partial Q_h)=d^c_{JM}(\overline{q_0},\partial Q_h)$, which  is a contradiction. 
The neighborhood $\mathcal U$ of  of the collision locus described in the Theorem can be   taken to be the union of the ${\mathcal O}(\overline{q_0})$ as $\overline{q_0}$ varies over the collision locus. This finishes the proof of Theorem \ref{thm:variational}.

{\bf Proof of Theorem \ref{thm:equal_masses}.}
In order to prove part (a) of the Theorem, assume $m_1=m_2$ and let $q_0$ be a configuration on the collision ray $r_{12}=0$. Let  $\gamma : [0,T]\rightarrow Q_h$ be a $JM$-minimizer from $q_0$ to $\partial Q_h$. If $q_0$ is the triple collision point, by 
Theorem \ref{thm:variational} we know that $\gamma$ is the Lagrange brake solution, which  is an isosceles brake orbit. 
If $q_0$ is a double collision point, then  express $\gamma$ in terms of the   spherical coordinates system $(r, \phi,\theta)$ on the shape space  $S^2$, 
centered at the double collision $b_{12}$ as in   Lemma  \ref{mutual-distances} and in the proof of 
Theorem \ref{thm:variational}.   We can assume that $\gamma(t)$ lies in the Northern hemisphere  on the upper half-space  $0 \le \theta\le \pi$.   Because   $m_1=m_2$, the isosceles curve $r_{13}=r_{23}$ in the Northern hemisphere
corresonds to the   great circle 
$\{ \theta=\pi/2\}$.  By 
Lemma \ref{mutual-distances}, for fixed $\phi$, the absolute minimum of $\theta\mapsto V(\phi,\theta)$ is achieved at 
the isosceles curve 
$\theta=\pi/2$.   Let  $0<\tau\le T$ be the first 
time such that  $V(\phi(t),\pi/2)/r(t)=h$. Then the  path $\tilde{\gamma}(t)=(r(t),\phi(t),\pi/2)$, $t\in [0,\tau]$, joins $q_0$ to 
$\partial Q_h$, has $U(\tilde \gamma (t)) \le U(\gamma(t))$ with equality if and only if $\tilde \gamma = \gamma$
on its domain, and has kinetic energy less than or equal to that of $\gamma$ on its domain. 
Thus  
$A_{JM}(\tilde{\gamma})\le A_{JM}(\gamma)$, with equality if and only if $\gamma(t)$ is isosceles at all times $t\in [0,T]$. 
Since $\gamma$ is a minimizer, it is necessarily isosceles. This proves part (a). 

To prove part (b), assume $r_{13}<r_{23}$ at the starting point $q_0$. The isosceles set $r_{13}=r_{23}$ is invariant by 
the flow. If we assume, for the sake of contradiction, that $\gamma : [0,T]\rightarrow Q_h$ 
intersects $r_{13}=r_{23}$ at some time $0<\tau \le T$, then it is necessarily transverse to $r_{13}=r_{23}$, otherwise 
$\gamma(t)$ would be isosceles at all times. The reflection with respect to the isosceles plane $r_{13}=r_{23}$ is a 
symmetry for the kinetc energy and for the potentiel $U$, therefore the path $\tilde{\gamma}$ obtained by keeping 
$\gamma\left|_{[0,\tau]}\right.$ as it is, and reflecting  $\gamma\left|_{[\tau,T]}\right.$ about the isosceles plane  $r_{13}=r_{23}$,
is still a $JM$-minimizer. But $\tilde{\gamma}$ is not differentiable at $t=\tau$. This is  a contradiction. 

The proof of part (c) is obtained applying (b) to the three isosceles plane $r_{ij}=r_{ik}$.  QED.

\vspace{2mm}  
{\bf Proof of Proposition \ref{seifert-nbd}: Seifert's Tubular Neighborhood Theorem.}

Our construction   is the same as Seifert's   \cite{Seifert}. The one real difference is that our
 Hill boundary $\partial Q_h$ is not compact, whereas his is compact.

We will construct  Seifert's neighborhood  in the non-reduced configuration space 
$\C^2$ and then project  it via  the quotient map $\pi$ of eq (\ref{eq:quotientmap}) to $Q_h$
to arrive at the reduced Seifert \nbhd.  
We  abuse notation by using the same symbol $Q_h$ for  the reduced and non-reduced Hill regions
  and  similarly using the same symbol 
   $S^{\epsilon}_h\subset Q_h$
for the reduced and non-reduced Seifert neighborhoods to be constructed.

The configuration space  $\C^2$ is endowed with the mass inner product $<\, ,\, >$ which is the
real part  of the Hermitian inner product of eq.  (\ref{eq_Hermitian}).
The equations there are  
\begin{equation} \label{equat-newt-non-reduced}
\ddot{q}=\nabla U (q),
\end{equation}
where the gradient of $U$ is calculated with respect to the mass inner product. 
The Jacobi-Maupertuis action of a curve $\gamma :[t_0,t_1]\rightarrow Q_h$ is defined by
$$
A_{JM}(\gamma)=\int_{t_0}^{t_1} \sqrt{K} \sqrt{2(U-h)} dt,
$$
where $K=\|\dot{\gamma}\|^2/2$.

As a first step  we construct a map $F:\partial Q_h \times [0,\overline{\delta}]\rightarrow Q_h$ which is   an 
analytic diffeomorphism onto  its image and is such that the curves  $t\mapsto F(x,t^2)$ are  the 
brake solutions starting from $x$ at $t=0$.
 Given $\alpha<h<\beta$,  define the open set
$$
D_{\alpha,\beta}=\{x\in \C^2, \quad \alpha<U(x)<\beta \}.
$$   
If $x\in D_{\alpha,\beta}$, then the  smallest of the  mutual distances $r_{ij}$
 of the triangle defined by  $x$ is bounded below by a constant depending only on $\beta$ 
and the masses. We denote by $q_x(t)$ the solution to  (\ref{equat-newt-non-reduced}) with initial conditions 
$q_x(0)=x$ and $\dot{q}_x(0)=0$. (Note: these solutions need not have energy $h$.) Choose positive  numbers  $a$ and $b$ such that $a<\alpha<h<\beta<b$. 
By classical results on differential equations, 
there exists $T>0$ such that for every $x\in D_{\alpha,\beta}$, the solution $q_x(t)$ is well defined for $t\in [-T,T]$ and 
satisfies $q_x(t)\in D_{a,b}$. The map $(x,t)\mapsto q_x(t)$ is analytic,  even in $t$ (i.e. $q_x(-t)=q_x(t)$), 
and its derivatives up to the second order are uniformly bounded on $D_{\alpha,\beta}\times [-T,T]$. By equations 
of motion (\ref{equat-newt-non-reduced})
$$
q_x(t)=x+\frac{\nabla U(x)}{2} t^2 +{\mathcal O}(t^4).
$$       
Let us set $T_1=\sqrt{T}$. Since $q_x(t)$ is even in $t$, the map $F(x,\tau)=q_x(\sqrt{\tau})$, 
$(x,\tau)\in D_{\alpha,\beta}\times [0,T_1]$ is still analytic, so it can be extended to negative values of $\tau$, 
giving an analytic map 
$F:D_{\alpha,\beta}\times [-T_1,T_1]\rightarrow D_{a,b}$ satisfying
\begin{equation} \label{ciao}
F(x,\tau)=x+\frac{\nabla U(x)}{2}\tau +f(x,\tau),
\end{equation}
where the $f(x,\tau)=\mathcal O(\tau^2)$, uniformly for $x\in D_{\alpha,\beta}$. 

Define  
$$
G: D_{\alpha,\beta}\times [-T_1,T_1]\rightarrow D_{a,b}\times \R,\qquad G(x,\tau)=(F(x,\tau),U(x)).
$$
By (\ref{ciao}), the differential of $G$ at a point $(x,\tau=0)$ gives
$$
DG_{(x,0)}(\delta x, \delta \tau)=(\delta x+\frac{\nabla U(x)}{2} \delta \tau, <\nabla U(x),\delta x>).
$$
An easy computation shows that $DG_{(x,0)}$ is invertible and $\|(DG_{(x,0)})^{-1}\|$ is bounded as $x\in D_{\alpha,\beta}$. 
Moreover, since all derivatives of $G$ up to the second order are uniformly bounded on 
$D_{\alpha,\beta}\times [-T_1,T_1]$, we can apply a strong version of the inverse function theorem and find a
$\overline{\delta}>0$ such that for all $x\in D_{\alpha,\beta}$, the map $G$ defines a diffeomorphism from 
$B(x,\overline{\delta})\times [-\overline{\delta},\overline{\delta}]$ into its image, where $B(x,\overline{\delta})$ denotes 
the closed ball centered in $x$ with radius $\overline{\delta}$.   If we take the restriction to $U(x)=h$ and define
$B_{\partial Q_h}(x,\overline{\delta})=B(x,\overline{\delta})\cap \partial Q_h$ we find that $F$ defines an analytic 
diffeomorphism from $B_{\partial Q_h}(x,\overline{\delta})\times [-\overline{\delta},\overline{\delta}]$ into its image.
Let us prove now that by decreasing sufficiently $\overline{\delta}$, the restriction of $F$ to 
$\partial Q_h\times [-\overline{\delta},\overline{\delta}]$ is an analytic 
diffeomorphism onto its image. Indeed, we have proven that at every point it is a local diffeomorphism.
Assume, for the sake of contradiction, there exist two sequence $(x_n,\tau_n)_{n\in \N}$ and 
$(x^\prime_n,\tau^\prime_n)_{n\in \N}$ satisfying $(x_n,\tau_n)\neq (x^\prime_n,\tau^\prime_n)$   
such that $\tau_n\rightarrow 0$, $\tau^\prime_n \rightarrow 0$ as 
$n\rightarrow +\infty$ and $F(x_n,\tau_n)=F(x^\prime_n,\tau^\prime_n)$ for all $n\in \N$.  By (\ref{ciao}) 
and uniform 
boundedness of $\nabla U$ on $\partial Q_h$, we have $\|x_n-x^\prime_n\|\rightarrow 0$ as $n\rightarrow +\infty$, 
therefore 
$x^\prime_n\in  B_{\partial Q_h}(x_n,\overline{\delta})$ if $n$ is sufficiently great. 
This contradicts the fact that $F$ is a
diffeomorphism on every $B_{\partial Q_h}(x,\overline{\delta})\times [-\overline{\delta},\overline{\delta}]$. 
\\ \noindent

The second step consists in defining a diffeomorphism $\Phi$ on $\partial Q_h \times [0,\overline{\delta}]$
 of the form
$\Phi(x, y) = F(x, \beta(x,y))$  where the scalar  function $\beta$ allows us to 
``straighten out'' the JM action.  
Let $S(x,\tau)$ denotes the $JM$-length of the extremal $F(x,s)_{s\in [0,\tau]}$.  From  (\ref{ciao}) we compute  
\begin{equation} \label{S(x,tau)}
S(x,\tau)=\tau^{3/2} \left(\frac{1}{3}\|\nabla U(x)\|^2+g(x,\tau)\right),
\end{equation}
where $g(x,\tau)$ is analytic, and $g(x,\tau)={\mathcal O}(\tau)$, uniformly on $x\in \partial Q_h$. The map 
$H(x,\tau)=(x,S(x,\tau)^{2/3})$ is analytic in $\partial Q_h\times [0,\overline{\delta}]$, and we can find 
$\epsilon>0$ such that $H$ maps  a neighborhood of $\partial Q_h\times \{0\}$ 
(in $\partial Q_h\times [0,\overline{\delta}]$) diffeomorphically into $\partial Q_h\times [0,\epsilon]$.  
Then $H^{-1}$ has the form $(x,y) \mapsto (x, \beta(x,y))$.  We set  
$\Phi=F\circ H^{-1}$ where the domain of $H^{-1}$ is $\partial Q_h\times [0,\epsilon]$. It is clear now that for 
every $x\in \partial Q_h$, and for every $\delta\in (0,\epsilon)$, the curve $(\Phi(x,y))_{y\in [0,\delta]}$ is a 
segment of brake orbit and its $JM$-action is $\delta^{3/2}$.

The Jacobi metric $g_h$ in $Q_h \subset \C^2$ is given by 
$$
g_h (x) (v,v)=2(U(x)-h)\langle v, v \rangle  
$$ 
We now  show that brake solutions $(\Phi(x,y))_{y\in [0,\epsilon]}$ are $g_h$-orthogonal to hypersurfaces 
$\Phi(\partial Q_h\times\{y\})$.   The argument
is the standard one used to prove  the Gauss lemma in \Ri geometry.  
We are to show that $g_h(\partial_y \Phi(x,y),\partial_{x} \Phi(x,y)\xi)=0$ for all $\xi$ tangent to the hypersurface.
For this purpose, set $f(y;x) = g_h(\partial_y \Phi(x,y),\partial_{x} \Phi(x,y)\xi)$ for fixed $\xi$. 
Because the lengths of the curves $y \to \Phi(x,y)$ up to the value $y_0$
are all  $y_0 ^{3/2}$ we have that $g_h(\partial_y \Phi(x,y),\partial_{y} \Phi(x,y)) = \frac{9}{4} y$,
independent of $x$. Differentiating this identity in the $\xi$-direction 
 tangent to the hypersurface and commuting derivatives
  we get $g_h(\partial_y \Phi(x,y),\nabla_{y} \partial_{x} \Phi(x,y)\xi) = 0$ where $\nabla$ denotes the $g_h$-Levi-Civita connection.   Now, because the $y$-curves are reparameterized geodesics with lengths only depending   on $y$
we have that $\nabla_{y}\partial_{y}\Phi(x,y) = \mu(y)\partial_{y}\Phi(x,y)$ for some function $\mu(y)$.
Differentiate $f(y;x)$  with   respect to $y$,  and use  metric compatibility and the previous equations to derive the  linear differential equation $\partial_y f = \mu(y) f(y;x)$ with $x$
as a parameter. But  $f(0;x) = 0$ so the initial condition of this differential equation is zero,  from which it follows that $f$ is identically $0$. (Indeed $\mu$ is the second derivative of $y^{3/2}$ with respect to $y$, which is singular at
$y = 0$, but the
analysis goes through.) 

An arbitrary  tangent vector $v\in T_q Q_h$ at $q = \Phi(x,y)$
can  be written  
$$v=\partial_x\Phi(x,y)\xi+\partial_y \Phi(x,y)\lambda,$$ where 
$(\xi,\lambda)\in T_x\partial Q_h\times \R$. By the previous  orthogonality discussion and 
the fact that $y^{3/2}$  is the arclength of 
$(\Phi(x,y))_{y\in [0,\epsilon]}$ we have for such $v$:
\begin{equation} \label{epress-Jacobi}
g_h(v,v)=(U-h)\kappa_h(x,y)(\xi,\xi)+\frac{9 y}{4} \lambda^2. 
\end{equation} 
where $\kappa_h(x,y)$ is a positive definite quadratic form on $T_x\partial Q_h$. 
If $\gamma : [t_0,t_1]\rightarrow Q_h$ is any rectifiable curve joining a point $q=\Phi(x,\delta)$ (with $0<\delta<\epsilon$) 
to $\partial Q_h$,  and if $I$ is the set of times $t\in [t_0,t_1]$ such that $\gamma(t)\in S^{\epsilon}_h=\Phi(\partial Q_h,[0,\epsilon])$, by (\ref{epress-Jacobi}) we have 
\begin{equation} \label{AJM-delta}
A_{JM}(\gamma)=\int_{t_0}^{t_1} \sqrt{g_h(\dot{\gamma}(t),\dot{\gamma}(t))} dt \ge \int_{I} \frac{3\sqrt{y(t)}}{2}|\dot{y}(t)| dt \ge \delta^{3/2},
\end{equation}
where we term $(x(t),y(t))=\Phi^{-1}(\gamma(t))$ for $t\in I$. Moreover, we have equality in (\ref{AJM-delta}) 
if and only if, up to an arc contained in $\partial Q_h$, the curve $\gamma$ is a reparametrization of 
$(\Phi(x,y))_{y\in [0,\delta]}$. The Euclidean length of $(\Phi(x,y))_{y\in [0,\delta]}$ is given by the 
integral $\int_0^{\delta} \left\|\partial_y \Phi (x,y) dy \right\| dy$. By (\ref{ciao}), (\ref{S(x,tau)}) and by definition of $\Phi$ we have 
\begin{equation} \label{Phi-asympt}
\Phi(x,y)=x+\frac{\nabla U(x)}{2\cdot 3^{2/3} \|\nabla U(x)\|^{4/3}}y+{\mathcal O}(y^2).
\end{equation}
Since $\nabla U(x)$ is uniformly bounded on $\partial Q_h$, there exists a strictly positive constant $M$, 
independent of $x$ and $\delta$ as long as $\delta < \epsilon$, such that 
$$
\int_0^{\delta} \left\|\partial_y \Phi (x,y) dy \right\| dy\le M\delta. 
$$
Let us show now there exists $\alpha>0$ such that for every $\delta\in (0,\epsilon)$ we have $U\ge h+\alpha \delta$ on 
$Q_h\setminus S_h^\delta$. By (\ref{Phi-asympt}) we have 
\begin{equation} \label{U(Phi)}
U(\Phi(x,y))=h+\frac{\|\nabla U(x)\|^{2/3}}{2\cdot 3^{2/3}}y+{\mathcal O}(y^2),
\end{equation}
for $(x,y)\in \partial Q_h\times [0,\epsilon]$. Since $\|\nabla U(x)\|$ is bounded and uniformly bounded below below 
by a positive constant, we can find two constant $0<\alpha<\beta$ such that
$$
h+\alpha y\le U(\Phi(x,y))\le h+\beta y
$$  
for every $(x,y)\in \partial Q_h\times [0,\epsilon]$. Assume now, for the sake of contradiction, the exisence of a point 
$q\in Q_h\setminus S_h^\delta$ such that $U(q)<h+\alpha \delta$. Since level set of $U$ are connected by arc, 
there would exist a point $q^\prime$ on $\Phi(\partial Q_h,\delta)$ such that $U(q^\prime)< h+\alpha \delta$, 
and this is in contradiction with (\ref{U(Phi)}).      
%
%
%
%                     CONTINUER AVEC LA CONSTRUCTION DU VOISINAGE DE SEIFERT
%
%
%
%
\section{Periodic Brake Orbits}\label{SecPeriodicBrake}
The goal of this section is to prove Theorem~\ref{theorem_isosbrake} about  the existence of simple, periodic brake orbits for the isosceles three-body problem.  Other examples of periodic, isosceles brake orbits, more complicated than the one described here, are given in  \cite{SimoMartinez}.

Let $p$ be a brake initial condition other than the Lagrange homothetic one, $p_0$.  According the results of section~\ref{SecExistSyzygy}, $p$ can be followed forward to meet the syzygy submanifold, $\tilde C_h$, i.e., the submanifold of the energy manifold with collinear shapes.  There is a natural reflection symmetry of the energy manifold through $\tilde C_h$, obtained by reflecting the shape variables $(x,y)$ and their velocities $(x',y')$ while leaving the size variables $(r,v)$ unchanged.  Call this reflection map $R$.  Then $R$ is a symmetry of the differential equation if one also reverses time.  If the orbit, $\g(t)$, of $p$ meets  $\tilde C_h$ orthogonally after time $T_1$, then reflecting the orbit segment and reversing time gives the continuation of $\g$ to the time interval $[T_1,2T_1]$ and $\g(2T)=R(p)$, the reflection of $p$.  Now it follows from symmetry during the interval $[2T_1,4T_1]$ the orbit retraces its path, returning to $p$.  So $p$ determines a periodic brake orbit of period $4T$.

If $\g(t)$ does not meet  $\tilde C_h$ orthogonally and if $\g(T)$ is not in the local stable manifold of the Lagrange triple collision, then it can be followed to a {\em second syzygy}, say at time $T_2>T_1$.  If this crossing is orthogonal one obtains a periodic brake orbit of period $4T_2$ by reflection.

We will see that such a second-syzygy periodic brake orbit exists in the isosceles subsystem of the three-body problem, at least for certain choices of the mass parameters.  It is not known whether a first-syzygy brake orbit exists.  Numerical experiments  suggest that no such orbits exist in the equal mass case.

\subsection{The Isosceles Three-Body Problem}
Assume that two masses are equal, say $m_1=m_2=1$.  Then there is an invariant submanifold of the three body problem such that the shape remains an isosceles triangle with $m_3$ on the symmetry axis for all time.  Up to rotation, this isosceles subsystem is obtained by making the  Jacobi variables $\xi_1, \dot \xi_1$ real and $\xi_2, \dot \xi_2$ imaginary.

As in section~\ref{SecEq} separate size and shape variables will be used.  In addition, double collisions will have to be explicitly regularized.  A convenient way to do this is to define an angular variable which gives a multiple cover of the isosceles shape space and which is locally a branched double cover near the binary collisions, as in the familiar Levi-Civita regularization.   Using the projective Jacobi variables $[\xi_1,\xi_2]$, this parametrization of the isosceles shapes can be accomplished by setting
$$\xi_1 = \fr{1}{\sqrt{\mu_1}}\cos^2(\t)   \qquad   \xi_2 = \fr{i}{\sqrt{\mu_2}}2\sin(\t).$$
The isosceles binary collision shape corresponds to $\xi_1 = 0$ or $\t = \pm \smfr{\pi}{2} \bmod 2\pi$.

Note that with the assumptions about the masses, one has
$$\mu_1= \fr{1}{2} \qquad \mu_2 = \fr{2m_3}{2+m_3}.$$

Substitution gives a reduced Lagrangian
$$L_{red}(r,\dot r,\t,\dot\t) = K_0 + \fr{1}{r}V(\t)$$
where 
$$K_0 = \fr{\dot r^2}{2} + \fr{2 r^2\cos^2\t\, \dot\t^2}{(1+\sin^2\t)^2}$$
and
$$V(\t) = (1+\sin^2\t)\left(\fr{1}{\sqrt{2}\cos^2\t} + \fr{2\sqrt{2}\,m_3}{\sqrt{ (1+\sin^2\t)^2 + \smfr{8}{m_3}\sin^2\t  }} \right).$$

Introducing a change of time scale ${}' = r^{\smfr32}\cos^2\t\;\dot{}$ leads to the following system of differential equations:
\begin{equation}\label{eq_isosceles}
\begin{aligned}
r' &= vr\cos^2\t\\
v' &= \smfr12 v^2 \cos^2\t  + \smfr14 w^2(1+\sin^2\t)^2-W(\t)\\
\t'  &= \smfr14 w(1+\sin^2\t)^2 \\
w'  &= W'(\t)-\smfr12 vw\cos^2\t +\sin\t\cos\t\left(2r+v^2-\smfr12w^2(1+\sin^2\t)\right)
\end{aligned}
\end{equation}
where $W(\t) = \cos^2\t \,V(\t)$.
The energy conservation equation is
\begin{equation}\label{eq_energyisosceles}
\smfr12 v^2\cos^2\t +  \smfr18 w^2(1+\sin^2\t)^2-W(\t) = -r\cos^2\t
\end{equation}
where the energy has been fixed at $-1$.  

Since $W(\t)$ is an analytic function,  the binary collisions have been regularized.  Moreover the triple collision singularity has been blown up into an invariant manifold at $\{r=0\}$ as before.  The isosceles Hill's region  for energy $-1$ is 
$Q_1 = \{(r,\t): 0 \le r \le V(\t) \}$.  This is shown in figure~\ref{fig_isosHill} for the equal mass case $m_3=1$.  

Using these coordinates, syzygies occur at the Euler shape ($\t = 0 \bmod 2\pi$) and at binary collision 
($\t = \pm \smfr{\pi}{2} \bmod 2\pi$).  These are indicated by the bold vertical lines in figure~\ref{fig_isosHill}.
The reflection through syzygy amounts to reflecting the position variables $(r,\t)$ through these vertical lines 
while taking the velocity $(v,w)$ to $(v,-w)$ and reversing time.  To get a symmetric periodic orbit, one needs to reach syzygy with $v=0$.  Since $r' = vr\cos^2\t$ this is equivalent to orthogonality at the Euler vertical lines $\t = 0 \bmod 2\pi$, but every orbit crosses the lines $\t = \pm \smfr{\pi}{2} \bmod 2\pi$ orthogonally.  A numerically computed periodic brake orbit is shown in the figure.  It begins on the zero velocity curve, crosses the syzygy line at $\t = -\smfr{\pi}{2}$ with $v<0$ (not apparent from the figure), then continues to its second syzygy at $\t=0$ where it crosses orthogonally with $v=0$.  The rest of the orbit is obtained by symmetry.  The goal of this section is to prove the existence of such an orbit for certain choices of the parameter $m_3$.  
\begin{figure}
\scalebox{.6}{\includegraphics{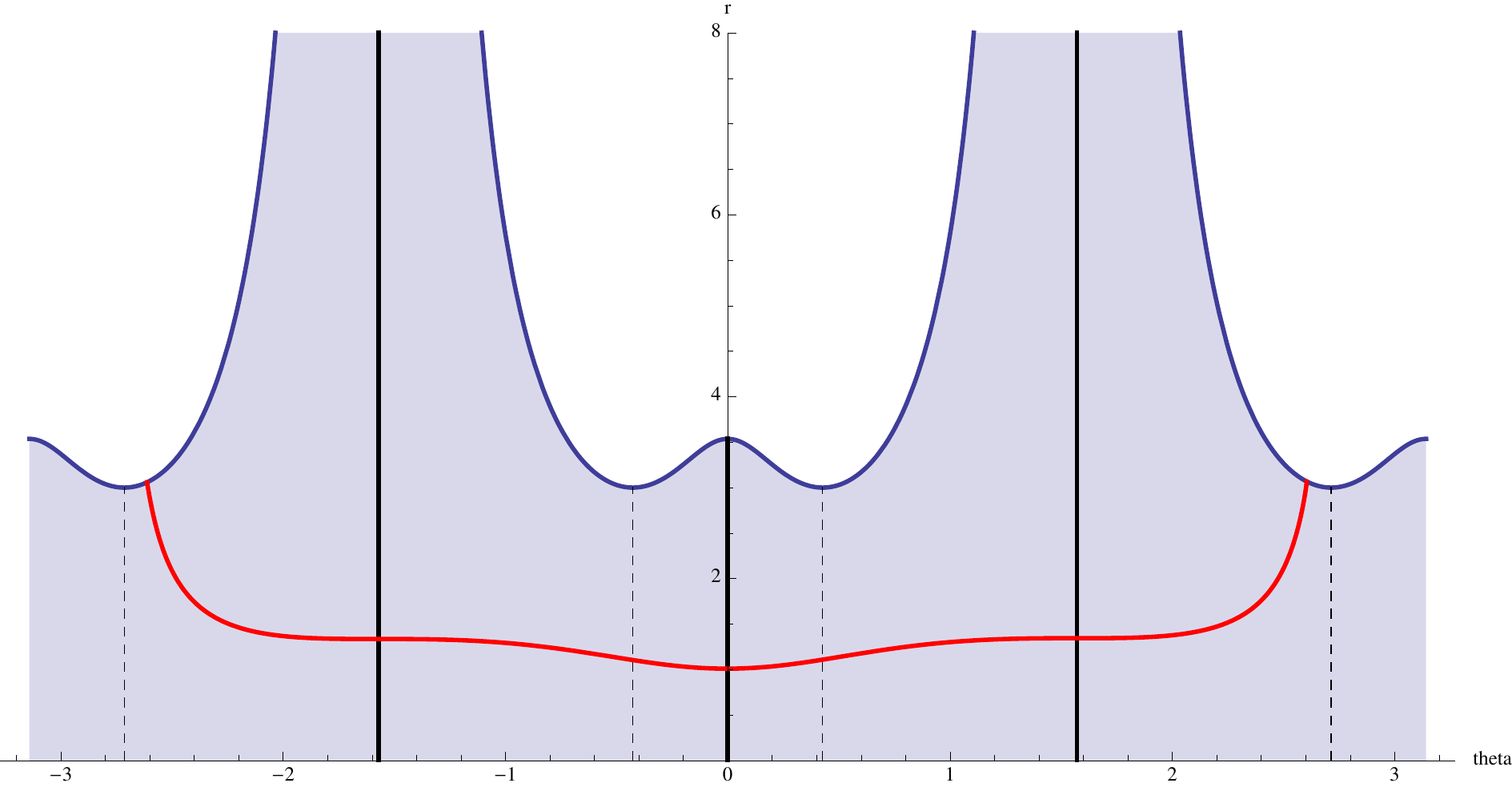}}
\caption{Isosceles Hill's region for the equal mass three-body problem using coordinates $(\t,r)$.  The zero velocity curve is the top boundary curve.  The syzygy configurations are represented by the thick vertical lines ($\t=0$ is the collinear central configuration and $\t=\pm\smfr{\pi}{2}$ are binary collision shapes). A numerically computed periodic brake orbit is also shown. }\label{fig_isosHill}
\end{figure}

The idea of the proof can be described briefly as follows. Consider a curve, $Z$, of brake initial conditions whose shapes vary from equilateral to binary collision.   In figure~\ref{fig_isosHill}, the equilateral shapes (minima of the shape potential) are shown with dashed vertical lines.    Let $\t^*$ be the equiliateral shape in the interval $[0,\smfr{\pi}{2}]$.  Other equilateral shape occur at the points $\pm \t^* +k\pi$ where $k$ is an integer.  The curve $Z$ will consist of the brake initial conditions with $\t\in[\t^*-\pi,-\smfr{\pi}{2}]$ (in figure~\ref{fig_isosHill}, this is the part of the top boundary curve between the left-most dashed and left-most bold vertical lines).  It will be shown that as the initial condition $p\in Z$ varies, there is at least one point which can be followed to meet the vertical line $\t=0$ orthogonally.  

The proof will use a geometrical argument in the three-dimensional energy manifold, $P_{1}=\{(r,\t,v,w):r\ge 0, H = -1\}$. This manifold can be visualized through its projection to $(r,v,\t)$-space, which is given by the inequality $r +\smfr12 v^2\le V(\t)$.
This projection is shown in figure~\ref{fig_energymanifold}.   On the top surface of the projection, $w=0$.  The full energy manifold can be viewed as two copies of this projection (one with $w\ge 0$ and one with $w\le 0$) glued together along this top surface.  The desired orbit has the property that the shape angle $\t(t)$ will increase monotonically from $\t(0)$ to $\t(T_2) = 0$ where $T_2$ denotes the second-syzygy time described above.  Therefore it suffices to consider the part of the energy manifold with $\t' =\smfr 14 w (1+\sin^2\t)^2 \ge 0$.  On this half of the energy manifold, one can solve (\ref{eq_energyisosceles}) uniquely for $w(r,v,\t)$.  

\begin{figure}
\scalebox{1.0}{\includegraphics{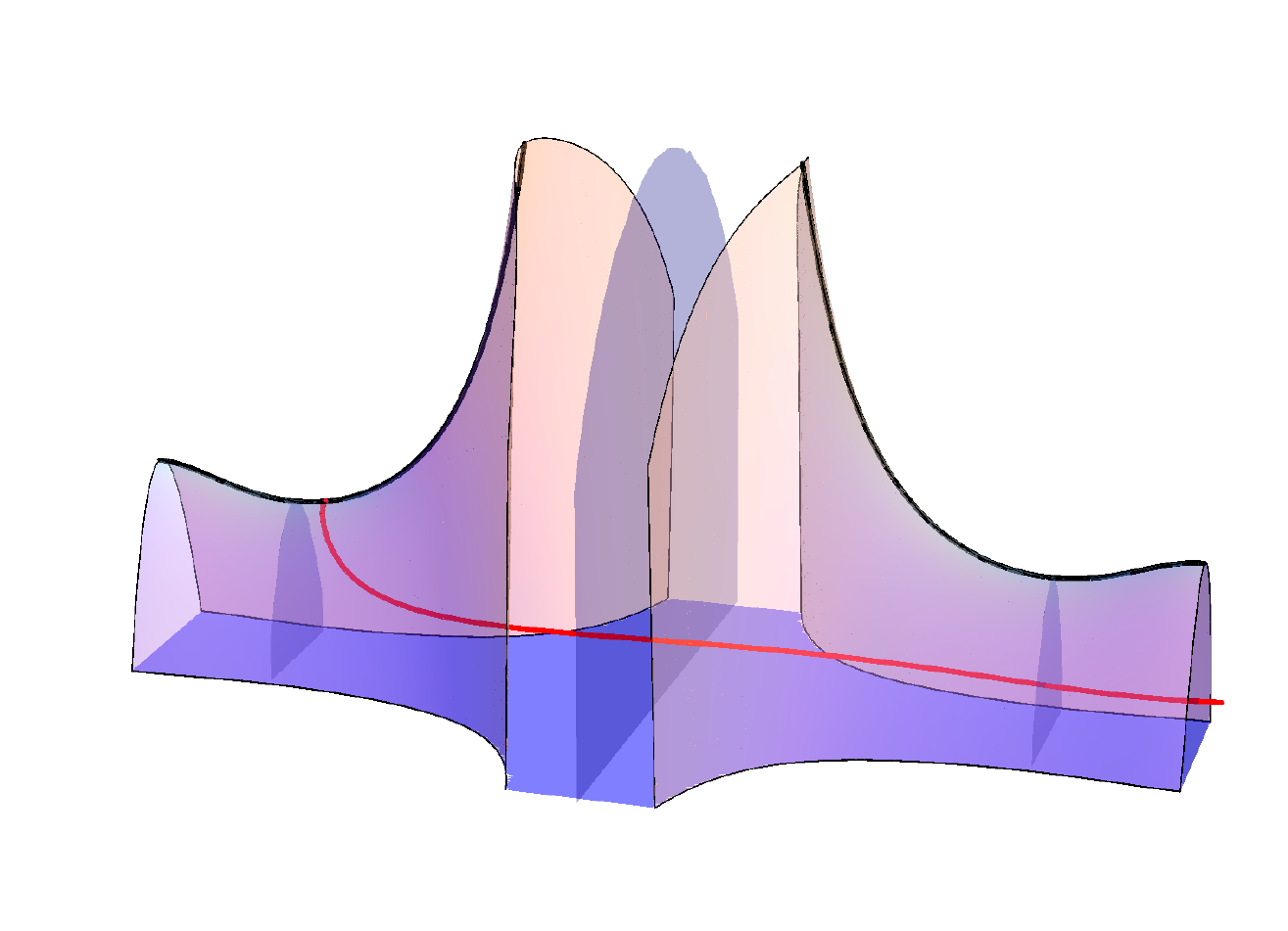}}
\caption{Projection of the $w\ge 0$ part of the energy manifold in coordinates $\t$ (width), $v$ (depth), $r$ (height),  The top surface is $\{w=0\}$; the floor is part of the collision manifold $\{r=0\}$.  A numerically computed periodic brake orbit is also shown, passing through the walls of the regions $R_I, R_{II}, R_{III}$ as described in the proof.}\label{fig_energymanifold}
\end{figure}

The relevant part of the energy manifold will be divided into three regions:
$$\begin{aligned}
R_I &= P_1\intersection\{\t\in[\t^*-\pi,-\smfr{\pi}{2}],w\ge 0\}\\
R_{II}&= P_1\intersection\{\t\in[-\smfr{\pi}{2},-\t^*],w\ge 0\}\\
R_{III}&= P_1\intersection\{\t\in[-\t^*,0],w\ge 0\}.
\end{aligned}
$$
The planes $\t =\t^*-\pi$ and $\t = -\smfr{\pi}{2}$ will be called the {\em left and right walls} of $R_{I}$ respectively with similar definitions for the other regions (the walls are the four vertical planes in figure~\ref{fig_energymanifold}).

These regions contain certain restpoints on the collision manifold which will now be described.  Let  $L_{\pm}$ be the Lagrange restpoints in $R_{I}$ at $(r,\t,v,w) = (0,\t^*-\pi,\pm v^*,0)$ where  $v^* = \sqrt{2V(\t^*)}$.  These are connected by the Lagrange homothetic orbit which is the curve of intersection of the left wall of $R_{I}$ with the boundary surface $\{w=0\}$.  Similarly,  there are Lagrange restpoints  $L'_{\pm}$ at $(r,\t,v,w) = (0,-\t^*,\pm v^*,0)$ and a corresponding homothetic orbit in the left wall of $R_{III}$ which is also the right wall of $R_{II}$.  Finally there are Eulerian restpoints $E_{\pm}$ at $(r,\t,v,w) = (0,0,\pm \sqrt{2V(0)},0)$ and an Eulerian homothetic orbit in the right wall of $R_{III}$.

Let $Z$ be the part of the zero velocity curve in $R_I$.   The proof will follow a subset of $Z$ forward under the flow through these three regions to obtain a curve in the right wall of $R_{III}$, i.e., in the syzygy set at $\t=0$.  It will be shown that this final curve crosses the plane $v=0$ and the point of crossing will determine the required periodic brake orbit.  The next lemma shows how the flow can be used to carry orbits across the various regions.

\begin{lemma}\label{lemma_regioncrossing}
Regions $R_{I}$ and $R_{III}$ are positively invariant sets for the flow while $R_{II}$ is negatively invariant. With the exception of the Lagrange homothetic orbits, orbits cross these regions as follows: any orbit beginning in the left wall of region $R_{I}$,  crosses the region and exits at the right wall.  The same hold for $R_{III}$ except for orbits in the stable manifold of $E_{+}$ (which is contained in the triple collision manifold $\{r=0\}$).

Similarly, except for the Lagrange orbit, any backward-time orbit beginning in the right wall of $R_{II}$ can be followed back to the  left wall.  Finally, forward orbits beginning in the left wall of $R_{II}$ either leave $R_{II}$ through the right wall, leave $R_{II}$ through the top surface $\{w=0\}$ or converge to one of the Lagrange restpoints $L'_{\pm}$ in the right wall as $s\into\infty$.
\end{lemma}
\begin{proof}
By definition, an orbit in any of the three regions satisfies $\t'\ge 0$ so $\t(s)$ is non-decreasing.   Referring to figure~\ref{fig_energymanifold}, the lower boundary surface $\{r=0\}$ is invariant.  On the upper boundary surface, $w=0$  and
$$w'=  W'(\t)+\sin\t\cos\t\left(2r+v^2\right) = W'(t)+2\sin\t\cos\t W(\t) = \cos^2\t\, V'(\t).$$
Now $R_{I}, R_{III}$ are are defined by $\t$-intervals where $V'(\t)\ge 0$ so for these regions, $w'\ge 0$ on the top boundary.  This proves positive invariance and also that it is only possible to leave through the right wall.  Similarly, in $R_{II}$ we have $w'\le 0$ on the top wall so the region is negatively invariant and backward orbits can only leave through the left wall.

Assume for the sake of contradiction, that an orbit other than the Lagrange homothetic orbit remains in $R_{I}$ for all time $s\ge 0$. Then $\t(s)$ converges monotonically to some limit $\t_\infty\in (\t^*-\pi,-\smfr{\pi}{2}]$ (it  cannot be $\t^*-\pi$ since all points in this plane but not on the Lagrange orbit have $w>0$ so they initially move to the right).  The omega limit set is either empty or else it must be a nonempty invariant subset of $\{\t = \t_\infty\}$.  However, these planes don't contain any nontrivial invariant sets.  So the omega limit set must be empty which is only possible if $\t_\infty=-\smfr{\pi}{2}$ and if  the orbit leaves every compact subset of the energy manifold.  To show that this is impossible, let $\l = \sqrt{2r+v^2}$ so that the energy equation becomes
$$\smfr12 \l^2 \cos^2\t +   \smfr18 w^2(1+\sin^2\t)^2= W(\t).$$
We will show that $\l(s)$ remains bounded to complete the argument.  From (\ref{eq_isosceles}) we find
$\l\l' = \smfr18 vw^2(1+\sin^2\t)^2$.  Since $\t(s)$ is increasing, we can reparametrize by $\t$ to get
$\fr{d\l}{d\t} = \smfr12 \frac{vw}{\l}$ and since $|v|\le \l$ we have
\begin{equation}\label{eq_lambdaestimate}
|\fr{d\l}{d\t}|  \le \smfr12 w.
\end{equation}
As $w(s)$ is bounded by the energy relation, we get a bound for $|\fr{d\l}{d\t}|$.  It follows that $\l(s)$ is bounded along the part of the orbit in $R_{I}$ as required.

The same argument applies to backward-time orbits in $R_{II}$, showing that they must reach the right wall.  The argument for forward orbits in $R_{III}$ is easier since this region is compact.   The omega limit set of an orbit remaining in $R_{III}$ for all time would have to be a nonempty invariant set in a plane $\t = \t_\infty \in (-\smfr{\pi}{2},0]$.  The only invariant sets are the Eulerian restpoints and homothetic orbit and so the omega limit set would have to be one of the restpoints.  However,  the stable manifold of $E_{-}$ is just the restpoint itself and the homothetic orbit, so the orbit must be in the stable manifold of $E_{+}$ which, it so happens, is contained in the collision manifold. 
\end{proof}

In addition to this lemma about region-crossing,  it will be necessary to use some facts about the flow on the isosceles triple collision manifold.  Setting $r=0$ in (\ref{eq_isosceles}) gives the dynamics on the triple collision manifold.  The energy equation (\ref{eq_energyisosceles}) gives
$$\smfr12 v^2\cos^2\t +  \smfr18 w^2(1+\sin^2\t)^2-W(\t) = 0.$$
Using this to eliminate $w$ gives a flow on part of the $(\t, v)$ plane satisfying $v^2 \le 2 V(\t)$.  This is shown in figure~\ref{fig_collisionwu} for the case $m_3=1$.  An important property of the flow is that it is gradient-like with respect to the variable $v$.  Indeed, using the energy equation gives $v' = \smfr18 w^2(1+\sin^2\t)^2\ge 0$ and it can be shown that $v(s)$ is strictly increasing except at the restpoints.
The restpoints $L_{\pm}, L'_{\pm}$ are saddle points.  

Certain properties of their stable and unstable manifolds will be used in the proof.  Let $\g, \g'$ denote the branches of $W^u(L_{-}), W^u(L'_{-})$ in  $\{w>0\}$ (bold lines in figure~\ref{fig_collisionwu}).
 The key properties needed to complete the existence proof refer to the intersections of these branches with the syzygy lines at $\t = -\fr{\pi}{2}$ and $\t = 0$.  We require that $\g$ remains in $\{w>0\}$ at least until it crosses these two syzygy lines and that  the intersection points are $(\t,v) = (-\fr{\pi}{2},v_1)$ where $v_1<0$ and $(\t,v) = (0,v_2)$ with $v_2>0$.   Furthermore we require that $\g'$ remains in $\{w>0\}$ at least until it crosses $\{\t=0\}$ at a point $(\t,v) = (0,v_3)$ with $v_3<0$.  Call a mass parameter $m_3$ {\em admissible} if these hypotheses hold.  The figure indicates that $m_3=1$ is admissible and the next lemma guarantees that this is so.  The proof will be given later.

\begin{figure}
\scalebox{.6}{\includegraphics{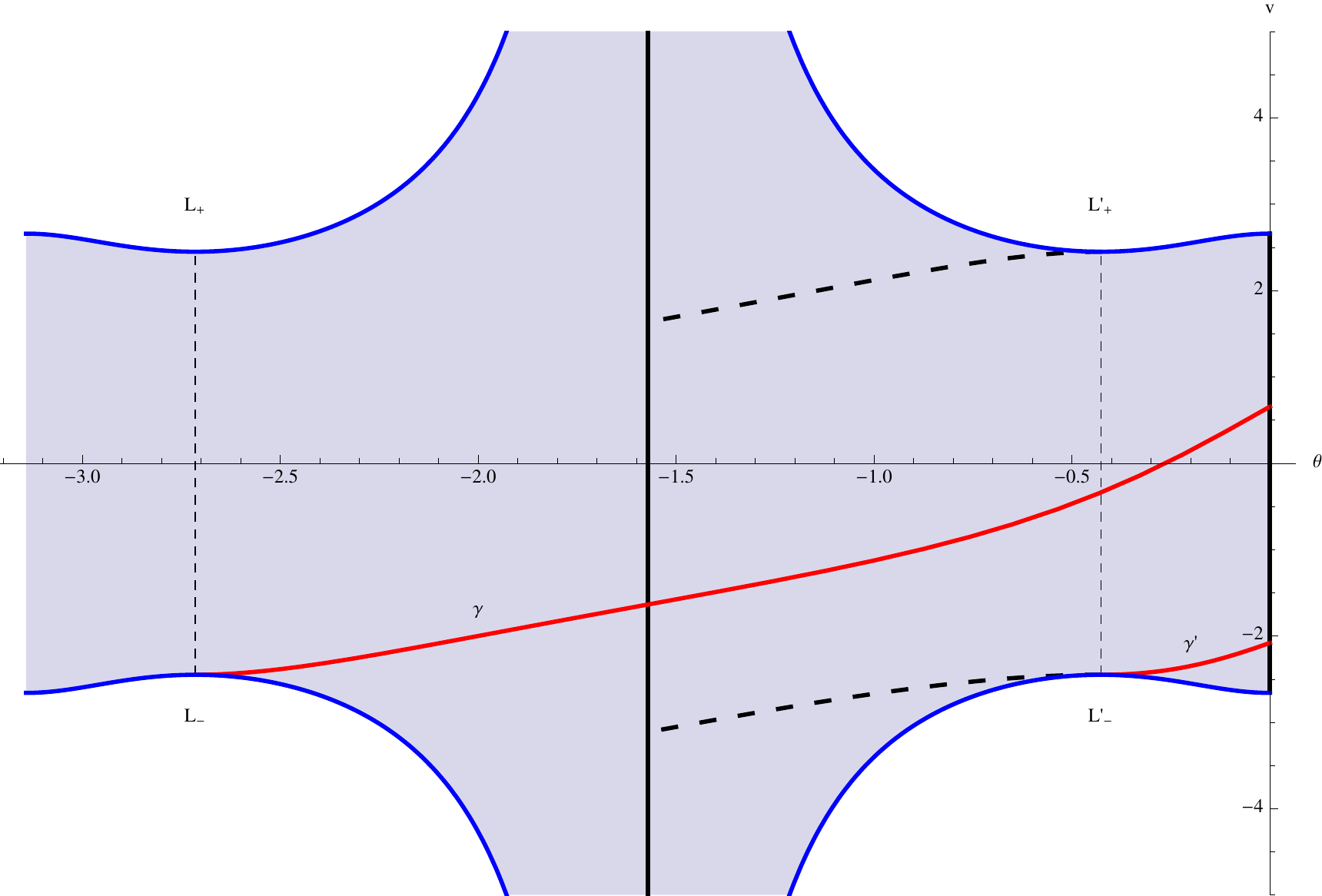}}
\caption{Flow on the $w>0$ part of the collision manifold in the equal mass case (the ``floor'' in figure~\ref{fig_energymanifold}). The crucial branches $\g, \g'$ of $W^u(L_{-}), W^u(L'_{-})$ are shown.  $\g$ intersects the line $\t = -\fr{\pi}{2}$ with $v<0$ and  $\t=0$ with $v>0$ and $\g'$ intersects $\t=0$ with $v<0$.  Two shorter branches of stable manifolds which play a role in the proof are also shown.}\label{fig_collisionwu}
\end{figure}

\begin{lemma}\label{lemma_admissiblemasses}
There is a nonempty open set of admissible masses which including the equal mass value $m_3=1$.
\end{lemma}

These lemmas can be applied to follow a subset of $Z$ through $R_I$ to a curve $Z_I$ in the right wall. 
The left endpoint of $Z$ is the brake initial condition with equilateral shape, $\t = \pi -\t^*$.  Call this point $p_0$.  The corresponding orbit is the Lagrange homothetic orbit, so it remains in $R_I$ for all $t\ge 0$ and converges to the restpoint  $L_{-}$.  It follows from lemma~\ref{lemma_regioncrossing}  that $Z\setminus p_0$ can be followed through $R_I$ to the right wall.  It will be important to understand the image curve $Z_I$ whose three-dimensional projection lies in the half-plane $\{(r,\t,v):r\ge 0 ,\t = -\smfr{\pi}{2}\}$. 
Use $(v,r)$ as coordinates in this half-plane  and let $p_I $ be the point $(v,r) = (v_1,0)$ where the branch $\g$ of $W^u(L_{-})$ meets the right wall.

\begin{lemma}\label{lemma_regionI}
Assume $m_3$ is admissible.   Then the image curve $Z_I$ of $Z\setminus p_0$ is a continuous open arc in the right wall of $R_I$.  One end converges to $p_I = (v_1,0)$ and at the other end $|(r,v)|\rightarrow \infty$.
\end{lemma}
\begin{proof}
The initial curve $Z\setminus p_0$ is a continuous open arc in the zero velocity curve from $p_0$ to $(r,\t,v,w) = (\infty, -\smfr{\pi}{2},0,0)$.   The flow across the right wall of $R_I$ is transverse since $w>0$ there.  So $Z_I$ is a continuous open arc.
Initial conditions near $p_0$ on $Z$ will follow the Lagrange homothetic orbit to a neighborhood of the restpoint $L_{-}$ and then follow the branch $\g$ of the unstable manifold to meet the right wall near the point $p_I$.

Next consider an initial point near the other end of the curve.  Initially the quantity  $\l = \sqrt{2r+v^2}$ is large and $\t\approx -\smfr{\pi}{2}$.  It follows from (\ref{eq_lambdaestimate}) that $\l$ is still large when the orbit meets the right wall.
\end{proof}

Now part of the curve $Z_{I}$ will be followed across region $R_{II}$.  Unfortunately, this region is not positively invariant (solutions can leave by $w$ becoming negative, i.e., $\theta$ begins to decrease).  It turns out, however, that part of the curve $Z_{I}$ is trapped  inside $R_{II}$ by an invariant surface, namely, the stable manifold $W^s(L'_-)$.   From the linearization at the restpoint, it follows that $W^s(L'_-)$ has dimension 2.  One of the orbits in the stable manifold is the Lagrange homothetic orbit which lies in the right wall of $R_{II}$ connectin $L'_+$ to $L'_-$.   It is known that $W^u(L'_+)$ and $W^s(L'_-)$ intersect transversely along this orbit \cite{SimoLLibre}.  $W^s(L'_-)$ also contains two orbits in the collision manifold, one of which lies in region $R_{II}$.   The one-dimensional manifold $W^s(L'_+)$ also contains an orbit in $R_{II}$.  These two branches of stable manifolds are shown as dashed curves in figure~\ref{fig_collisionwu}.

Consider the  ``quadrant'' of the surface $W^s(L'_-)$  in region $R_{II}$.  One edge is the orbit in the collision manifold just described and the other is the Lagrange homothetic orbit in the right wall.  It follows from  lemma~\ref{lemma_regioncrossing} that  with the exception of the homothetic orbit itself, orbits in this quadrant can be followed backward under the flow to reach the left wall of the region.  The intersection of the surface and the wall will be a curve.  One endpoint of the curve arises from the branch of  $W^s(L'_-)$ in the collision manifold.   Since $v$ is decreasing for backward orbits in the collision manifold, this endpoint will be of the form $(v,r) = (v_0,0)$ with $v_0<-v^*<v_1 <0$.    To find the other endpoint, note that backward orbits in the quadrant near the homothetic orbit will follow the homothetic orbit back near the restpoint $L'_+$ and then follow the branch of $W^s(L'_+)$  in $R_{II}$.  This branch of stable manifold is related by symmetry to the branch $\g$ of $W^u(L_-)$.  Hence it meets the left wall of $R_{II}$ at $(v,r) = (-v_1,0)$, $-v_1>0$.  

So the surface $W^s(L'_-)$ intersects the left wall of $R_{II}$ in a curve connecting the two points $(v_0,0), (-v_1,0)$ in the collision manifold but otherwise lying in $\{r>0\}$.  It follows from lemma~\ref{lemma_regionI} that $Z_{I}$ crosses this curve.  Let $Z'_{I}$ denote the part of $Z_{I}$ below the stable manifold from $(v,r) = (-v_1,0)$ to the first intersection, call it $q_I$, with $W^s(L'_-)$.  Points in $Z'_{I}\setminus q_I$ can be followed forward across region $R_{II}$ to its right wall.

Indeed if $p\in Z'_{I}\setminus q_I$, the forward orbit of $p$ remains in the part of $R_{II}$ below the invariant manifold $W^s(L'_-)$.  By lemma~\ref{lemma_regioncrossing}, it either reaches the right wall or converges to $L'_{\pm}$.   By construction, the only point of $Z'_{I}$ in $W^s(L'_{\pm})$ is $q_I$.  Points of $Z'_{I}$ near $q_I$ will reach the right wall of $R_{II}$ near $L'_-$.  The other endpoint  of $Z'_{I}$  is the point $p_I$ with $(v,r) = (v_1,0)$ in the branch of $W^u(L_-)$ described in the definition of admissible mass.  This point continues to follow that branch to its intersection point with the right wall.
Thus the image of $Z'_{I}\setminus q_i$ is an open arc $Z_{II}$ in the right wall  of $R_{II}$ connecting $L'_-$ to  $\g$  and otherwise lying in $\{r>0\}$.

Since $Z_{II}$ does not intersect the Lagrange homothetic orbit or the stable manifold of $E_+$, lemma~\ref{lemma_regioncrossing} shows that it can be followed across region $R_{III}$ to the right wall to form an open arc $Z_{III}$.   The endpoints of $Z_{III}$ are at the point $(v,r) =(v_2,0)$ and  $(v_3,0)$ where the branches of unstable manifolds $\g$ and $\g'$ cross the wall.  Since $m_3$ is admissible, we have $v_3<0<v_2$, so the is at least one point of $Z_{III}$ with $v=0$.  Thus there is a point $p$ of the zero velocity curve $Z$ which can be followed through all three regions to reach $\t=0$ with $v=0$ (the curve in figure~\ref{fig_energymanifold}).  Using the reflection symmetries we get a periodic brake orbit, completing the proof of Theorem~\ref{theorem_isosbrake}.

\subsection{Proof of Lemma~\ref{lemma_admissiblemasses}}
The behavior of the stable and unstable manifolds on the isosceles triple collision manifold as the mass ratio $m_3$ varies has been studied using a combination of analytical and numerical methods by Sim\'o (\cite{Simo}).  His results imply that the admissible masses, $m_3$, are those in the open interval  $(0,2.6620)$.  In this section we will only prove that $m_3=1$ is admissible.  The admissible masses form an open set so they will then include some open interval around 1.  
Using $(\t,v)$ as coordinates on the collision manifold and using the energy equation with $r=0$ in (\ref{eq_isosceles}) gives
\begin{equation*}
\begin{aligned}
v' &= \smfr18 w^2(1+\sin^2\t)^2\\
\t'  &= \smfr14 w(1+\sin^2\t)^2.
\end{aligned}
\end{equation*}
Any orbit segment with $w>0$ can be parametrized by $\t$ and we have
\begin{equation}\label{eq_dvdtheta}
\fr{dv}{d\t} =  \smfr12 w = \fr{2\sqrt{2W(\t)-v^2\cos^2\t}}{1+\sin^2\t} = \fr{2|\cos \t|\sqrt{2V(\t)-v^2}}{1+\sin^2\t}.
\end{equation}
Let $m_3=1$.  We need to follow the $w>0$ branches $\g, \g'$ of the unstable manifolds $W^u(L_{-}), W^u(L'_{-})$. These are 
solutions of (\ref{eq_dvdtheta}) beginning at $L_-, L'_-$ whose coordinates are $(\t,v) = (\pi-\t^*, -v^*),  (-\t^*, -v^*)$ where $v^*=\sqrt{2V(\t^*)} = \sqrt{6}$.  

The branch $\g'$ is short and easy to understand using crude estimates.  Let $v(\t)$ denote the corresponding solution of (\ref{eq_dvdtheta}).  We have  $v(-\t^*) = -v^* = -\sqrt{6}$ and we need to show that $v(0)<0$.  In the interval $[-\t^*,0]$ we have 
$V(\t) \le V(0) =  \smfr{5}{\sqrt{2}}$.  From (\ref{eq_dvdtheta}) we get $\fr{dv}{d\t} \le \sqrt{2V(0)}$ and
$$v(0) \le -v^* + \sqrt{2V(0)}\, \t^*.$$
For $m_3=1$ we have $\t^* = \sin^{-1}(\sqrt{2}-1)$.  Evaluating these constants numerically gives $v(0) \le -1.3 <0$ as required.

The branch $\g$ requires more work.  First, it must be shown that $\g$ remains in the $w>0$ part of the energy manifold at least until it crosses the line $\t=0$.  Thus is equivalent to $v(\t)$ being defined for $\t\in [\t^*-\pi,0]$.  Then we also require that  that $v_1=v(-\smfr{\pi}{2})<0$ and $v_2 = v(0)>0$.  Recall that solutions can only leave $\{w>0\}$ in region $R_{II}$.  If  $v(\t)$ is defined for $\t\in [\t^*-\pi,-\smfr{\pi}{2}]$ and if $v(-\smfr{\pi}{2})<0$ then in $R_{II}$, $\g$ is trapped between two branches of stable manifolds and cannot reach the boundary curve $\{w=0\}$ (see figure~\ref{fig_collisionwu}).  Thus it suffices to show that $v(-\smfr{\pi}{2})<0$ and $v(0)>0$.

To show that $v(-\smfr{\pi}{2})<0$.  Recall that the initial value is $v(\t^*-\pi) = -v^*$ where 
$v^* = \sqrt{6} \approx 2.44949$.   First we show $v(-\smfr{3\pi}{4})\le -1.6$ then we show that the change of $v(\t)$ on 
$[-\smfr{3\pi}{4},-\smfr{\pi}{2}]$ is at most $1.56$.  For the first part, use the estimate $V(\t)\le V(-\smfr{3\pi}{4}) <4$ for 
$\t\in[\t^*-\pi,-\smfr{3\pi}{4}]$.   From (\ref{eq_dvdtheta}) we have
$$
\fr{dv}{d\t} \le \fr{2|\cos \t|\sqrt{8-v^2}}{1+\sin^2\t}.
$$
If $v(-\smfr{3\pi}{4})> -1.6$ we would have
$$\int_{-v^*}^{-1.6}\fr{dv}{\sqrt{8-v^2}} < \int_{\t^*-\pi}^{-\smfr{3\pi}{4}}\fr{2|\cos\t|\,d\t}{1+\sin^2\t}.$$
Both sides can be integrated exactly and the resulting formulas evaluated numerically to arrive at the contradiction
$0.4459<0.4456$.  For $\t\in[-\smfr{3\pi}{4},-\smfr{\pi}{2}]$ we use the crude approximation
$$\fr{dv}{d\t} \le \fr{2\sqrt{2W(\t)}}{1+\sin^2\t}.$$
Symbolic differentiation of the right-side shows that it is strictly decreasing on $[-\pi,-\smfr{\pi}{2}]$ which implies that its integral can be easily bounded using upper and lower Riemann sums.  In particular, the upper sum  for $[-\smfr{3\pi}{4},-\smfr{\pi}{2}]$ with
$100$ steps gives an upper bound of $1.559<1.56$ for the change of $v$ over this interval.

Finally, we show that $v(0)>0$.   From (\ref{eq_dvdtheta}) we have
$$
\fr{dv}{d\t} \ge \fr{2|\cos \t|\sqrt{v^{*2}-v^2}}{1+\sin^2\t}.
$$
This can be used to show that  $v(-\smfr{\pi}{2}) \ge -\smfr{v*}{\sqrt{2}}$.  If this were false we would have
$$\int_{-v^*}^{-\smfr{v*}{\sqrt{2}}}\fr{dv}{\sqrt{v^{*2}-v^2}} > \int_{\pi-\t^*}^{-\smfr{\pi}{2}}\fr{2|\cos\t|\,d\t}{1+\sin^2\t}.$$
The left side is $\smfr{\pi}{4}$ and using $\sin(\t^*) = -\sin(\pi-\t^*) = \sqrt{2}-1$ one finds that the right side is also $\smfr{\pi}{4}$, a contradiction.  Therefore, $v(\smfr{\pi}{2}) \ge -\smfr{v*}{\sqrt{2}}$.

Since $v(\t)$ is increasing we have $v(\t)\ge v(-\smfr{\pi}{2}) \ge -\smfr{v*}{\sqrt{2}} = -\sqrt{3}$ for $\t\in [-\smfr{\pi}{2},0]$.  If we had $v(0)\le 0$ then we would have $v(t)^2\le 3$ for $\t\in [-\smfr{\pi}{2},0]$. 
Then we would have
$$
\fr{dv}{d\t} \ge \fr{2|\cos \t|\sqrt{v^{*2}-v(t)^2}}{1+\sin^2\t} \ge \fr{2|\cos \t|\sqrt{3}}{1+\sin^2\t}.
$$
Integration shows that $v(\t)$ increases by at least $\smfr{\pi}{2}\sqrt{3}$ for  $\t\in [-\smfr{\pi}{2},0]$ giving the estimate
$$v(0) \ge v(-\smfr{\pi}{2}) + \smfr{\pi}{2}\sqrt{3} \ge ( \smfr{\pi}{2}-1) \sqrt{3} >0$$
a contradiction to the assumption that $v(0)\le 0$.  Hence $v(0)>0$ as required.

\section{Acknowledgements.}
The authors would like to acknowledge helpful discussions and correspondences with Alan Weinstein and Alain Chenciner. 
We would also like to thank Greg Laughlin for pointing us to the story of the Pythagorean
3-body problem.  

%\vfill\break
\bibliographystyle{amsplain}

\providecommand{\bysame}{\leavevmode\hbox to3em{\hrulefill}\thinspace}
\providecommand{\MR}{\relax\ifhmode\unskip\space\fi MR }
% \MRhref is called by the amsart/book/proc definition of \MR.
\providecommand{\MRhref}[2]{%
  \href{http://www.ams.org/mathscinet-getitem?mr=#1}{#2}
}
\providecommand{\href}[2]{#2}
\begin{thebibliography}{}

\end{thebibliography}


\begin{thebibliography}{10}

\bibitem{Albouy-Chenciner} A. Albouy and A. Chenciner, \textit{Le probl\'eme des n corps et les distances mutuelles}, 
Inventiones mathematicae 131 pp. 151-184 (1998).

\bibitem{Birkhoff} G.D. Birkhoff {\sl Dynamical Systems}, American Mathematical Society Colloquium Publications, vol. 
\textbf{9} (1927).  

\bibitem{Burrau} C. Burrau \textit{Numerische Berechnung eines Spezialfalles des Dreikorperproblems},
\textit{Astronomische Nachrichten},  Band 195. Nr. 4662, 6, (1913), 114-118.
\url{http://adsabs.harvard.edu/abs/1913AN....195..113B}
.

\bibitem{Chenciner-2002} A. Chenciner,  \textit{Action minimizing solutions of the Newtonian
n-body problem : from homology to symmetry}, Proceedings of the International Congress of Mathematics, Vol. III (Beijing,
2002), 279-294, Higner Ed. Press, 2002 

\bibitem{ChencinerMontgomery} A. Chenciner and R. Montgomery, \textit{A remarkable periodic solution of the 
three-body problem in the case of equal masses}, Ann. of Math., \textbf{152} (2000) no. 3, 881--901.

\bibitem{Easton} R. Easton, \textit{Parabolic orbits for the planar three-body problem}, JDE,
\textbf{52} (1984) 116--134.

\bibitem{EastonMcGehee} R. Easton and R. McGehee, \textit{Homoclinic phenomena for orbits doubly asymptotic to an invariant three-sphere}, Indiana Univ. Math. J.,
\textbf{28} (1979) 211--240.

\bibitem{Ewing} G.M. Ewing {\sl Calculus of Variations with Applications}, Dover Publications, Inc, New York, (1985).

\bibitem{Lagrange} J-L. Lagrange,  {\sl Essai sur le Probl\'eme des Trois Corps.} Prix de
l'Acad\'emie Royale des Sciences de Paris, tome IX,   {\sl in} volume
6 of {\oe}uvres (page 292), (1772).


\bibitem{Marchal} C. Marchal,  \textit{How the minimization of action avoids
singularities},  Celestial Mech. Dynam. Astronom.
\textbf{83}, 325-354 (2002).

\bibitem{Mars} J. Marsden, \textit{Lecture on Mechanics}, London Mathematical Society Lecture Notes Series \textbf{174}, Cambridge University Press (1992).




\bibitem{Mc} R. McGehee, \textit{Triple collision in the collinear three-body problem}, Inv. Math,
\textbf{27} (1974) 191--227.

\bibitem{Mc2} R. McGehee, \textit{A stable manifold theorem for degenerate fixed points with applications to celestial mechanics}, JDE, \textbf{14} (1973) 70--88.


\bibitem{Mikkola}  K. Tanikawa and S. Mikkola, 
{\it A trial symbolic dynamics of the planar three-body}, (2008), 
arXiv:0802.2465v1. 
 


\bibitem{Moe} R. Moeckel, \textit{ Orbits near triple collision in the   three-body problem}, Indiana Univ. Math J.,,
\textbf{32,4} (1983) 221--240.


\bibitem{Mont} R. Montgomery, \textit{Infinitely many syzygies}, Arch. Rat. Mech..,
\textbf{164, 4} (2002) 311--340.

\bibitem{Mont2} R. Montgomery, \textit{The zero angular momentum three-body problem: all but one solution has syzygies}, Erg.Th.Dyn.Sys., \textbf{27,6} (2007) 1933--1946.

\bibitem{Mont3}R. Montgomery,\textit{The N-body problem, the braid group, and action-minimizing periodic orbits}, Nonlinearity, 
\textbf{11,2} (1998) 363--376.


\bibitem{Moore} C. Moore,  {\it Braids in Classical Gravity,
    Physical Review Letters 70, pp. 3675--3679, (1993)}
    
\bibitem{Robinson} C. Robinson, \textit{Homoclinic orbits and oscillation for the planar three-body problem}, JDE, 
\textbf{52} (1984) 356--377.


\bibitem{Ruiz} Otto Raul Ruiz, {\it Existence of Brake-Orbits in Finsler Mechanical Systems}  U.C. Berkeley thesis in Mathematics,      1975

\bibitem{Seifert} H. Seifert, \textit{Periodische Bewegungen Mechanischer Systeme}, Math. Z. 51 (1948), 
transl. by W. McCain at
\url{http://count.ucsc.edu/~rmont/papers/list.html}
under the year 2006

\bibitem{Simo} C. Sim\'o, \textit{Analysis of triple collision in the isosceles three-body problem}, in \textit{Classical Mechanics and Dynamical Systems}, Marcel Dekkar (1981), 203--224.

\bibitem{SimoLLibre} C. Sim\'o, J. LLibre, \textit{Charcterization of transversal homothetic solutions in the n-body  problem}, \textit{Arch. Rat. Mech.}, 77 (1981), 189--198.

\bibitem{SimoMartinez} C. Sim\'o, R.Martinez, \textit{Qualitative study of the planar isosceles three-body  problem}, \textit{
Cel. Mech.}, 41.1-4, (1987/88), 179--251.


\bibitem{SM} C.L. Siegel and J. Moser,          
\textit{Lectures on Celestial Mechanics},
Springer-Verlag,New York (1971).

\bibitem{Sz1} V. Szehebely\textit{Burrau's problem of three bodies} 
\textit{Proc. Nat. Acad. Sci.}, 58, (1967), 60-65.
\url{http://adsabs.harvard.edu/abs/1967PNAS...58...60S}


\bibitem{Sz2} V. Szehebely and F. Peters, \textit{A new periodic solution to the problem of three bodies.}
, The Astronomical Journal, v. 72, no.9, (1967),1187-1190.
\url{http://adsabs.harvard.edu/abs/1967AJ.....72.1187S}



\bibitem{Tanikawa}K Tanikawa and S.  Mikkola,
\textit{A trial symbolic dynamics of the planar three-body problem},
arXiv:0802.2465


\bibitem{Weinstein}  A. Weinstein, Normal modes for non�linear Hamiltonian systems, Inv. Math. 20 (1973), 47�57.


\end{thebibliography}

\end{document}